\ifx\documentclass\undefined
\documentstyle[12pt]{article}
\else
\documentclass[12pt]{article}
\fi

\usepackage{pst-all}
\usepackage{here}
\usepackage{amssymb}
\usepackage{graphicx}
\makeatletter
\renewcommand{\theequation}{\thesection.\arabic{equation}}
\@addtoreset{equation}{section}
\def\eqnarray{%
\stepcounter{equation}%
\let\@currentlabel=\theequation
\global\@eqnswtrue
\global\@eqcnt\z@
\tabskip\@centering
\let\\=\@eqncr
$$\halign to \displaywidth\bgroup\@eqnsel\hskip\@centering
$\displaystyle\tabskip\z@{##}$&\global\@eqcnt\@ne
\hfil$\displaystyle{{}##{}}$\hfil
&\global\@eqcnt\tw@$\displaystyle\tabskip\z@{##}$\hfil
\tabskip\@centering&\llap{##}\tabskip\z@\cr}
\makeatother
\def\bbbz{{\mathchoice {\hbox{$\sf\textstyle Z\kern-0.4em Z$}}
{\hbox{$\sf\textstyle Z\kern-0.4em Z$}}
{\hbox{$\sf\scriptstyle Z\kern-0.3em Z$}}
{\hbox{$\sf\scriptscriptstyle Z\kern-0.2em Z$}}}}
\def\bbbq{{\mathchoice {\setbox0=\hbox{$\displaystyle\rm Q$}\hbox{\raise
0.15\ht0\hbox to0pt{\kern0.4\wd0\vrule height0.8\ht0\hss}\box0}}
{\setbox0=\hbox{$\textstyle\rm Q$}\hbox{\raise
0.15\ht0\hbox to0pt{\kern0.4\wd0\vrule height0.8\ht0\hss}\box0}}
{\setbox0=\hbox{$\scriptstyle\rm Q$}\hbox{\raise
0.15\ht0\hbox to0pt{\kern0.4\wd0\vrule height0.7\ht0\hss}\box0}}
{\setbox0=\hbox{$\scriptscriptstyle\rm Q$}\hbox{\raise
0.15\ht0\hbox to0pt{\kern0.4\wd0\vrule height0.7\ht0\hss}\box0}}}}
\def\bbbc{{\mathchoice {\setbox0=\hbox{$\displaystyle \rm C$}\hbox{\raise
0.06\ht0\hbox to0pt{\kern0.4\wd0\vrule height0.9\ht0\hss}\box0}}
{\setbox0=\hbox{$\textstyle\rm C$}\hbox{\raise
0.06\ht0\hbox to0pt{\kern0.4\wd0\vrule height0.9\ht0\hss}\box0}}
{\setbox0=\hbox{$\scriptstyle\rm C$}\hbox{\raise
0.06\ht0\hbox to0pt{\kern0.4\wd0\vrule height0.8\ht0\hss}\box0}}
{\setbox0=\hbox{$\scriptscriptstyle\rm C$}\hbox{\raise
0.06\ht0\hbox to0pt{\kern0.4\wd0\vrule height0.8\ht0\hss}\box0}}}}
\makeatletter
  \renewcommand{\theequation}{%
 \thesection.\arabic{equation}}
  \@addtoreset{equation}{section}
\makeatother
\newtheorem{theorem}{Theorem}[section]

\newtheorem{corollary}[theorem]{Corollary}
\newtheorem{proposition}[theorem]{Proposition}
\newtheorem{remark}[theorem]{Remark}

\newsavebox{\toy}
\savebox{\toy}{\framebox[0.65em]{\rule{0cm}{1ex}}}
\newcommand{\QED}{\usebox{\toy}}
\def\nlni{\par\ifvmode\removelastskip\fi\vskip\baselineskip\noindent}
\newenvironment{proof}{\nlni\begingroup\it Proof.\rm}{
\endgroup\vskip\baselineskip}

\begin{document}
%%%%%%% DOUBLE SPACED %%%%%%%%
\setlength{\baselineskip}{15pt}
\title{
A bijection theorem 
for domino tiling with diagonal impurities
}
\author{
Fumihiko Nakano
\thanks{Faculty of Science, 
Department of Mathematics and Information Science,
Kochi University,
2-5-1, Akebonomachi, Kochi, 780-8520, Japan.
e-mail : 
nakano@math.kochi-u.ac.jp}
\and
Taizo Sadahiro
\thanks{
Faculty of Administration, 
Prefectural University of Kumamoto, 
Tsukide 3-1-100, Kumamoto, 862-8502, Japan.
e-mail : sadahiro@pu-kumamoto.ac.jp}
}
%\date{$B:G=*99?7F|!'(B}
\maketitle
%%%%%%% ABSTRACT %%%%%%%%%%%%%
\begin{abstract}
We consider 
the dimer problem on a non-bipartite graph
$G$, 
where there are two types of dimers 
one of which we regard impurities. 
Results of simulations using Markov chain 
seem to indicate that impurities are tend to distribute on the boundary, which we set as a conjecture. 
We first show that 
there is a bijection between the set of dimer coverings on 
$G$
and the set of spanning forests on two graphs which are made from 
$G$,  
with configuration of impurities 
satisfying a pairing condition. 
This bijection can be regarded as a extension of 
the Temperley bijection.
We consider local move consisting of two operations, 
and by using the bijection mentioned above, 
we prove local move connectedness. 
We further 
obtained some bound of the number of dimer coverings and the probability finding an impurity at given edge, 
by extending the argument in \cite{NS}. 
\end{abstract}

Mathematics Subject Classification (2000): 
60C05, 
82B20.  

Key Words : dimer model, impurity, Temperly bijection

%\tableofcontents
%%%%% INTRODUCTION %%%%% 
\section{Introduction}
Let 
$G = (V(G), E(G))$
be a graph. 
We say that a subset 
$M$
of 
$E(G)$
is a dimer covering of 
$G$ 
whenever it satifies the following condition : 
``for any 
$x \in V(G)$ 
we can uniquely find 
$e \in M$
with 
$x \in e$" ; 
in other words, 
any vertex in 
$G$
is covered by an edge 
$e \in M$
and this edge 
$e$
is unique. 
Each element 
$e \in M$
is called the dimer in 
$M$.  
If 
$G$
is bipartite, many beautiful results are known 
(e.g., \cite{K} and references therein), 
while much less seems to be known for non-bipartite case, one of the difficulty may be that there are no appropriate notion of height function \cite{T} so that we may need an alternative to study the global structure.

In this paper, 
we consider two kinds 
(we denote by 
$G^{(m,n)}$, $G^{(k)}$ 
to be introduced below) 
of graphs both of which are finite subgraphs of 
${\cal G}:= R({\bf Z}^2)$. 
$R({\bf Z}^2)$ 
is the radial graph of 
${\bf Z}^2$ 
defined by (Figure \ref{fig:1-1}(1))
\begin{eqnarray*}
{\cal V}
&:=&
V({\cal G})
=
{\bf Z}^2 \cup 
\left(
{\bf Z}^2 + (\frac 12, \frac 12)
\right)
\\
{\cal E}
&:=&
E({\cal G})
=
\left\{
({\bf x}, {\bf y}) 
\, : \, 
{\bf x}, {\bf y} \in {\cal V}, 
| {\bf x} - {\bf y} | = \frac {\sqrt{2}}{2}, 1
\right\}.
\end{eqnarray*}

\begin{figure}[H]
 \begin{center}
  \begin{pspicture}(10,3)(-5,-17)
  \psset{unit=1.}
   \psline[linewidth=.03,arrowsize=.2]{->}(-3,0)(5,0)
   \psline[linewidth=.03,arrowsize=.2]{->}(0,-3)(0,3)
   \psclip{\pspolygon[linewidth=0,linecolor=white](-2.3,-2.3)(4.3,-2.3)(4.3,2.3)(-2.3,2.3)}
   \multiput(-12,-6)(1,0){18}{
   \psline[linewidth=0.05,linestyle=dotted](0,0)(10,10)
   }
   \multiput(14,-6)(-1,0){18}{
   \psline[linewidth=0.05,linestyle=dotted](0,0)(-10,10)
   }
   \multiput(-4,-4)(0,1){8}{\psline(0,0)(12,0)}
   \multiput(-4,-4)(1,0){12}{\psline(0,0)(0,10)}
   \multiput(0,-4)(0,2){4}{
   \multiput(-4,0)(2,0){6}{
   \pscircle[fillstyle=solid,fillcolor=white](0,0){.1}
   \pscircle[fillstyle=solid,fillcolor=white](1,1){.1}
   \pscircle[fillstyle=solid,fillcolor=black](1,0){.1}
   \pscircle[fillstyle=solid,fillcolor=black](0,1){.1}
   \psdot(.5,.5)\psdot(1.5,.5)\psdot(.5,1.5)\psdot(1.5,1.5)
   }
   }
  \put(-3,2.3){$(1)$}
   \endpsclip

   \put(0,-7){
      \psline[linewidth=.03,arrowsize=.2]{->}(-3,0)(5,0)
   \psline[linewidth=.03,arrowsize=.2]{->}(0,-3)(0,3)
   \psclip{\pspolygon[linewidth=0,linecolor=white](-2.3,-2.3)(4.3,-2.3)(4.3,2.3)(-2.3,2.3)}
   \multiput(-4,-4)(0,1){8}{\psline(0,0)(12,0)}
   \multiput(-4,-4)(1,0){12}{\psline(0,0)(0,10)}
   \multiput(0,-4)(0,2){4}{
   \multiput(-4,0)(2,0){6}{
   \pscircle[fillstyle=solid,fillcolor=white](0,0){.1}
   \pscircle[fillstyle=solid,fillcolor=white](1,1){.1}
   \pscircle[fillstyle=solid,fillcolor=black](1,0){.1}
   \pscircle[fillstyle=solid,fillcolor=black](0,1){.1}
   }
   }
  \put(-3,2.3){$(2)$}
   \endpsclip
   }

   \put(0,-14){
   \psline[linewidth=.03,arrowsize=.2]{->}(-3,0)(5,0)
   \psline[linewidth=.03,arrowsize=.2]{->}(0,-3)(0,3)
   \psclip{\pspolygon[linewidth=0,linecolor=white](-2.3,-2.3)(4.3,-2.3)(4.3,2.3)(-2.3,2.3)}
   \multiput(-12,-6)(1,0){18}{
   \psline[linewidth=0.05,linestyle=dotted](0,0)(10,10)
   }
   \multiput(14,-6)(-1,0){18}{
   \psline[linewidth=0.05,linestyle=dotted](0,0)(-10,10)
   }
   \multiput(0,-4)(0,2){4}{
   \multiput(-4,0)(2,0){6}{
   \pscircle[fillstyle=solid,fillcolor=white](0,0){.1}
   \pscircle[fillstyle=solid,fillcolor=white](1,1){.1}
   \pscircle[fillstyle=solid,fillcolor=black](1,0){.1}
   \pscircle[fillstyle=solid,fillcolor=black](0,1){.1}
   \psdot(.5,.5)\psdot(1.5,.5)\psdot(.5,1.5)\psdot(1.5,1.5)
   }
   }
  \put(-3,2.3){$(3)$}
   \endpsclip
   }
  \end{pspicture}
  \caption{
(1) The graph ${\cal G}$, 
(2) Square lattice ${\cal G}_0$, 
(3) Square lattice ${\cal G}'_0$
  }{\label{fig:1-1}}
 \end{center}
\end{figure}

\noindent
The set 
${\cal V}$ 
of vertices of 
${\cal G}$
is given by 
${\cal V} = {\cal V}_1 \cup {\cal V}_2 \cup {\cal V}_3$ 
where  
\begin{eqnarray*}
{\cal V}_1
& := & 
\left\{
(x,y) \in {\bf Z}^2 
\, : \, 
x - y \in 2{\bf Z}
\right\}
\\
{\cal V}_2
& := & 
\left\{
(x,y) \in {\bf Z}^2 
\, : \, 
x - y \in 2{\bf Z}+1
\right\}
\\
{\cal V}_3
& := &
{\bf Z}+\left(\frac 12, \frac 12 \right).
\end{eqnarray*}
In
Figure \ref{fig:1-1}(1), 
white circles are vertices in 
${\cal V}_1$, 
and  
black circles (resp. black dots)
are those of 
${\cal V}_2$
(resp. ${\cal V}_3$).
${\cal V}_1 \cup {\cal V}_2 =: V({\cal G}_0)$
is the vertex set of a square (and hence bipartite) lattice
${\cal G}_0$
(Figure \ref{fig:1-1}(2)),
and 
${\cal V}_3$
is the set of points centered on each faces of 
${\cal G}_0$. 
Bipartiteness of 
${\cal G}_0$ 
means 
${\cal V}_1$, ${\cal V}_2$ 
satisfies the following condition : 
for 
$x,y \in V({\cal G}_0)$, 
$(x,y) \in E({\cal G}_0)$
implies 
$x \in {\cal V}_1$, $y \in {\cal V}_2$
or
$x \in {\cal V}_2$, $y \in {\cal V}_1$. 
Unless stated otherwise, 
we do not consider orientation 
and identify 
$e=(x,y)$ 
with 
$\overline{e} = (y,x)$. 
The edge set
${\cal E}=E({\cal G})$
of 
${\cal G}$ 
has the following decomposition. 
\begin{eqnarray*}
{\cal E} &=& {\cal E}_1 \cup {\cal E}_2
\\
{\cal E}_1
& := &
\{ (x,y) \in E({\cal G}) 
\; : \;
x \in {\cal V}_1 \cup {\cal V}_2, 
\;
y \in {\cal V}_3
\}
\\
{\cal E}_2
& := & 
\{ (x,y) \in E({\cal G})
\; : \; 
x, y \in {\cal V}_1 \cup {\cal V}_2
\}.
\end{eqnarray*}
In 
Figure \ref{fig:1-1} (1), 
dotted lines are edges in 
${\cal E}_1$, 
while solid ones are those in  
${\cal E}_2$.

${\cal G}$
can also be regarded as a graph obtained by adding a diagonal edges (those in 
${\cal E}_2$) in alternate directions to a square lattice
${\cal G}'_0$
(Figure \ref{fig:1-1} (3)),
whose vertex and edge sets are given by 
$V({\cal G}'_0) = V({\cal G})$, 
$E({\cal G}'_0)={\cal E}_1$.
Let 
$G(\subset {\cal G})$ 
be a finite subgraph of 
${\cal G}$
and 
$M$ 
be a dimer covering on 
$G$.
A dimer 
$e\in M \cap {\cal E}_1$
is also a dimer of a finite subgraph of the square lattice 
$G \cap {\cal G}'_0$. 
In this respect, 
it may be natural to call the dimers
$e \in M \cap {\cal E}_2$
impurities. 
For instance 
in  Figure \ref{fig:1-2}, 
impurities are those on vertical or horizontal edges. 
\\

%1-2
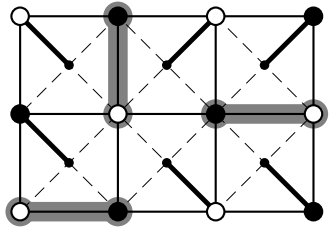
\begin{figure}[H]
 \begin{center}
  \begin{pspicture}(5,3)
  \psset{unit=1.3}
   \psline[linewidth=.2,dotsize=.3,linecolor=gray]{*-*}(0,0)(1,0)
   \psline[linewidth=.05,dotsize=.3]{-}(0,1)(.5,0.5)
   \psline[linewidth=.05,dotsize=.3]{-}(0,2)(.5,1.5)
   \psline[linewidth=.2,dotsize=.3,linecolor=gray]{*-*}(1,2)(1,1)
   \psline[linewidth=.05,dotsize=.3]{-}(2,2)(1.5,1.5)
   \psline[linewidth=.05,dotsize=.3]{-}(3,2)(2.5,1.5)
   \psline[linewidth=.2,dotsize=.3,linecolor=gray]{*-*}(3,1)(2,1)
   \psline[linewidth=.05,dotsize=.3]{-}(1.5,.5)(2,0)
   \psline[linewidth=.05,dotsize=.3]{-}(2.5,.5)(3,0)
  \psline(0,0)(3,0)(3,2)(0,2)(0,0)
  \psline(0,1)(3,1)
  \psline(1,0)(1,2)
  \psline(2,0)(2,2)
  \psline[linewidth=0.01,linestyle=dashed](0,0)(2,2)(3,1)(2,0)(0,2)
  \psline[linewidth=0.01,linestyle=dashed](3,0)(1,2)(0,1)(1,0)(3,2)
  \pscircle[fillstyle=solid,fillcolor=white](0,0){.1}
  \pscircle[fillstyle=solid,fillcolor=white](2,0){.1}
  \pscircle[fillstyle=solid,fillcolor=white](1,1){.1}
  \pscircle[fillstyle=solid,fillcolor=white](3,1){.1}
  \pscircle[fillstyle=solid,fillcolor=white](0,2){.1}
  \pscircle[fillstyle=solid,fillcolor=white](2,2){.1}
  \pscircle[fillstyle=solid,fillcolor=black](1,0){.1}
  \pscircle[fillstyle=solid,fillcolor=black](3,0){.1}
  \pscircle[fillstyle=solid,fillcolor=black](0,1){.1}
  \pscircle[fillstyle=solid,fillcolor=black](2,1){.1}
  \pscircle[fillstyle=solid,fillcolor=black](1,2){.1}
  \pscircle[fillstyle=solid,fillcolor=black](3,2){.1}
  \multiput(0,0)(0,1){2}{
  \multiput(0,0)(1,0){3}{\psdot(.5,.5)}
  }
  \end{pspicture}
  \caption{
  Example of finite subgraph of 
  ${\cal G}$ 
  and dimer covering on that. 
  Impurities are those on vertical or horizontal edges.
  }\label{fig:1-2}
 \end{center}
\end{figure}

$e \in M \cap {\cal E}_1$ 
connects vertices between 
$V(G) \cap ({\cal V}_1 \cup {\cal V}_2)$ 
and 
$V(G) \cap {\cal V}_3$ 
while 
$e \in M \cap {\cal E}_2$ 
connects those of 
$V(G) \cap ({\cal V}_1 \cup {\cal V}_2)$. 
Therefore 
the number of impurities is constant and given by 
\begin{equation}
\sharp \{ \mbox{  impurities } \}
=
\frac { |V(G)\cap{\cal V}_1| + |V(G)\cap{\cal V}_2| - |V(G)\cap{\cal V}_3|}{2}
\label{impurities}
\end{equation}
\begin{remark}
If the location of impurities is fixed, 
then this problem becomes a special case of the dimer-monomer problem where many results are known. 
(e.g.,\cite{HL})
\end{remark}
We next 
introduce the two classes of finite subgraphs of 
${\cal G}$ 
studied in this paper. 
\\

\noindent
(1)
$G^{(m,n)}$
\\
$G^{(m,n)}(\subset {\cal G})$ 
is the rectangle which has 
$m$-blocks in the horizontal direction and  
$n$-blocks in the vertical direction.
Figure \ref{fig:1-2}
shows the case of 
$(m,n) = (3,2)$.
Substituting 
$|V(G) \cap {\cal V}_1| + |V(G) \cap {\cal V}_2| = (m+1)(n+1)$,
$|V(G) \cap {\cal V}_3| = mn$ 
to the equation
(\ref{impurities}), 
we see that the number of impurities is equal to 
$(m+n+1)/2$ 
and the parity of 
$m$ 
and 
$n$
must be opposite. \\

\noindent
(2)
$G^{(k)}$\\
$G^{(k)}$
which is made by the following procedure : 
(i) composing freely the ``basic block" in Figure \ref{fig:1-3} (1), and then 
(ii) attaching 
$(2k-1)$ 
vertices of 
${\cal V}_1$ 
which we call 
``terminals" to the boundary, so that the dimer covering of 
$G^{(k)}$ 
has 
$k$
impurities. 
An example of 
$G^{(2)}$ 
is given in Figure \ref{fig:1-3}(2). \\
%

%1-3

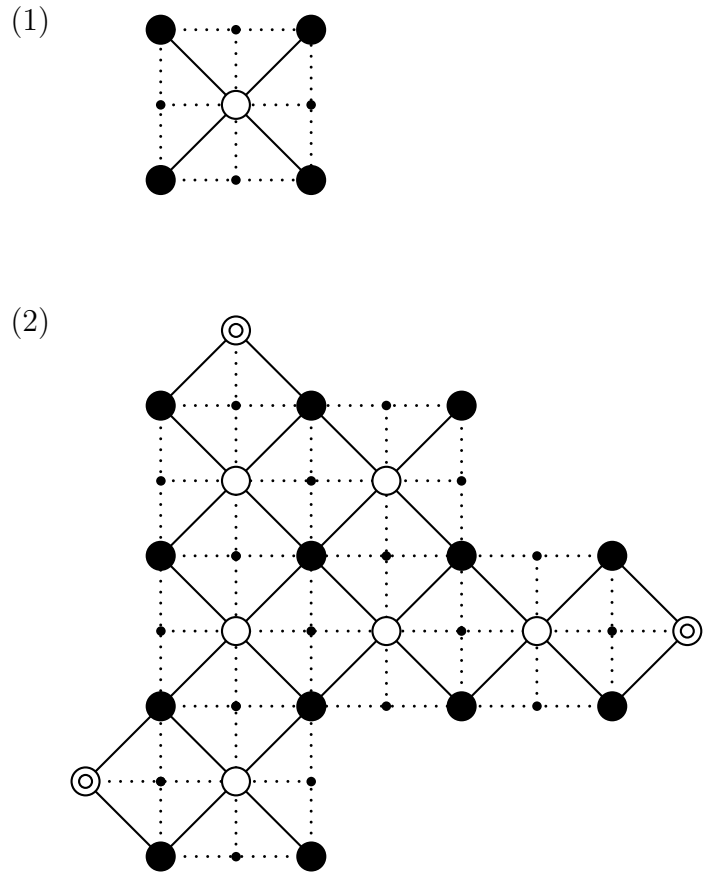
\begin{figure}[H]
 \begin{center}
  \begin{pspicture}(6,10)(0,-2.5)
   \psset{unit=2}

   \put(.5,4){
   \put(-1,1){$(1)$}
   \psline(0,0)(1,1)\psline(1,0)(0,1)
   \psline[linewidth=.02,linestyle=dotted](0,0)(1,0)(1,1)(0,1)(0,0)
   \psline[linewidth=.02,linestyle=dotted](0,0.5)(1,0.5)
   \psline[linewidth=.02,linestyle=dotted](0.5,0)(0.5,1)
   \psdot(.5,0)\psdot(1,.5)\psdot(.5,1)\psdot(0,.5)
   \pscircle[fillstyle=solid,fillcolor=black](0,0){.1}
   \pscircle[fillstyle=solid,fillcolor=black](1,0){.1}
   \pscircle[fillstyle=solid,fillcolor=black](1,1){.1}
   \pscircle[fillstyle=solid,fillcolor=black](0,1){.1}
   \pscircle[fillstyle=solid,fillcolor=white](0.5,.5){.1}
   }

   \put(-.5,3){$(2)$}
   \put(.5,-.5){
   \psline(0,0)(2,2)(3,1)(3.5,1.5)(3,2)(2,1)(0,3)(.5,3.5)(1,3)(0,2)(1,1)}
   \psline(0,0)(2.5,2.5)
   \psline(1.5,2.5)(2.5,1.5)
   \psline(.5,.5)(1.5,-.5)
   \psline(0,0)(.5,-.5)
   \put(.5,-.5){
   \multiput(0,0)(0,1){4}{\pscircle[fillstyle=solid,fillcolor=black](0,0){.1}}
   \multiput(1,0)(0,1){4}{\pscircle[fillstyle=solid,fillcolor=black](0,0){.1}}
   \multiput(2,1)(0,1){3}{\pscircle[fillstyle=solid,fillcolor=black](0,0){.1}}
   \multiput(3,1)(0,1){2}{\pscircle[fillstyle=solid,fillcolor=black](0,0){.1}}
   \psset{linestyle=dotted}
   {\psset{linewidth=.02}
   \psline(0,0)(0,3) \psline(1,0)(1,3) \psline(2,1)(2,3)
   \psline(3,1)(3,2)
   \psline(0,0)(1,0)\psline(0,1)(3,1)\psline(0,2)(3,2)\psline(0,3)(2,3)
   }
   }
   {
   \psset{linestyle=dotted}
   {
   \psset{linewidth=.02}
   \psline(0,0)(1.5,0)\psline(1,-.5)(1,3)\psline(0.5,1)(4,1)\psline(.5,2)(2.5,2)\psline(2,2.5)(2,.5)
   \psline(3,.5)(3,1.5)
   }
   \psset{linestyle=solid}
   \pscircle[fillstyle=solid,fillcolor=white](0,0){.1}
   \pscircle[fillstyle=solid,fillcolor=white](0,0){.05}
   \pscircle[fillstyle=solid,fillcolor=white](1,3){.1}
   \pscircle[fillstyle=solid,fillcolor=white](1,3){.05}
   \pscircle[fillstyle=solid,fillcolor=white](1,0){.1}
   \pscircle[fillstyle=solid,fillcolor=white](1,1){.1}
   \pscircle[fillstyle=solid,fillcolor=white](1,2){.1}
   \pscircle[fillstyle=solid,fillcolor=white](2,1){.1}
   \pscircle[fillstyle=solid,fillcolor=white](2,2){.1}
   \pscircle[fillstyle=solid,fillcolor=white](3,1){.1}
   \pscircle[fillstyle=solid,fillcolor=white](4,1){.1}
   \pscircle[fillstyle=solid,fillcolor=white](4,1){.05}
   }
   \psdot(0.5,0)\multiput(0,0)(0,1){3}{\psdot(1.,0.5)}
   \multiput(-1,0)(1,0){4}{\psdot(1.5,1)}
   \multiput(-1,1)(1,0){3}{\psdot(1.5,1)}
   \multiput(-1,-1)(1,0){2}{\psdot(1.5,1)}
   \multiput(-.5,-1.5)(0,1){3}{\psdot(1.5,1)}
   \multiput(.5,-.5)(0,1){3}{\psdot(1.5,1)}
   \multiput(1.5,-.5)(0,1){2}{\psdot(1.5,1)}
  \end{pspicture}
  \caption{(1) the basic block
(2) an example of 
$G^{(2)}$.
  \label{fig:1-3}
The double circles are the terminals.}
 \end{center}
\end{figure}

%\noindent
%
We consider two kinds of local move operation : 
square-move(s-move) and triangular-move(t-move), 
which transform a dimer covering to another one. 
(Figure \ref{fig:1-4})\\

%1-4
\begin{figure}[H]
 \begin{center}
  \begin{pspicture}(13,10)
   \psset{unit=1.5}
   \put(0,0.0){
   \put(-.5,2.5){$(2)$}
   \pspolygon[linewidth=0,fillstyle=solid,fillcolor=yellow](2,1)(3,1)(3,2)
   \put(0,1){\psline[linewidth=.06](0,0)(.5,-.5)}
   \put(0,2){\psline[linewidth=.06](0,0)(.5,-.5)}
   \put(2,0){\psline[linewidth=.06](0,0)(-.5,.5)}
   \put(2,2){\psline[linewidth=.06](0,0)(-.5,-.5)}
   \put(3,0){\psline[linewidth=.06](0,0)(-.5,.5)}
   \put(3,2){\psline[linewidth=.06](0,0)(-.5,-.5)}
   \psline[linewidth=.12,linecolor=gray]{*-*}(0,0)(1,0)
   \psline[linewidth=.12,linecolor=gray]{*-*}(1,1)(1,2)
   \psline[linewidth=.12,linecolor=gray]{*-*}(2,1)(3,1)
   \psline(0,0)(3,0)(3,2)(0,2)(0,0)
  \psline(0,1)(3,1)
  \psline(1,0)(1,2)
  \psline(2,0)(2,2)
  \psline[linewidth=0.03,linestyle=dotted](0,0)(2,2)(3,1)(2,0)(0,2)
  \psline[linewidth=0.03,linestyle=dotted](3,0)(1,2)(0,1)(1,0)(3,2)
  \pscircle[fillstyle=solid,fillcolor=white](0,0){.1}
  \pscircle[fillstyle=solid,fillcolor=white](2,0){.1}
  \pscircle[fillstyle=solid,fillcolor=white](1,1){.1}
  \pscircle[fillstyle=solid,fillcolor=white](3,1){.1}
  \pscircle[fillstyle=solid,fillcolor=white](0,2){.1}
  \pscircle[fillstyle=solid,fillcolor=white](2,2){.1}
  \pscircle[fillstyle=solid,fillcolor=black](1,0){.1}
  \pscircle[fillstyle=solid,fillcolor=black](3,0){.1}
  \pscircle[fillstyle=solid,fillcolor=black](0,1){.1}
  \pscircle[fillstyle=solid,fillcolor=black](2,1){.1}
  \pscircle[fillstyle=solid,fillcolor=black](1,2){.1}
  \pscircle[fillstyle=solid,fillcolor=black](3,2){.1}
  \multiput(0,0)(0,1){2}{
  \multiput(0,0)(1,0){3}{\psdot(.5,.5)}
  }
   \psline[linewidth=0.1]{->}(4,1)(5,1)
   \put(3.5,1.2){triangle-move}
   }

   \put(6.0,0){
   \pspolygon[linewidth=0,fillstyle=solid,fillcolor=yellow](2,1)(3,1)(3,2)
   \put(0,1){\psline[linewidth=.06](0,0)(.5,-.5)}
   \put(0,2){\psline[linewidth=.06](0,0)(.5,-.5)}
   \put(2,0){\psline[linewidth=.06](0,0)(-.5,.5)}
   \put(2,2){\psline[linewidth=.06](0,0)(-.5,-.5)}
   \put(3,0){\psline[linewidth=.06](0,0)(-.5,.5)}
   \put(2.5,1.5){\psline[linewidth=.06](0,0)(-.5,-.5)}
   \psline[linewidth=.12,linecolor=gray]{*-*}(0,0)(1,0)
   \psline[linewidth=.12,linecolor=gray]{*-*}(1,1)(1,2)
   \psline[linewidth=.12,linecolor=gray]{*-*}(3,2)(3,1)
   \psline(0,0)(3,0)(3,2)(0,2)(0,0)
  \psline(0,1)(3,1)
  \psline(1,0)(1,2)
  \psline(2,0)(2,2)
  \psline[linewidth=0.03,linestyle=dotted](0,0)(2,2)(3,1)(2,0)(0,2)
  \psline[linewidth=0.03,linestyle=dotted](3,0)(1,2)(0,1)(1,0)(3,2)
  \pscircle[fillstyle=solid,fillcolor=white](0,0){.1}
  \pscircle[fillstyle=solid,fillcolor=white](2,0){.1}
  \pscircle[fillstyle=solid,fillcolor=white](1,1){.1}
  \pscircle[fillstyle=solid,fillcolor=white](3,1){.1}
  \pscircle[fillstyle=solid,fillcolor=white](0,2){.1}
  \pscircle[fillstyle=solid,fillcolor=white](2,2){.1}
  \pscircle[fillstyle=solid,fillcolor=black](1,0){.1}
  \pscircle[fillstyle=solid,fillcolor=black](3,0){.1}
  \pscircle[fillstyle=solid,fillcolor=black](0,1){.1}
  \pscircle[fillstyle=solid,fillcolor=black](2,1){.1}
  \pscircle[fillstyle=solid,fillcolor=black](1,2){.1}
  \pscircle[fillstyle=solid,fillcolor=black](3,2){.1}
  \multiput(0,0)(0,1){2}{
  \multiput(0,0)(1,0){3}{\psdot(.5,.5)}
  }
   }

   \put(0,4.0){
   \put(-.5,2.5){$(1)$}
   \pspolygon[linewidth=0,fillstyle=solid,fillcolor=cyan](2,1)(2.5,1.5)(2,2)(1.5,1.5)
   \put(0,1){\psline[linewidth=.06](0,0)(.5,-.5)}
   \put(0,2){\psline[linewidth=.06](0,0)(.5,-.5)}
   \put(2,0){\psline[linewidth=.06](0,0)(-.5,.5)}
   \put(2,2){\psline[linewidth=.06](0,0)(-.5,-.5)}
   \put(3,0){\psline[linewidth=.06](0,0)(-.5,.5)}
   \put(2,1){\psline[linewidth=.06](0,0)(.5,.5)}
   \psline[linewidth=.12,linecolor=gray]{*-*}(0,0)(1,0)
   \psline[linewidth=.12,linecolor=gray]{*-*}(1,1)(1,2)
   \psline[linewidth=.12,linecolor=gray]{*-*}(3,2)(3,1)
   \psline(0,0)(3,0)(3,2)(0,2)(0,0)
  \psline(0,1)(3,1)
  \psline(1,0)(1,2)
  \psline(2,0)(2,2)
  \psline[linewidth=0.03,linestyle=dotted](0,0)(2,2)(3,1)(2,0)(0,2)
  \psline[linewidth=0.03,linestyle=dotted](3,0)(1,2)(0,1)(1,0)(3,2)
  \pscircle[fillstyle=solid,fillcolor=white](0,0){.1}
  \pscircle[fillstyle=solid,fillcolor=white](2,0){.1}
  \pscircle[fillstyle=solid,fillcolor=white](1,1){.1}
  \pscircle[fillstyle=solid,fillcolor=white](3,1){.1}
  \pscircle[fillstyle=solid,fillcolor=white](0,2){.1}
  \pscircle[fillstyle=solid,fillcolor=white](2,2){.1}
  \pscircle[fillstyle=solid,fillcolor=black](1,0){.1}
  \pscircle[fillstyle=solid,fillcolor=black](3,0){.1}
  \pscircle[fillstyle=solid,fillcolor=black](0,1){.1}
  \pscircle[fillstyle=solid,fillcolor=black](2,1){.1}
  \pscircle[fillstyle=solid,fillcolor=black](1,2){.1}
  \pscircle[fillstyle=solid,fillcolor=black](3,2){.1}
  \multiput(0,0)(0,1){2}{
  \multiput(0,0)(1,0){3}{\psdot(.5,.5)}
  }
   \psline[linewidth=0.1]{->}(4,1)(5,1)
   \put(3.5,1.2){square-move}
   }

   \put(6,4){
   \pspolygon[linewidth=0,fillstyle=solid,fillcolor=cyan](2,1)(2.5,1.5)(2,2)(1.5,1.5)
   \put(0,1){\psline[linewidth=.06](0,0)(.5,-.5)}
   \put(0,2){\psline[linewidth=.06](0,0)(.5,-.5)}
   \put(2,0){\psline[linewidth=.06](0,0)(-.5,.5)}
   \put(2,2){\psline[linewidth=.06](0,0)(.5,-.5)}
   \put(3,0){\psline[linewidth=.06](0,0)(-.5,.5)}
   \put(2,1){\psline[linewidth=.06](0,0)(-.5,.5)}
   \psline[linewidth=.12,linecolor=gray]{*-*}(0,0)(1,0)
   \psline[linewidth=.12,linecolor=gray]{*-*}(1,1)(1,2)
   \psline[linewidth=.12,linecolor=gray]{*-*}(3,2)(3,1)
   \psline(0,0)(3,0)(3,2)(0,2)(0,0)
  \psline(0,1)(3,1)
  \psline(1,0)(1,2)
  \psline(2,0)(2,2)
  \psline[linewidth=0.03,linestyle=dotted](0,0)(2,2)(3,1)(2,0)(0,2)
  \psline[linewidth=0.03,linestyle=dotted](3,0)(1,2)(0,1)(1,0)(3,2)
  \pscircle[fillstyle=solid,fillcolor=white](0,0){.1}
  \pscircle[fillstyle=solid,fillcolor=white](2,0){.1}
  \pscircle[fillstyle=solid,fillcolor=white](1,1){.1}
  \pscircle[fillstyle=solid,fillcolor=white](3,1){.1}
  \pscircle[fillstyle=solid,fillcolor=white](0,2){.1}
  \pscircle[fillstyle=solid,fillcolor=white](2,2){.1}
  \pscircle[fillstyle=solid,fillcolor=black](1,0){.1}
  \pscircle[fillstyle=solid,fillcolor=black](3,0){.1}
  \pscircle[fillstyle=solid,fillcolor=black](0,1){.1}
  \pscircle[fillstyle=solid,fillcolor=black](2,1){.1}
  \pscircle[fillstyle=solid,fillcolor=black](1,2){.1}
  \pscircle[fillstyle=solid,fillcolor=black](3,2){.1}
  \multiput(0,0)(0,1){2}{
  \multiput(0,0)(1,0){3}{\psdot(.5,.5)}
  }
   }

  \end{pspicture}
  \caption{(1) square move(s-move)~~
  (2) triangular move(t-move)}\label{fig:1-4}
 \end{center}
\end{figure}

\noindent
In Section 3, 
we show the local move connectedness(LMC in short), 
i.e., 
any two dimer coverings on 
$G^{(m,n)}$, $G^{(k)}$ 
can be transformed each other by a successive application of local moves. 

By LMC, 
we can simulate an ergodic Markov chain whose state space is 
\[
{\cal D}(G) := \{ \mbox{ dimer covering of $G$ } \}.
\]
Under a suitable choice of transition probabilities, 
its stationary distribution is uniform so that we can obtain (approximately)uniform sample. 
When we carry out the simulation, 
we see that the impurities are always pushed out to the boundary of 
$G$, 
whichever the initial state is(Figure $\ref{fig:1-5}$), 
\\
%

%1-5
\begin{figure}[H]
\begin{center}
  \includegraphics[width=10cm]{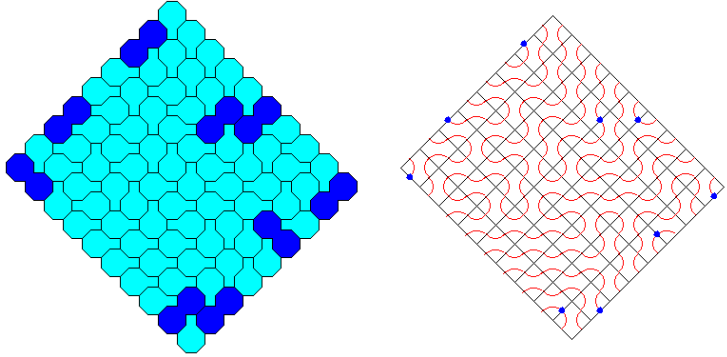}
 \caption{a result of simulation. Impurities are tend to distribute on
 boundaries.}
 \label{fig:1-5}
\end{center}
\end{figure}

\noindent
Thus we are led to the following conjecture. \\
\noindent
``The number of dimer coverings 
with given configuration of impurities 
is maximized if all impurities are on the boundary of 
$G$"
$\quad\cdots$
$(*)$\\
\noindent
which is our motivation to consider this problem
\footnote{
Let 
$D$ 
be the number of dimer coverings and 
$C$ 
be the number of dimer coverings with an impurity in the center of the graph 
$G$. 
For 
$G=G^{(3,2)}$, 
$D=160$, $C=8$, 
and for 
$G=G^{(4,3)}$, 
$D=12400$, $C=400$. 
}. 
In our previous work \cite{NS}, 
we studied the case of 1-impurity and obtained the formula to compute 
$| {\cal D}(G^{(1)}) |$ 
and the probability of finding the impurity at given edge, which, if 
$G^{(1)}$ 
is the 1-dimentional chain, 
decreases exponentially as the edge being far from the end of 
$G^{(1)}$, 
and thus solving this conjecture for 
$G^{(1)}$. 
In Section 4, 
we partially extend the analysis given in \cite{NS} to 
$G^{(k)}$. 

The rest of this paper is organized as follows.
In Section 2, 
we consider graphs 
$G_1$, $G_2$
with 
$V(G_1)=V(G) \cap {\cal V}_1$, $V(G_2)=V(G) \cap {\cal V}_2$
and show that there is a bijection between 
${\cal D}(G)$
and the set of spanning forests of 
$G_1$
(or $G_2$) 
and the configuration of impurities satisfying certain condition
(Theorem \ref{pairing}, \ref{k impurity case}).
This bijection 
can be regarded as a generalization of the 
Temperley bijection 
\cite{KPW}
and gives us the global structure of the configuration of dimers.
In Section 3, 
we prove LMC by using Theorem \ref{pairing}, \ref{k impurity case}.
LMC 
has been proved by \cite{OS} for normal subgraphs of 
${\cal G}$. 
Our proof 
works only for 
$G^{(m,n)}$ 
and 
$G^{(k)}$ 
but is elementary and straightforward. 
It is also possible 
to obtain a bound on the number of steps needed to transform two given dimer coverings each other, 
which depends polynomially on the volume of 
$G$. 
Thus 
it would be interesting to study the mixing property of the Markov chain discussed above. 
In Section 4, 
we extend the result in \cite{NS} to the case of 
$G^{(k)}$
($k \ge 2$).
We are not able to derive formulas for 
$| {\cal D}(G^{(k)}) |$ 
and the probability of finding the impurity, but derived bounds of them.
In Appendix, 
we compute 
$| {\cal D}(G) |$ 
for some examples by using Theorem \ref{pairing}, \ref{k impurity case}. 
%

%%%%%
\section{Transform to the spanning forest}
In this section 
we construct a mapping from the set of dimer covering of 
$G$
to the set of spanning forests of two graphs made from 
$G$. 
We do this for 
$G = G^{(m,n)}$
in the subsection 2.1, 
and for 
$G = G^{(k)}$
in the subsection 2.2. 
We omit 
the superscript and write 
$G$ 
instead of 
$G^{(m,n)}$ 
(resp. instead of $G^{(k)}$) 
in subsection 2.1
(resp. in subsection 2.2). 
We first set 
\begin{eqnarray*}
&&V_1 := V(G) \cap {\cal V}_1, 
\quad
V_2 := V(G) \cap {\cal V}_2, 
\\
&&E_1 := E(G) \cap {\cal E}_1, 
\quad
E_2 := E(G) \cap {\cal E}_2, 
\quad
G = G^{(m,n)}, G^{(k)}.
\end{eqnarray*}
\subsection{Construction of a bijection : for 
$G^{(m,n)}$}
Let 
$G_1$, $G_2$
be graphs such that  
$V(G_j)=V_j$
($j=1,2$)
and for 
$x, y \in V_j$, 
we set 
$(x,y) \in E(G_j)$
iff we find 
$z \in V_3$,  
which we call the middle vertex, 
with 
$(x,z), (z,y) \in E_1$. 
An explicit description is given in 
Figure \ref{fig:2-1-1} (2), (3).\\
%

%2-1-1
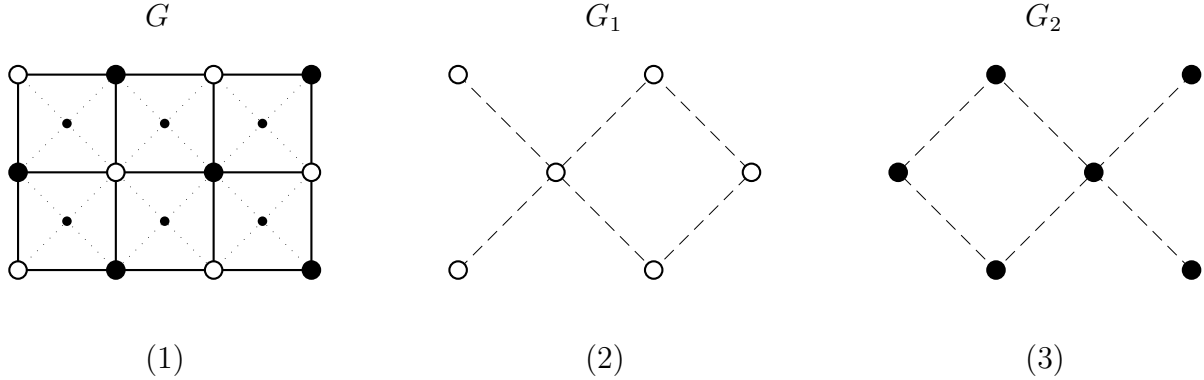
\begin{figure}[H]
\begin{center}
 \begin{pspicture}(15,4)(0,-1.5)
  \psset{unit=1.3}
  \psline(0,0)(3,0)(3,2)(0,2)(0,0)
  \psline(0,1)(3,1)
  \psline(1,0)(1,2)
  \psline(2,0)(2,2)
  \psline[linewidth=0.01,linestyle=dotted](0,0)(2,2)(3,1)(2,0)(0,2)
  \psline[linewidth=0.01,linestyle=dotted](3,0)(1,2)(0,1)(1,0)(3,2)
  \pscircle[fillstyle=solid,fillcolor=white](0,0){.1}
  \pscircle[fillstyle=solid,fillcolor=white](2,0){.1}
  \pscircle[fillstyle=solid,fillcolor=white](1,1){.1}
  \pscircle[fillstyle=solid,fillcolor=white](3,1){.1}
  \pscircle[fillstyle=solid,fillcolor=white](0,2){.1}
  \pscircle[fillstyle=solid,fillcolor=white](2,2){.1}
  \pscircle[fillstyle=solid,fillcolor=black](1,0){.1}
  \pscircle[fillstyle=solid,fillcolor=black](3,0){.1}
  \pscircle[fillstyle=solid,fillcolor=black](0,1){.1}
  \pscircle[fillstyle=solid,fillcolor=black](2,1){.1}
  \pscircle[fillstyle=solid,fillcolor=black](1,2){.1}
  \pscircle[fillstyle=solid,fillcolor=black](3,2){.1}
  \multiput(0,0)(0,1){2}{
  \multiput(0,0)(1,0){3}{\psdot(.5,.5)}
  }
  \put(1.3,2.5){$G$}
  \put(1.3,-1){$(1)$}

  \put(4.5,0){
  \psline[linewidth=0.01,linestyle=dashed](0,0)(2,2)(3,1)(2,0)(0,2)
  \pscircle[fillstyle=solid,fillcolor=white](0,0){.1}
  \pscircle[fillstyle=solid,fillcolor=white](2,0){.1}
  \pscircle[fillstyle=solid,fillcolor=white](1,1){.1}
  \pscircle[fillstyle=solid,fillcolor=white](3,1){.1}
  \pscircle[fillstyle=solid,fillcolor=white](0,2){.1}
  \pscircle[fillstyle=solid,fillcolor=white](2,2){.1}
  \put(1.3,-1){$(2)$}
  \put(1.3,2.5){$G_1$}
  }

  \put(9,0){
  \psline[linewidth=0.01,linestyle=dashed](3,0)(1,2)(0,1)(1,0)(3,2)
  \pscircle[fillstyle=solid,fillcolor=black](1,0){.1}
  \pscircle[fillstyle=solid,fillcolor=black](3,0){.1}
  \pscircle[fillstyle=solid,fillcolor=black](0,1){.1}
  \pscircle[fillstyle=solid,fillcolor=black](2,1){.1}
  \pscircle[fillstyle=solid,fillcolor=black](1,2){.1}
  \pscircle[fillstyle=solid,fillcolor=black](3,2){.1}
  \put(1.3,-1){$(3)$}
  \put(1.3,2.5){$G_2$}
  }
 \end{pspicture}
  \caption{(1) graph 
$G^{(3,2)}$ 
(2), (3) 
$G_1$, $G_2$ 
corresponding to 
$G^{(3,2)}$.}\label{fig:2-1-1}
\end{center}
\end{figure}

Let 
$k := \frac {m+n+1}{2}$ 
be the number of impurities. 
\begin{theorem}
\label{pairing}
We have a bijection between the following two sets. 
\begin{eqnarray*}
{\cal D}(G) & := & 
\{ \mbox{ dimer covering on 
$G$ } \}
\\
{\cal F}(G, P) & := & 
\{ (F_1, F_2, \{ e_j \}_{j=1}^k)
\, | \, 
F_j  : 
\mbox{spanning forests on 
$G_j (j=1,2)$ }
\\
&&\qquad\qquad\qquad\qquad
\mbox{
with 
$k$-components,  }
\\
&&\qquad\qquad\qquad\qquad
\{ e_j \}_{j=1}^k \subset E_2 : 
\mbox{configuration of impurities,}
\\
&&\qquad\qquad\qquad\qquad
\mbox{  
with condition (P) } \}
\end{eqnarray*}
(P) : 
(1) 
$F_1$, $F_2$ 
have no intersections, 
\\
(2)
each subtrees of 
$F_1$
are paired with those of
$F_2$ 
by impurities. 
\end{theorem}
Under condition (P)(1), the spanning forest of 
$G_1$
uniquely determines that of 
$G_2$ 
so that we have a redundancy in the statement. 
Figure \ref{fig:2-1-2} (1)
shows an example of spanning forests and impurities with condition 
(P)
for 
$G^{(3,2)}$.
In this figure, 
spanning trees of both 
$G_1$ 
and 
$G_2$ 
are composed of three trees for each and are paired by impurities. 
Figure \ref{fig:2-1-2} (2)
is the corresponding dimer covering of 
$G$. \\
%

%2-1-2
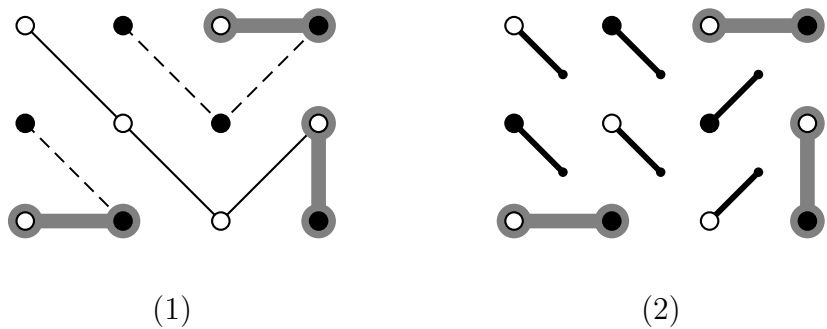
\begin{figure}[H]
 \begin{center}
  \begin{pspicture}(10,4)(0,-1.5)
   \psset{unit=1.3}
   \psline[linewidth=0.02,linestyle=dashed](1,2)(2,1)(3,2)
   \psline[linewidth=0.02,linestyle=dashed](0,1)(1,0)
   \psline(0,2)(2,0)(3,1)
   \psline[linewidth=.15,linecolor=gray]{*-*}(0,0)(1,0)
   \psline[linewidth=.15,linecolor=gray]{*-*}(3,0)(3,1)
   \psline[linewidth=.15,linecolor=gray]{*-*}(3,2)(2,2)
  \pscircle[fillstyle=solid,fillcolor=white](0,0){.1}
  \pscircle[fillstyle=solid,fillcolor=white](2,0){.1}
  \pscircle[fillstyle=solid,fillcolor=white](1,1){.1}
  \pscircle[fillstyle=solid,fillcolor=white](3,1){.1}
  \pscircle[fillstyle=solid,fillcolor=white](0,2){.1}
  \pscircle[fillstyle=solid,fillcolor=white](2,2){.1}
  \pscircle[fillstyle=solid,fillcolor=black](1,0){.1}
  \pscircle[fillstyle=solid,fillcolor=black](3,0){.1}
  \pscircle[fillstyle=solid,fillcolor=black](0,1){.1}
  \pscircle[fillstyle=solid,fillcolor=black](2,1){.1}
  \pscircle[fillstyle=solid,fillcolor=black](1,2){.1}
  \pscircle[fillstyle=solid,fillcolor=black](3,2){.1}
  \put(1.3,-1){$(1)$}

   \put(5,0){
   \psline[linewidth=.06](0,1)(.5,.5)\psdot(.5,.5)
   \put(1,0){\psline[linewidth=.06](0,1)(.5,.5)\psdot(.5,.5)}
   \put(0,1){\psline[linewidth=.06](0,1)(.5,.5)\psdot(.5,.5)}
   \put(1,1){\psline[linewidth=.06](0,1)(.5,.5)\psdot(.5,.5)}
   \put(2,0){\psline[linewidth=.06](0,0)(.5,.5)\psdot(.5,.5)}
   \put(2,1){\psline[linewidth=.06](0,0)(.5,.5)\psdot(.5,.5)}
   \psline[linewidth=.15,linecolor=gray]{*-*}(0,0)(1,0)
   \psline[linewidth=.15,linecolor=gray]{*-*}(3,0)(3,1)
   \psline[linewidth=.15,linecolor=gray]{*-*}(3,2)(2,2)
   \pscircle[fillstyle=solid,fillcolor=white](0,0){.1}
   \pscircle[fillstyle=solid,fillcolor=white](2,0){.1}
   \pscircle[fillstyle=solid,fillcolor=white](1,1){.1}
   \pscircle[fillstyle=solid,fillcolor=white](3,1){.1}
   \pscircle[fillstyle=solid,fillcolor=white](0,2){.1}
   \pscircle[fillstyle=solid,fillcolor=white](2,2){.1}
   \pscircle[fillstyle=solid,fillcolor=black](1,0){.1}
   \pscircle[fillstyle=solid,fillcolor=black](3,0){.1}
   \pscircle[fillstyle=solid,fillcolor=black](0,1){.1}
   \pscircle[fillstyle=solid,fillcolor=black](2,1){.1}
   \pscircle[fillstyle=solid,fillcolor=black](1,2){.1}
   \pscircle[fillstyle=solid,fillcolor=black](3,2){.1}
   \put(1.3,-1){$(2)$}
   }
  \end{pspicture}
  \caption{(1) an example of spanning forests and impurities with condition 
(P). 
Solid lines are forest of 
$G_1$, 
and dotted lines are that of 
$G_2$. 
(2) corresponding dimer covering of 
$G$}\label{fig:2-1-2}
 \end{center}
\end{figure}

The proof of Theorem \ref{pairing}
requires some preparations.
First of all, 
given dimer covering of
$G$, 
we draw curves on 
$G$,  
which we call the slit curve,  
as is described in 
Figure \ref{fig:2-1-3}. \\
%

%2-1-3
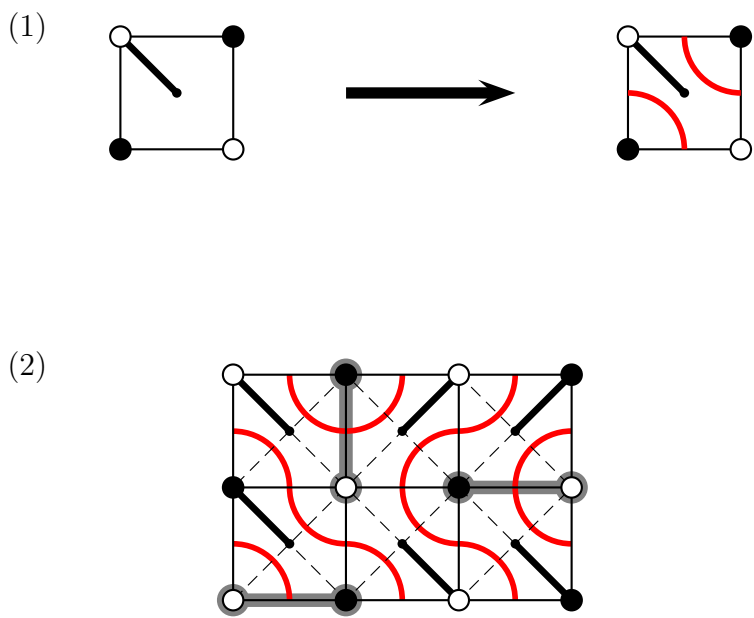
\begin{figure}[H]
 \begin{center}
  \begin{pspicture}(8,8)
   \psset{unit=1.5}
   \put(0,4){
   \psline(0,0)(1,0)(1,1)(0,1)(0,0)
   \psline[linewidth=.06](0,1)(.5,.5)
   \psdot(.5,.5)
   \pscircle[fillstyle=solid,fillcolor=black](0,0){.1}
   \pscircle[fillstyle=solid,fillcolor=black](1,1){.1}
   \pscircle[fillstyle=solid,fillcolor=white](1,0){.1}
   \pscircle[fillstyle=solid,fillcolor=white](0,1){.1}

   \put(4.5,0){
   \psline(0,0)(1,0)(1,1)(0,1)(0,0)   
   \psline[linewidth=.06](0,1)(.5,.5)
   \psdot(.5,.5)
   \pscircle[fillstyle=solid,fillcolor=black](0,0){.1}
   \pscircle[fillstyle=solid,fillcolor=black](1,1){.1}
   \pscircle[fillstyle=solid,fillcolor=white](1,0){.1}
   \pscircle[fillstyle=solid,fillcolor=white](0,1){.1}
   \psarc[linewidth=.05,linecolor=red](0,0){.5}{0}{90}
   \psarc[linewidth=.05,linecolor=red](1,1){.5}{180}{270}
   }
   \psline[linewidth=.1]{->}(2.,.5)(3.5,.5)
   }
   \put(-1,5){$(1)$}

   \put(1,.0){
   \put(0,1){\psline[linewidth=.06](0,0)(.5,-.5)}
   \put(0,2){\psline[linewidth=.06](0,0)(.5,-.5)}
   \put(2,0){\psline[linewidth=.06](0,0)(-.5,.5)}
   \put(2,2){\psline[linewidth=.06](0,0)(-.5,-.5)}
   \put(3,0){\psline[linewidth=.06](0,0)(-.5,.5)}
   \put(3,2){\psline[linewidth=.06](0,0)(-.5,-.5)}
   \psline[linewidth=.12,linecolor=gray]{*-*}(0,0)(1,0)
   \psline[linewidth=.12,linecolor=gray]{*-*}(1,1)(1,2)
   \psline[linewidth=.12,linecolor=gray]{*-*}(2,1)(3,1)
   \psarc[linecolor=red,linewidth=.05](0,0){.5}{0}{90}
   \multiput(0,1)(1,-1){2}{\psarc[linecolor=red,linewidth=.05](0,0){.5}{0}{90}}
   \psarc[linecolor=red,linewidth=.05](1,1){.5}{180}{270}
   \psarc[linecolor=red,linewidth=.05](1,2){.5}{180}{360}
   \psarc[linecolor=red,linewidth=.05](2,2){.5}{270}{360}
   \psarc[linecolor=red,linewidth=.05](2,0){.5}{0}{90}
   \psarc[linecolor=red,linewidth=.05](2,1){.5}{90}{270}
   \psarc[linecolor=red,linewidth=.05](3,1){.5}{90}{270}
   \psline(0,0)(3,0)(3,2)(0,2)(0,0)
  \psline(0,1)(3,1)
  \psline(1,0)(1,2)
  \psline(2,0)(2,2)
  \psline[linewidth=0.01,linestyle=dashed](0,0)(2,2)(3,1)(2,0)(0,2)
  \psline[linewidth=0.01,linestyle=dashed](3,0)(1,2)(0,1)(1,0)(3,2)
  \pscircle[fillstyle=solid,fillcolor=white](0,0){.1}
  \pscircle[fillstyle=solid,fillcolor=white](2,0){.1}
  \pscircle[fillstyle=solid,fillcolor=white](1,1){.1}
  \pscircle[fillstyle=solid,fillcolor=white](3,1){.1}
  \pscircle[fillstyle=solid,fillcolor=white](0,2){.1}
  \pscircle[fillstyle=solid,fillcolor=white](2,2){.1}
  \pscircle[fillstyle=solid,fillcolor=black](1,0){.1}
  \pscircle[fillstyle=solid,fillcolor=black](3,0){.1}
  \pscircle[fillstyle=solid,fillcolor=black](0,1){.1}
  \pscircle[fillstyle=solid,fillcolor=black](2,1){.1}
  \pscircle[fillstyle=solid,fillcolor=black](1,2){.1}
  \pscircle[fillstyle=solid,fillcolor=black](3,2){.1}
  \multiput(0,0)(0,1){2}{
  \multiput(0,0)(1,0){3}{\psdot(.5,.5)}
  }
   }
   \put(-1,2){$(2)$}

  \end{pspicture}
  \caption{(1) drawing slit curves on a block in 
$G$
with given dimer covering, 
(2) example of slit curves corresponding to the dimer covering in 
Figure \ref{fig:1-2}}\label{fig:2-1-3}

 \end{center}
\end{figure}

\noindent
We study
some properties of these curves. 
Some inspection 
leads us to the following observation. 
\begin{proposition}
(1)
By the triangular-move, 
impurities move along the slit curves, but the slit curves  remain unchanged. \\
(2)
By the square-move, 
slit curves may switch each other, 
but the impurities remain unchanged. 
\end{proposition}
Figure \ref{fig:2-1-4} 
explains what they mean by presenting an explicit example. \\
%

%2-1-4
\begin{figure}[H]
 \begin{center}
  \begin{pspicture}(5,15)
   \psset{unit=1.5}
   \put(0,8.0){
   \pspolygon[linewidth=0,fillstyle=solid,fillcolor=yellow](2,1)(3,1)(3,2)
   \put(0,1){\psline[linewidth=.06](0,0)(.5,-.5)}
   \put(0,2){\psline[linewidth=.06](0,0)(.5,-.5)}
   \put(2,0){\psline[linewidth=.06](0,0)(-.5,.5)}
   \put(2,2){\psline[linewidth=.06](0,0)(-.5,-.5)}
   \put(3,0){\psline[linewidth=.06](0,0)(-.5,.5)}
   \put(3,2){\psline[linewidth=.06](0,0)(-.5,-.5)}
   \psline[linewidth=.12,linecolor=gray]{*-*}(0,0)(1,0)
   \psline[linewidth=.12,linecolor=gray]{*-*}(1,1)(1,2)
   \psline[linewidth=.12,linecolor=gray]{*-*}(2,1)(3,1)
   \psarc[linecolor=red,linewidth=.05](0,0){.5}{0}{90}
   \multiput(0,1)(1,-1){2}{\psarc[linecolor=red,linewidth=.05](0,0){.5}{0}{90}}
   \psarc[linecolor=red,linewidth=.05](1,1){.5}{180}{270}
   \psarc[linecolor=red,linewidth=.05](1,2){.5}{180}{360}
   \psarc[linecolor=red,linewidth=.05](2,2){.5}{270}{360}
   \psarc[linecolor=red,linewidth=.05](2,0){.5}{0}{90}
   \psarc[linecolor=red,linewidth=.05](2,1){.5}{90}{270}
   \psarc[linecolor=red,linewidth=.05](3,1){.5}{90}{270}
   \psline(0,0)(3,0)(3,2)(0,2)(0,0)
  \psline(0,1)(3,1)
  \psline(1,0)(1,2)
  \psline(2,0)(2,2)
  \psline[linewidth=0.01,linestyle=dashed](0,0)(2,2)(3,1)(2,0)(0,2)
  \psline[linewidth=0.01,linestyle=dashed](3,0)(1,2)(0,1)(1,0)(3,2)
  \pscircle[fillstyle=solid,fillcolor=white](0,0){.1}
  \pscircle[fillstyle=solid,fillcolor=white](2,0){.1}
  \pscircle[fillstyle=solid,fillcolor=white](1,1){.1}
  \pscircle[fillstyle=solid,fillcolor=white](3,1){.1}
  \pscircle[fillstyle=solid,fillcolor=white](0,2){.1}
  \pscircle[fillstyle=solid,fillcolor=white](2,2){.1}
  \pscircle[fillstyle=solid,fillcolor=black](1,0){.1}
  \pscircle[fillstyle=solid,fillcolor=black](3,0){.1}
  \pscircle[fillstyle=solid,fillcolor=black](0,1){.1}
  \pscircle[fillstyle=solid,fillcolor=black](2,1){.1}
  \pscircle[fillstyle=solid,fillcolor=black](1,2){.1}
  \pscircle[fillstyle=solid,fillcolor=black](3,2){.1}
  \multiput(0,0)(0,1){2}{
  \multiput(0,0)(1,0){3}{\psdot(.5,.5)}
  }
   \psline[linewidth=0.1]{->}(1.5,-.5)(1.5,-1.5)
   \put(1.7,-1){tri-move}
   }

   \put(0,4.0){
   \pspolygon[linewidth=0,fillstyle=solid,fillcolor=cyan](2,1)(2.5,1.5)(2,2)(1.5,1.5)
   \put(0,1){\psline[linewidth=.06](0,0)(.5,-.5)}
   \put(0,2){\psline[linewidth=.06](0,0)(.5,-.5)}
   \put(2,0){\psline[linewidth=.06](0,0)(-.5,.5)}
   \put(2,2){\psline[linewidth=.06](0,0)(-.5,-.5)}
   \put(3,0){\psline[linewidth=.06](0,0)(-.5,.5)}
   \put(2,1){\psline[linewidth=.06](0,0)(.5,.5)}
   \psline[linewidth=.12,linecolor=gray]{*-*}(0,0)(1,0)
   \psline[linewidth=.12,linecolor=gray]{*-*}(1,1)(1,2)
   \psline[linewidth=.12,linecolor=gray]{*-*}(3,2)(3,1)
   \psarc[linecolor=red,linewidth=.05](0,0){.5}{0}{90}
   \multiput(0,1)(1,-1){2}{\psarc[linecolor=red,linewidth=.05](0,0){.5}{0}{90}}
   \psarc[linecolor=red,linewidth=.05](1,1){.5}{180}{270}
   \psarc[linecolor=red,linewidth=.05](1,2){.5}{180}{360}
   \psarc[linecolor=red,linewidth=.05](2,2){.5}{270}{360}
   \psarc[linecolor=red,linewidth=.05](2,0){.5}{0}{90}
   \psarc[linecolor=red,linewidth=.05](2,1){.5}{90}{270}
   \psarc[linecolor=red,linewidth=.05](3,1){.5}{90}{270}
   \psline(0,0)(3,0)(3,2)(0,2)(0,0)
  \psline(0,1)(3,1)
  \psline(1,0)(1,2)
  \psline(2,0)(2,2)
  \psline[linewidth=0.01,linestyle=dashed](0,0)(2,2)(3,1)(2,0)(0,2)
  \psline[linewidth=0.01,linestyle=dashed](3,0)(1,2)(0,1)(1,0)(3,2)
  \pscircle[fillstyle=solid,fillcolor=white](0,0){.1}
  \pscircle[fillstyle=solid,fillcolor=white](2,0){.1}
  \pscircle[fillstyle=solid,fillcolor=white](1,1){.1}
  \pscircle[fillstyle=solid,fillcolor=white](3,1){.1}
  \pscircle[fillstyle=solid,fillcolor=white](0,2){.1}
  \pscircle[fillstyle=solid,fillcolor=white](2,2){.1}
  \pscircle[fillstyle=solid,fillcolor=black](1,0){.1}
  \pscircle[fillstyle=solid,fillcolor=black](3,0){.1}
  \pscircle[fillstyle=solid,fillcolor=black](0,1){.1}
  \pscircle[fillstyle=solid,fillcolor=black](2,1){.1}
  \pscircle[fillstyle=solid,fillcolor=black](1,2){.1}
  \pscircle[fillstyle=solid,fillcolor=black](3,2){.1}
  \multiput(0,0)(0,1){2}{
  \multiput(0,0)(1,0){3}{\psdot(.5,.5)}
  }
   \psline[linewidth=0.1]{->}(1.5,-.5)(1.5,-1.5)
   \put(1.7,-1){sq-move}
   }

   \put(0,0.0){
   \put(0,1){\psline[linewidth=.06](0,0)(.5,-.5)}
   \put(0,2){\psline[linewidth=.06](0,0)(.5,-.5)}
   \put(2,0){\psline[linewidth=.06](0,0)(-.5,.5)}
   \put(2,2){\psline[linewidth=.06](0,0)(.5,-.5)}
   \put(3,0){\psline[linewidth=.06](0,0)(-.5,.5)}
   \put(2,1){\psline[linewidth=.06](0,0)(-.5,.5)}
   \psline[linewidth=.12,linecolor=gray]{*-*}(0,0)(1,0)
   \psline[linewidth=.12,linecolor=gray]{*-*}(1,1)(1,2)
   \psline[linewidth=.12,linecolor=gray]{*-*}(3,2)(3,1)
   \psarc[linecolor=red,linewidth=.05](0,0){.5}{0}{90}
   \multiput(0,1)(1,-1){2}{\psarc[linecolor=red,linewidth=.05](0,0){.5}{0}{90}}
   \psarc[linecolor=red,linewidth=.05](1,1){.5}{180}{270}
   \psarc[linecolor=red,linewidth=.05](1,2){.5}{180}{270}
   \psarc[linecolor=red,linewidth=.05](2,2){.5}{180}{270}
   \psarc[linecolor=red,linewidth=.05](2,0){.5}{0}{90}
   \psarc[linecolor=red,linewidth=.05](2,1){.5}{180}{270}
   \psarc[linecolor=red,linewidth=.05](3,1){.5}{180}{270}
   \psarc[linecolor=red,linewidth=.05](1,1){.5}{0}{90}
   \psarc[linecolor=red,linewidth=.05](2,1){.5}{0}{90}
   \psarc[linecolor=red,linewidth=.05](3,2){.5}{180}{270}
   \psline(0,0)(3,0)(3,2)(0,2)(0,0)
  \psline(0,1)(3,1)
  \psline(1,0)(1,2)
  \psline(2,0)(2,2)
  \psline[linewidth=0.01,linestyle=dashed](0,0)(2,2)(3,1)(2,0)(0,2)
  \psline[linewidth=0.01,linestyle=dashed](3,0)(1,2)(0,1)(1,0)(3,2)
  \pscircle[fillstyle=solid,fillcolor=white](0,0){.1}
  \pscircle[fillstyle=solid,fillcolor=white](2,0){.1}
  \pscircle[fillstyle=solid,fillcolor=white](1,1){.1}
  \pscircle[fillstyle=solid,fillcolor=white](3,1){.1}
  \pscircle[fillstyle=solid,fillcolor=white](0,2){.1}
  \pscircle[fillstyle=solid,fillcolor=white](2,2){.1}
  \pscircle[fillstyle=solid,fillcolor=black](1,0){.1}
  \pscircle[fillstyle=solid,fillcolor=black](3,0){.1}
  \pscircle[fillstyle=solid,fillcolor=black](0,1){.1}
  \pscircle[fillstyle=solid,fillcolor=black](2,1){.1}
  \pscircle[fillstyle=solid,fillcolor=black](1,2){.1}
  \pscircle[fillstyle=solid,fillcolor=black](3,2){.1}
  \multiput(0,0)(0,1){2}{
  \multiput(0,0)(1,0){3}{\psdot(.5,.5)}
  }
   }

  \end{pspicture}
  \caption{changes of impurities and slit curves under the local move}
  \label{fig:2-1-4}
 \end{center}
\end{figure}
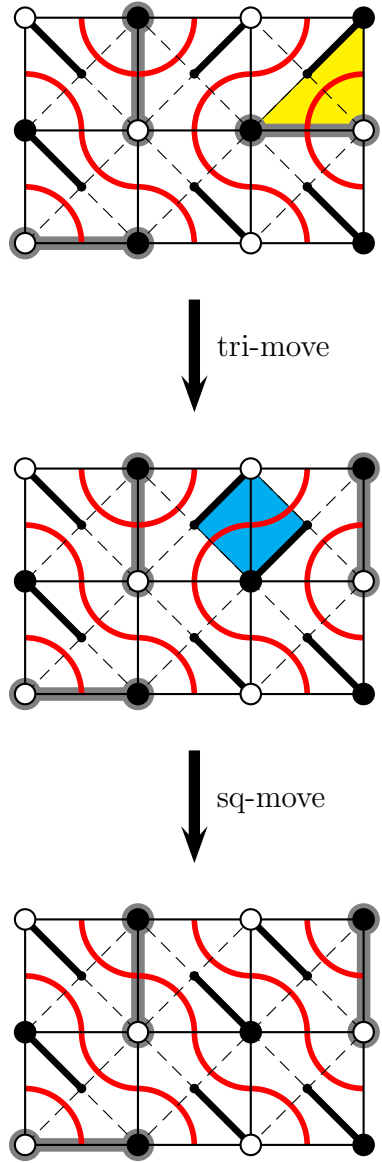

\noindent
$G$
is divided into some subgraphs by these slit curves, 
which we call domains. 
Vertices in 
$V_1$, $V_2$ 
are not in the same domain so that, by ignoring middle vertices, these domains can also be regarded as subgraphs of
$G_1$, $G_2$. 
Impurities 
always lie on slit curves and thus live in both of two neighboring domains. 
Figure \ref{fig:2-1-5} 
shows an example of domains corresponding to the dimer covering in Figure \ref{fig:2-1-3}(2). 
\\

%2-1-5
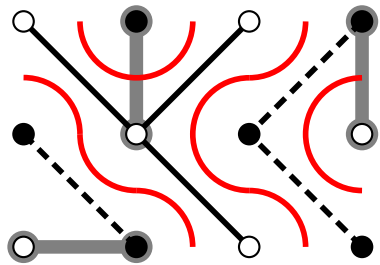
\begin{figure}[H]
 \begin{center}
  \begin{pspicture}(6,5)
   \psset{unit=1.5}
   \psline[linewidth=.12,linecolor=gray]{*-*}(0,0)(1,0)
   \psline[linewidth=.12,linecolor=gray]{*-*}(1,1)(1,2)
   \psline[linewidth=.12,linecolor=gray]{*-*}(3,2)(3,1)

   \psline[linewidth=0.05](0,2)(2,0)
   \psline[linewidth=0.05](1,1)(2,2)
   \psline[linestyle=dashed,linewidth=.05](0,1)(1,0)
   \psline[linestyle=dashed,linewidth=.05](3,0)(2,1)(3,2)
   \psarc[linecolor=red,linewidth=.05](1,1){.5}{180}{270}
   \psarc[linecolor=red,linewidth=.05](1,2){.5}{180}{360}
   \psarc[linecolor=red,linewidth=.05](2,2){.5}{270}{360}
   \psarc[linecolor=red,linewidth=.05](2,0){.5}{0}{90}
   \psarc[linecolor=red,linewidth=.05](2,1){.5}{90}{270}
   \psarc[linecolor=red,linewidth=.05](3,1){.5}{90}{270}

   \multiput(0,1)(1,-1){2}{\psarc[linecolor=red,linewidth=.05](0,0){.5}{0}{90}}
   \pscircle[fillstyle=solid,fillcolor=white](0,0){.1}
   \pscircle[fillstyle=solid,fillcolor=white](2,0){.1}
   \pscircle[fillstyle=solid,fillcolor=white](1,1){.1}
   \pscircle[fillstyle=solid,fillcolor=white](3,1){.1}
   \pscircle[fillstyle=solid,fillcolor=white](0,2){.1}
   \pscircle[fillstyle=solid,fillcolor=white](2,2){.1}
   \pscircle[fillstyle=solid,fillcolor=black](1,0){.1}
   \pscircle[fillstyle=solid,fillcolor=black](3,0){.1}
   \pscircle[fillstyle=solid,fillcolor=black](0,1){.1}
   \pscircle[fillstyle=solid,fillcolor=black](2,1){.1}
   \pscircle[fillstyle=solid,fillcolor=black](1,2){.1}
   \pscircle[fillstyle=solid,fillcolor=black](3,2){.1}
  \end{pspicture}
  \caption{Impurities are penetrated by a slid curve at their middle and
  thus they live both of neighboring domains}
  \label{fig:2-1-5}
 \end{center}
\end{figure}

\begin{proposition}
\label{properties of lines}
(0)
Slit curves do not branch and do not terminate inside
$G$. \\
(1)
The number of vertices in each domains is always odd.\\
(2)
Impurities always make pairings between two neighboring domains. \\
(3)
Slit curves are not closed. 
\end{proposition}
(0) and (3) 
implies that every slit curve terminates at boundary. 
Figure \ref{fig:2-1-6}
shows an example of (2). 
\\

%2-1-6
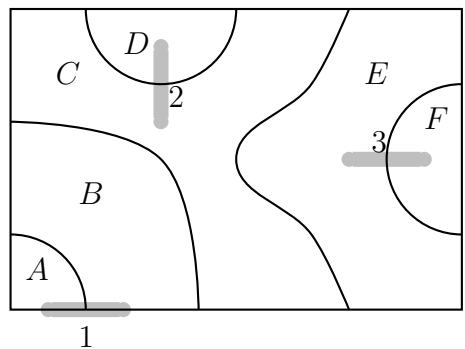
\begin{figure}[H]
 \begin{center}
  \begin{pspicture}(6,4)(0,-.5)
   \psline[linewidth=0.2,linecolor=lightgray,dotsize=.2]{*-*}(0.5,0)(1.5,0)
   \psline[linewidth=0.2,linecolor=lightgray,dotsize=.2]{*-*}(2,3.5)(2,2.5)
   \psline[linewidth=0.2,linecolor=lightgray,dotsize=.2]{*-*}(5.5,2)(4.5,2)
   \pspolygon(0,0)(6,0)(6,4)(0,4)
   \psarc(0,0){1}{0}{90}
   \pscurve(0,2.5)(2,2)(2.5,0)
   \psarc(2,4){1}{180}{360}
   \pscurve(4.5,4)(4.,3)(3,2)(4.,1)(4.5,0)
   \psarc(6,2){1}{90}{270}
   \put(.2,.4){$A$}
   \put(.9,1.4){$B$}
   \put(.6,3.){$C$}
   \put(1.5,3.4){$D$}
   \put(4.7,3.){$E$}
   \put(5.5,2.4){$F$}
   \put(0.9,-.5){$1$}
   \put(2.1,2.7){$2$}
   \put(4.8,2.1){$3$}
  \end{pspicture}
  \caption{An outline of Figure \ref{fig:2-1-3}$(2)$.
  Impurity $1$ connects domains $A$,$B$, impurity $2$ connects
  domains $C$, $D$ and impurity $3$ connects domain $E$, $F$.}
  \label{fig:2-1-6}
 \end{center}
\end{figure}

\begin{proof}
(0)
is clear. \\
(1)
It suffices to note that, 
when we add a block
\footnote{block is the one described in Figure 8(1).} to a domain, 
the number of vertices in this domain increases by two. 
\\
(2)
The number of impurities is equal to 
$\frac {m+n+1}{2}$, 
while that of domains is 
$m+n+1$, 
provided we have no closed curve
(we would have more if there are closed ones). 
By (1), 
we must have odd number of impurities in each domains. 
Since 
the number of impurities is half of that of domains, each domains have one impurity, hence impurities have to make pairings between domains. \\
(3)
If there were closed curves, 
the number of domains would be more than 
$m+n+1$, 
so that we would run out of impurities. 
\QED
\end{proof}
By Proposition \ref{properties of lines}(3), 
each domains, being regarded as subgraphs of 
$G_1$ 
or 
$G_2$, 
is a tree so that we obtain the spanning forest 
$F_1$, $F_2$ 
of 
$G_1$, $G_2$. 
Moreover, 
by Proposition \ref{properties of lines}(2), 
each subtrees of 
$F_1$, $F_2$ 
are paired by impurities so that we get an element 
$(F_1, F_2, \{ e_j \}_{j=1}^k) \in {\cal F}(G, P)$. 
Figure \ref{fig:2-1-7} 
shows a flowchart of our discussion. 
\\
%

%2-1-7
\begin{figure}[H]
 \begin{center}
  \begin{pspicture}(5,15)
   \psset{unit=1.5}
   \put(0,8){
   \put(-1,2){$(i)$}
   \put(0,1){\psline[linewidth=.06](0,0)(.5,-.5)}
   \put(0,2){\psline[linewidth=.06](0,0)(.5,-.5)}
   \put(2,0){\psline[linewidth=.06](0,0)(-.5,.5)}
   \put(2,2){\psline[linewidth=.06](0,0)(-.5,-.5)}
   \put(3,0){\psline[linewidth=.06](0,0)(-.5,.5)}
   \put(3,2){\psline[linewidth=.06](0,0)(-.5,-.5)}
   \psline[linewidth=.12,linecolor=gray]{*-*}(0,0)(1,0)
   \psline[linewidth=.12,linecolor=gray]{*-*}(1,1)(1,2)
   \psline[linewidth=.12,linecolor=gray]{*-*}(2,1)(3,1)
   \psarc[linecolor=red,linewidth=.05](0,0){.5}{0}{90}
   \multiput(0,1)(1,-1){2}{\psarc[linecolor=red,linewidth=.05](0,0){.5}{0}{90}}
   \psarc[linecolor=red,linewidth=.05](1,1){.5}{180}{270}
   \psarc[linecolor=red,linewidth=.05](1,2){.5}{180}{360}
   \psarc[linecolor=red,linewidth=.05](2,2){.5}{270}{360}
   \psarc[linecolor=red,linewidth=.05](2,0){.5}{0}{90}
   \psarc[linecolor=red,linewidth=.05](2,1){.5}{90}{270}
   \psarc[linecolor=red,linewidth=.05](3,1){.5}{90}{270}
   \psline(0,0)(3,0)(3,2)(0,2)(0,0)
   \psline(0,1)(3,1)
   \psline(1,0)(1,2)
   \psline(2,0)(2,2)
   \psline[linewidth=0.01,linestyle=dashed](0,0)(2,2)(3,1)(2,0)(0,2)
   \psline[linewidth=0.01,linestyle=dashed](3,0)(1,2)(0,1)(1,0)(3,2)
   \pscircle[fillstyle=solid,fillcolor=white](0,0){.1}
   \pscircle[fillstyle=solid,fillcolor=white](2,0){.1}
   \pscircle[fillstyle=solid,fillcolor=white](1,1){.1}
   \pscircle[fillstyle=solid,fillcolor=white](3,1){.1}
   \pscircle[fillstyle=solid,fillcolor=white](0,2){.1}
   \pscircle[fillstyle=solid,fillcolor=white](2,2){.1}
   \pscircle[fillstyle=solid,fillcolor=black](1,0){.1}
   \pscircle[fillstyle=solid,fillcolor=black](3,0){.1}
   \pscircle[fillstyle=solid,fillcolor=black](0,1){.1}
   \pscircle[fillstyle=solid,fillcolor=black](2,1){.1}
   \pscircle[fillstyle=solid,fillcolor=black](1,2){.1}
   \pscircle[fillstyle=solid,fillcolor=black](3,2){.1}
   \multiput(0,0)(0,1){2}{
   \multiput(0,0)(1,0){3}{\psdot(.5,.5)}
   }
   }

   \put(0,4){
   \put(-1,2){$(ii)$}
   \psline[linewidth=.12,linecolor=gray]{*-*}(0,0)(1,0)
   \psline[linewidth=.12,linecolor=gray]{*-*}(1,1)(1,2)
   \psline[linewidth=.12,linecolor=gray]{*-*}(2,1)(3,1)
   \psarc[linecolor=red,linewidth=.05](0,0){.5}{0}{90}
   \multiput(0,1)(1,-1){2}{\psarc[linecolor=red,linewidth=.05](0,0){.5}{0}{90}}
   \psarc[linecolor=red,linewidth=.05](1,1){.5}{180}{270}
   \psarc[linecolor=red,linewidth=.05](1,2){.5}{180}{360}
   \psarc[linecolor=red,linewidth=.05](2,2){.5}{270}{360}
   \psarc[linecolor=red,linewidth=.05](2,0){.5}{0}{90}
   \psarc[linecolor=red,linewidth=.05](2,1){.5}{90}{270}
   \psarc[linecolor=red,linewidth=.05](3,1){.5}{90}{270}
   \psline(0,0)(3,0)(3,2)(0,2)(0,0)
   \psline(0,1)(3,1)
   \psline(1,0)(1,2)
   \psline(2,0)(2,2)
   \pscircle[fillstyle=solid,fillcolor=white](0,0){.1}
   \pscircle[fillstyle=solid,fillcolor=white](2,0){.1}
   \pscircle[fillstyle=solid,fillcolor=white](1,1){.1}
   \pscircle[fillstyle=solid,fillcolor=white](3,1){.1}
   \pscircle[fillstyle=solid,fillcolor=white](0,2){.1}
   \pscircle[fillstyle=solid,fillcolor=white](2,2){.1}
   \pscircle[fillstyle=solid,fillcolor=black](1,0){.1}
   \pscircle[fillstyle=solid,fillcolor=black](3,0){.1}
   \pscircle[fillstyle=solid,fillcolor=black](0,1){.1}
   \pscircle[fillstyle=solid,fillcolor=black](2,1){.1}
   \pscircle[fillstyle=solid,fillcolor=black](1,2){.1}
   \pscircle[fillstyle=solid,fillcolor=black](3,2){.1}
   }

   \put(0,0){
   \put(-1,2){$(iii)$}
   \psline[linewidth=.12,linecolor=gray]{*-*}(0,0)(1,0)
   \psline[linewidth=.12,linecolor=gray]{*-*}(1,1)(1,2)
   \psline[linewidth=.12,linecolor=gray]{*-*}(2,1)(3,1)
   \psline(0,0)(3,0)(3,2)(0,2)(0,0)
   \psline(0,1)(3,1)
   \psline(1,0)(1,2)
   \psline(2,0)(2,2)
   \psline[linewidth=.05](0,2)(2,0)
   \psline[linewidth=.05](1,1)(2,2)
   \psline[linewidth=.05,linestyle=dashed](0,1)(1,0)
   \psline[linewidth=.05,linestyle=dashed](3,2)(2,1)(3,0)
   \pscircle[fillstyle=solid,fillcolor=white](0,0){.1}
   \pscircle[fillstyle=solid,fillcolor=white](2,0){.1}
   \pscircle[fillstyle=solid,fillcolor=white](1,1){.1}
   \pscircle[fillstyle=solid,fillcolor=white](3,1){.1}
   \pscircle[fillstyle=solid,fillcolor=white](0,2){.1}
   \pscircle[fillstyle=solid,fillcolor=white](2,2){.1}
   \pscircle[fillstyle=solid,fillcolor=black](1,0){.1}
   \pscircle[fillstyle=solid,fillcolor=black](3,0){.1}
   \pscircle[fillstyle=solid,fillcolor=black](0,1){.1}
   \pscircle[fillstyle=solid,fillcolor=black](2,1){.1}
   \pscircle[fillstyle=solid,fillcolor=black](1,2){.1}
   \pscircle[fillstyle=solid,fillcolor=black](3,2){.1}
   }

  \end{pspicture}
  \caption{$(i)$ dimer covering $(ii)$ corresponding configuration of slit
  curves
  and impurities $(iii)$ corresponding spanning forests of $G_1$ and
  $G_2$
  with condition $(P)$. Thick lines are forest of $G_1$, and dotted
  lines are that of $G_2$.}\label{fig:2-1-7}
 \end{center}
\end{figure}
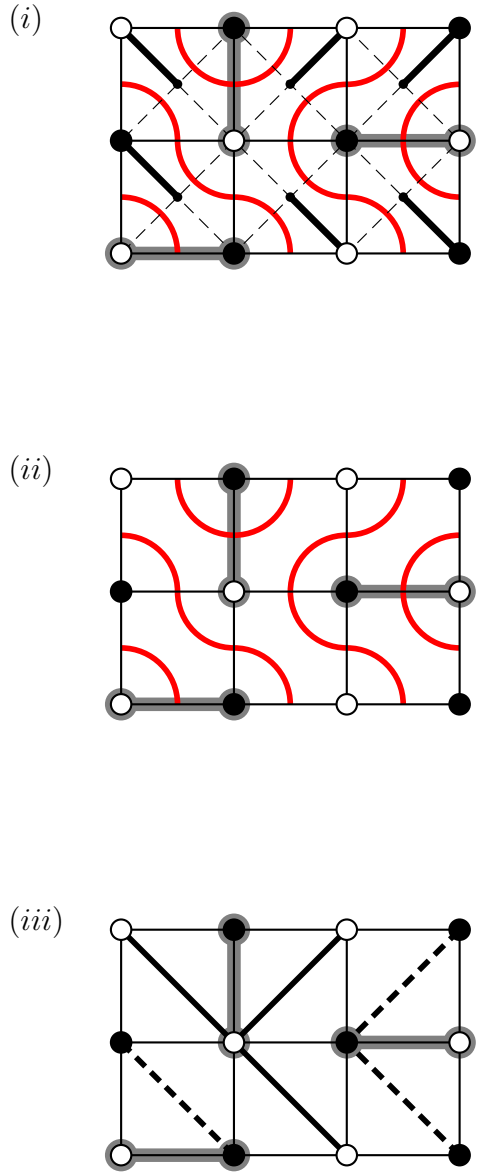

\noindent
Thus it suffices to construct the inverse mapping to finish the proof of Theorem \ref{pairing}. 
\begin{proposition}
\label{opposite way}
For given element 
$(F_1, F_2, \{ e_j \}_{j=1}^k) \in {\cal F}(G, P)$, 
we can find the corresponding dimer covering uniquely. 
\end{proposition}
\begin{proof}
It clearly suffices to find 
the dimer covering on each subtrees of spanning forests, being regarded as subgraph of 
$G$ 
by adding middle vertices. 
Each subtrees are further divided into a number of subtrees by impurities, 
where the numbers of vertices are always even.
It remains  
to make dimer coverings on each subtrees, 
regarding the impurity as the root. 
\QED
\end{proof}
The proof of 
Theorem \ref{pairing}
is completed.
\begin{remark}
Theorem \ref{pairing} 
also works for graphs which is made by composing the unit cube freely, provided the circumference 
$L$ 
of that satisfies 
$L \in 4 {\bf N} + 2$, 
in which case the number of impurities is equal to 
$\frac {L+2}{2}$ (Figure \ref{fig:2-1-8}). 
\end{remark}
%

%2-1-8
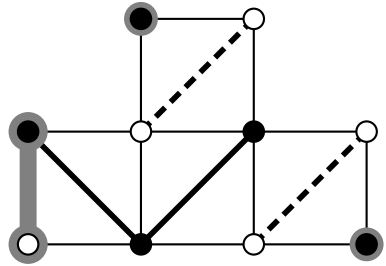
\begin{figure}[H]
 \begin{center}
  \begin{pspicture}(5,3)
   \psset{unit=1.5}
   \psline(0,0)(3,0)   \psline(0,1)(3,1) \psline(1,2)(2,2)
   \psline(0,0)(0,1)\psline(1,0)(1,2)\psline(2,0)(2,2)\psline(3,0)(3,1)
   \psline[linewidth=.05](0,1)(1,0)(2,1)
   \psline[linewidth=.05,linestyle=dashed](1,1)(2,2)
   \psline[linewidth=.05,linestyle=dashed](2,0)(3,1)
   \psline[linewidth=.15,linecolor=gray]{*-*}(0,0)(0,1)
   \pscircle[fillstyle=solid,fillcolor=gray,linecolor=gray](1,2){.15}
   \pscircle[fillstyle=solid,fillcolor=gray,linecolor=gray](3,0){.15}
   \pscircle[fillstyle=solid,fillcolor=white](0,0){.1}
   \pscircle[fillstyle=solid,fillcolor=white](2,0){.1}
   \pscircle[fillstyle=solid,fillcolor=white](1,1){.1}
   \pscircle[fillstyle=solid,fillcolor=white](3,1){.1}
   \pscircle[fillstyle=solid,fillcolor=white](2,2){.1}
   \pscircle[fillstyle=solid,fillcolor=black](1,0){.1}
   \pscircle[fillstyle=solid,fillcolor=black](3,0){.1}
   \pscircle[fillstyle=solid,fillcolor=black](0,1){.1}
   \pscircle[fillstyle=solid,fillcolor=black](2,1){.1}
   \pscircle[fillstyle=solid,fillcolor=black](1,2){.1}
  \end{pspicture}
  \caption{Example of spanning forest in general case}
  \label{fig:2-1-8}
 \end{center}
\end{figure}

%
%%%%%
\subsection{Construction of a bijection : for 
$G^{(k)}$}
In this subsection, 
we set  
$G := G^{(k)}$
and numerate its terminals as 
$T_1$, $T_2$, $\cdots$, $T_k$. 
To state our theorem, 
we need numerous notations which are introduced here. 
As was done for 
$G^{(m,n)}$, 
let 
$G_j$ 
($j=1,2$)
be the graph such that 
$V(G_j) = V_j$, 
and for 
$x,y \in V_j$, 
we set 
$(x,y) \in E(G_j)$ 
iff there is 
$z \in V_3$, 
which we call the middle vertex, with 
$(x,z), (z,y) \in E_1$. 
Figure \ref{fig:2-2-1}  
shows 
$G_1$, $G_2$ 
for the example given in Figure \ref{fig:1-3} (2). 
Putting 
back the middle vertices on 
$G_1$, $G_2$ 
yields subgraphs of 
$G$ 
which we call
$G'_1$, $G'_2$ 
(Figure \ref{fig:2-2-2}).\\
%

%2-2-1
\begin{figure}[H]
 \begin{center}
  \begin{pspicture}(6,10)(0,-1)
   \put(0,5){
   \psline(0,0)(1,0)(1,3)\psline(1,1)(4,1)\psline(1,2)(2,2)\psline(2,2)(2,1)
   \pscircle[fillstyle=solid,fillcolor=white](0,0){.17}
   \pscircle[fillstyle=solid,fillcolor=white](0,0){.1}
   \pscircle[fillstyle=solid,fillcolor=white](1,3){.17}
   \pscircle[fillstyle=solid,fillcolor=white](1,3){.1}
   \pscircle[fillstyle=solid,fillcolor=white](1,0){.1}
   \pscircle[fillstyle=solid,fillcolor=white](1,1){.1}
   \pscircle[fillstyle=solid,fillcolor=white](1,2){.1}
   \pscircle[fillstyle=solid,fillcolor=white](2,1){.1}
   \pscircle[fillstyle=solid,fillcolor=white](2,2){.1}
   \pscircle[fillstyle=solid,fillcolor=white](3,1){.1}
   \pscircle[fillstyle=solid,fillcolor=white](4,1){.17}
   \pscircle[fillstyle=solid,fillcolor=white](4,1){.1}
   \put(1.5,-.5){$G_1$}
   }

   \put(.5,-.5){
   \multiput(0,0)(0,1){4}{\pscircle[fillstyle=solid,fillcolor=black](0,0){.1}}
   \multiput(1,0)(0,1){4}{\pscircle[fillstyle=solid,fillcolor=black](0,0){.1}}
   \multiput(2,1)(0,1){3}{\pscircle[fillstyle=solid,fillcolor=black](0,0){.1}}
   \multiput(3,1)(0,1){2}{\pscircle[fillstyle=solid,fillcolor=black](0,0){.1}}
   \psline(0,0)(0,3) \psline(1,0)(1,3) \psline(2,1)(2,3)
   \psline(3,1)(3,2)
   \psline(0,0)(1,0)\psline(0,1)(3,1)\psline(0,2)(3,2)\psline(0,3)(2,3)
   }
   {\psset{linecolor=lightgray}
   \psline(0,0)(1,0)(1,3)\psline(1,1)(4,1)\psline(1,2)(2,2)\psline(2,2)(2,1)
   \pscircle[fillstyle=solid,fillcolor=white](0,0){.17}
   \pscircle[fillstyle=solid,fillcolor=white](0,0){.1}
   \pscircle[fillstyle=solid,fillcolor=white](1,3){.17}
   \pscircle[fillstyle=solid,fillcolor=white](1,3){.1}
   \pscircle[fillstyle=solid,fillcolor=white](1,0){.1}
   \pscircle[fillstyle=solid,fillcolor=white](1,1){.1}
   \pscircle[fillstyle=solid,fillcolor=white](1,2){.1}
   \pscircle[fillstyle=solid,fillcolor=white](2,1){.1}
   \pscircle[fillstyle=solid,fillcolor=white](2,2){.1}
   \pscircle[fillstyle=solid,fillcolor=white](3,1){.1}
   \pscircle[fillstyle=solid,fillcolor=white](4,1){.17}
   \pscircle[fillstyle=solid,fillcolor=white](4,1){.1}
   \put(2,-1){$G_2$}
   }
  \end{pspicture}
  \caption{$G_1,G_2$ corresponding to the example in Figure
  \ref{fig:1-3} $(2)$}
  \label{fig:2-2-1}
 \end{center}
\end{figure}

%
%2-2-2
\begin{figure}[H]
 \begin{center}
  \begin{pspicture}(6,10)(0,-1.5)
   \psset{unit=1.5}
   \put(0,5){
   \put(.5,-.5){
   \psset{linecolor=lightgray}
   \multiput(0,0)(0,1){4}{\pscircle[fillstyle=solid,fillcolor=lightgray](0,0){.1}}
   \multiput(1,0)(0,1){4}{\pscircle[fillstyle=solid,fillcolor=lightgray](0,0){.1}}
   \multiput(2,1)(0,1){3}{\pscircle[fillstyle=solid,fillcolor=lightgray](0,0){.1}}
   \multiput(3,1)(0,1){2}{\pscircle[fillstyle=solid,fillcolor=lightgray](0,0){.1}}
   \psline(0,0)(0,3) \psline(1,0)(1,3) \psline(2,1)(2,3)
   \psline(3,1)(3,2)
   \psline(0,0)(1,0)\psline(0,1)(3,1)\psline(0,2)(3,2)\psline(0,3)(2,3)
   }
   \psset{linestyle=dashed}
   \psline(0,0)(1,0)(1,3)\psline(1,1)(4,1)\psline(1,2)(2,2)\psline(2,2)(2,1)
   \psset{linestyle=solid}
   \pscircle[fillstyle=solid,fillcolor=white](0,0){.17}
   \pscircle[fillstyle=solid,fillcolor=white](0,0){.1}
   \pscircle[fillstyle=solid,fillcolor=white](1,3){.17}
   \pscircle[fillstyle=solid,fillcolor=white](1,3){.1}
   \pscircle[fillstyle=solid,fillcolor=white](1,0){.1}
   \pscircle[fillstyle=solid,fillcolor=white](1,1){.1}
   \pscircle[fillstyle=solid,fillcolor=white](1,2){.1}
   \pscircle[fillstyle=solid,fillcolor=white](2,1){.1}
   \pscircle[fillstyle=solid,fillcolor=white](2,2){.1}
   \pscircle[fillstyle=solid,fillcolor=white](3,1){.1}
   \pscircle[fillstyle=solid,fillcolor=white](4,1){.17}
   \pscircle[fillstyle=solid,fillcolor=white](4,1){.1}
   \psdot(0.5,0)\multiput(0,0)(0,1){3}{\psdot(1.,0.5)}
   \multiput(0,0)(1,0){3}{\psdot(1.5,1)}
   \psdot(1.5,2)\psdot(2,1.5)
   \put(1.5,-1.){$G'_1$}
   }

   \put(.5,-.5){
   \multiput(0,0)(0,1){4}{\pscircle[fillstyle=solid,fillcolor=black](0,0){.1}}
   \multiput(1,0)(0,1){4}{\pscircle[fillstyle=solid,fillcolor=black](0,0){.1}}
   \multiput(2,1)(0,1){3}{\pscircle[fillstyle=solid,fillcolor=black](0,0){.1}}
   \multiput(3,1)(0,1){2}{\pscircle[fillstyle=solid,fillcolor=black](0,0){.1}}
   \psset{linestyle=dashed}
   \psline(0,0)(0,3) \psline(1,0)(1,3) \psline(2,1)(2,3)
   \psline(3,1)(3,2)
   \psline(0,0)(1,0)\psline(0,1)(3,1)\psline(0,2)(3,2)\psline(0,3)(2,3)
   }
   {\psset{linecolor=lightgray}
   \psline(0,0)(1,0)(1,3)\psline(1,1)(4,1)\psline(1,2)(2,2)\psline(2,2)(2,1)
   \pscircle[fillstyle=solid,fillcolor=white](0,0){.17}
   \pscircle[fillstyle=solid,fillcolor=white](0,0){.1}
   \pscircle[fillstyle=solid,fillcolor=white](1,3){.17}
   \pscircle[fillstyle=solid,fillcolor=white](1,3){.1}
   \pscircle[fillstyle=solid,fillcolor=white](1,0){.1}
   \pscircle[fillstyle=solid,fillcolor=white](1,1){.1}
   \pscircle[fillstyle=solid,fillcolor=white](1,2){.1}
   \pscircle[fillstyle=solid,fillcolor=white](2,1){.1}
   \pscircle[fillstyle=solid,fillcolor=white](2,2){.1}
   \pscircle[fillstyle=solid,fillcolor=white](3,1){.1}
   \pscircle[fillstyle=solid,fillcolor=white](4,1){.17}
   \pscircle[fillstyle=solid,fillcolor=white](4,1){.1}
   \put(2,-1){$G'_2$}
   }
   \psdot(0.5,0)\multiput(0,0)(0,1){3}{\psdot(1.,0.5)}
   \multiput(-1,0)(1,0){4}{\psdot(1.5,1)}
   \multiput(-1,1)(1,0){3}{\psdot(1.5,1)}
   \multiput(-1,-1)(1,0){2}{\psdot(1.5,1)}
   \multiput(-.5,-1.5)(0,1){3}{\psdot(1.5,1)}
   \multiput(.5,-.5)(0,1){3}{\psdot(1.5,1)}
   \multiput(1.5,-.5)(0,1){2}{\psdot(1.5,1)}

  \end{pspicture}
  \caption{$G'_1,G'_2$ corresponding to the example in Figure
  \ref{fig:1-3} $(2)$}
  \label{fig:2-2-2}
 \end{center}
\end{figure}

Let 
$x, y$ 
be vertices. 
We say that 
$x$ 
is directly connected to 
$y$ 
iff we have the edge 
$e=(x,y)$. 
We call 
$x \in V_3$ 
is a boundary vertex iff 
$x$ 
lies in the boundary of 
$G$. 
Let 
$\overline{G}$ 
be the graph obtained from 
$G$ 
by the following procedure : 
(i) 
add an imaginary vertex 
$R$ 
which we call the root, and 
(ii) 
connect all terminals and boundary vertices directly to 
$R$. 
We call 
edges of the form 
$e=(R,y)$ 
the outer edge. 
$\overline{G}$ 
for the example in Figure \ref{fig:1-3} (2) is given in Figure \ref{fig:2-2-3}.
\\
%

%2-2-3
\begin{figure}[H]
 \begin{center}
  \begin{pspicture}(6,9)(0,-2.5)
   \psset{unit=2}
   \pscurve(0,0)(-1,1)(-1,4)(5,5)
   \pscurve(0.5,1)(-.5,1.5)(-.5,3.5)(5,5)
   \pscurve(0.5,2)(-.0,2.5)(-.0,3.)(5,5)
   \pscurve(1,3)(2,3.8)(5,5)
   \pscurve(2,2.5)(3,3.8)(5,5)
   \pscurve(2.5,2.)(4,3.8)(5,5)
   \pscurve(3,1.5)(4,3.)(5,5)
   \pscurve(4,1.)(4.7,3.)(5,5)
   \pscurve(3,.5)(3.5,0)(4.5,.5)(5,5)
   \pscurve(2,.5)(3,-0.5)(5,.5)(5,5)
   \pscurve(1.5,0)(3,-0.8)(5.5,.0)(5,5)
   \pscurve(1.,-0.5)(2,-1)(6,-0.5)(5,5)
   \pscircle[fillstyle=solid,fillcolor=white,linecolor=white](5,5){.2}
   \put(4.9,4.95){\large$\bigstar$}
   \put(.5,-.5){
   \psline(0,0)(2,2)(3,1)(3.5,1.5)(3,2)(2,1)(0,3)(.5,3.5)(1,3)(0,2)(1,1)}
   \psline(0,0)(2.5,2.5)
   \psline(1.5,2.5)(2.5,1.5)
   \psline(.5,.5)(1.5,-.5)
   \psline(0,0)(.5,-.5)
   \put(.5,-.5){
   \multiput(0,0)(0,1){4}{\pscircle[fillstyle=solid,fillcolor=black](0,0){.1}}
   \multiput(1,0)(0,1){4}{\pscircle[fillstyle=solid,fillcolor=black](0,0){.1}}
   \multiput(2,1)(0,1){3}{\pscircle[fillstyle=solid,fillcolor=black](0,0){.1}}
   \multiput(3,1)(0,1){2}{\pscircle[fillstyle=solid,fillcolor=black](0,0){.1}}
   \psset{linestyle=dashed}
   \psline(0,0)(0,3) \psline(1,0)(1,3) \psline(2,1)(2,3)
   \psline(3,1)(3,2)
   \psline(0,0)(1,0)\psline(0,1)(3,1)\psline(0,2)(3,2)\psline(0,3)(2,3)
   }
   {
   \psset{linestyle=dashed}
   \psline(0,0)(1.5,0)\psline(1,-.5)(1,3)\psline(0.5,1)(4,1)\psline(.5,2)(2.5,2)\psline(2,2.5)(2,.5)
   \psline(3,.5)(3,1.5)
   \psset{linestyle=solid}
   \pscircle[fillstyle=solid,fillcolor=white](0,0){.1}
   \pscircle[fillstyle=solid,fillcolor=white](0,0){.05}
   \pscircle[fillstyle=solid,fillcolor=white](1,3){.1}
   \pscircle[fillstyle=solid,fillcolor=white](1,3){.05}
   \pscircle[fillstyle=solid,fillcolor=white](1,0){.1}
   \pscircle[fillstyle=solid,fillcolor=white](1,1){.1}
   \pscircle[fillstyle=solid,fillcolor=white](1,2){.1}
   \pscircle[fillstyle=solid,fillcolor=white](2,1){.1}
   \pscircle[fillstyle=solid,fillcolor=white](2,2){.1}
   \pscircle[fillstyle=solid,fillcolor=white](3,1){.1}
   \pscircle[fillstyle=solid,fillcolor=white](4,1){.1}
   \pscircle[fillstyle=solid,fillcolor=white](4,1){.05}
   \put(2,-1){$\overline{G}$}
   }
   \psdot(0.5,0)\multiput(0,0)(0,1){3}{\psdot(1.,0.5)}
   \multiput(-1,0)(1,0){4}{\psdot(1.5,1)}
   \multiput(-1,1)(1,0){3}{\psdot(1.5,1)}
   \multiput(-1,-1)(1,0){2}{\psdot(1.5,1)}
   \multiput(-.5,-1.5)(0,1){3}{\psdot(1.5,1)}
   \multiput(.5,-.5)(0,1){3}{\psdot(1.5,1)}
   \multiput(1.5,-.5)(0,1){2}{\psdot(1.5,1)}

  \end{pspicture}
  \caption{$\overline{G}$ for the example in Figure $\ref{fig:1-3}$-(2)}
  \label{fig:2-2-3}
 \end{center}
\end{figure}
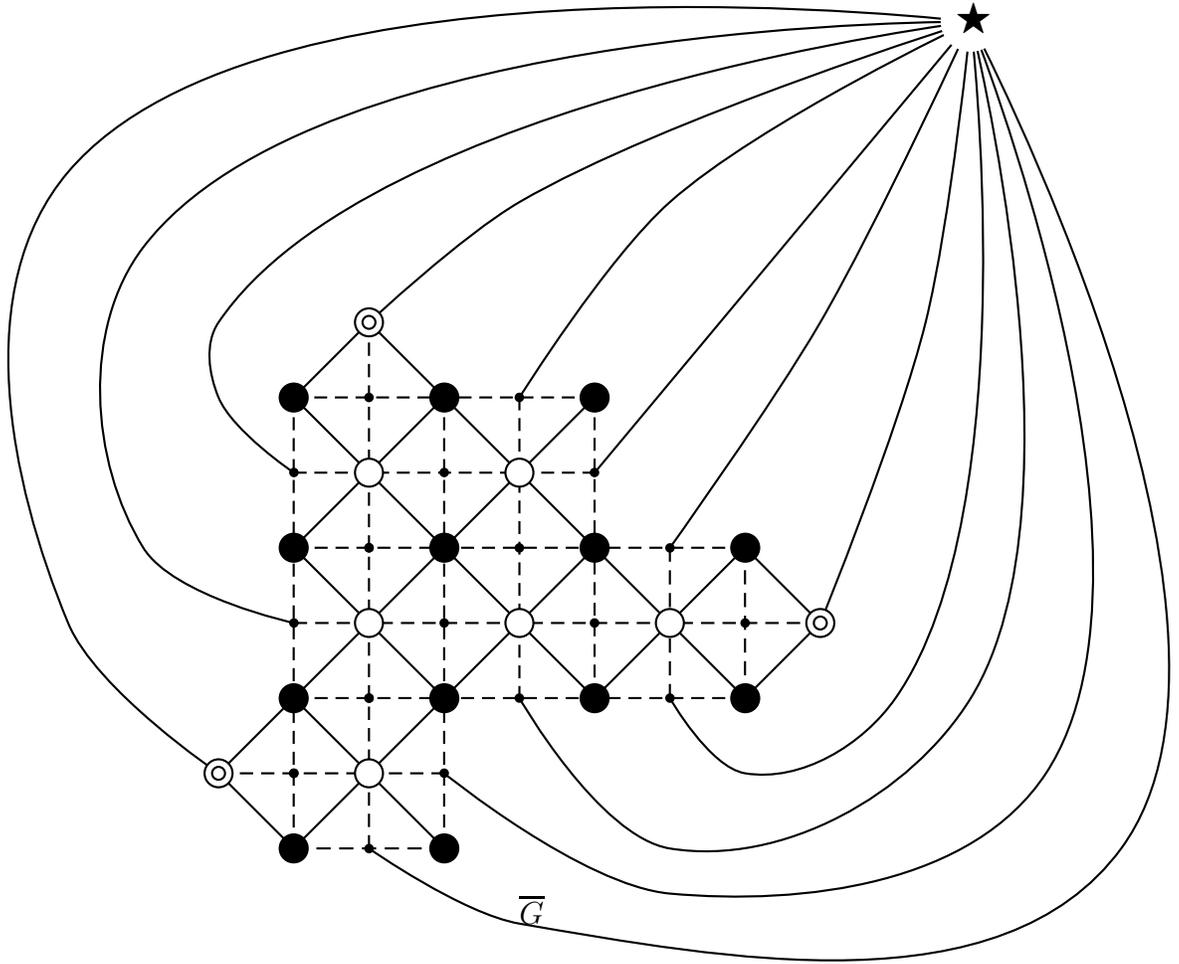

Let 
$\overline{G_1}$ 
be the graph such that 
$V(\overline{G_1}) = V_1 \cup \{ R \}$ 
and for 
$x,y \in V(\overline{G_1})$, 
we set 
$(x,y) \in E(\overline{G_1})$ 
iff 
$x=R$, $y \in \{ T_j \}_{j=1}^k$, 
or there is 
$z \in V_3$ 
with 
$(x,z), (z,y) \in E(\overline{G})$. 
$\overline{G_1}$ 
for the example in Figure \ref{fig:1-3} (2) is given in Figure \ref{fig:2-2-4}. 
Putting back middle vertices on 
$\overline{G_1}$ 
yields a sugraph 
$\overline{G'_1}$ 
of 
$\overline{G}$ (Figure \ref{fig:2-2-5}). \\
%

%2-2-4
\begin{figure}[H]
 \begin{center}
  \begin{pspicture}(6,10)(0,-2.5)
   \psset{unit=2}
   \pscurve(0,0)(-1,1)(-1,4)(5,5)
   \pscurve(1,1)(-.5,1.5)(-.5,3.5)(5,5)
   \pscurve(1,2)(-.0,2.5)(-.0,3.)(5,5)
   \pscurve(1,3)(2,3.8)(5,5)
   \pscurve(2,2.)(3,3.8)(5,5)
   \pscurve(2.,2.)(4,3.8)(5,5)
   \pscurve(3,1.)(4,3.)(5,5)
   \pscurve(4,1.)(4.7,3.)(5,5)
   \pscurve(3,1)(3.5,0)(4.5,.5)(5,5)
   \pscurve(2,1)(3,-0.5)(5,.5)(5,5)
   \pscurve(1,0)(3,-0.8)(5.5,.0)(5,5)
   \pscurve(1,0)(2,-1)(6,-0.5)(5,5)
   \pscircle[fillstyle=solid,fillcolor=white,linecolor=white](5,5){.2}
   \put(4.9,4.95){\large$\bigstar$}
   \put(.5,-.5){
   \psset{linestyle=dashed}

   }
   {
   \psset{linestyle=solid}
   \psline(0,0)(1.,0)
   \psline(1,0)(1,3)
   \psline(1,1)(4,1)\psline(1,2)(2,2)
   \psline(2,2)(2,1)
   \psset{linestyle=solid}
   \pscircle[fillstyle=solid,fillcolor=white](0,0){.1}
   \pscircle[fillstyle=solid,fillcolor=white](0,0){.05}
   \pscircle[fillstyle=solid,fillcolor=white](1,3){.1}
   \pscircle[fillstyle=solid,fillcolor=white](1,3){.05}
   \pscircle[fillstyle=solid,fillcolor=white](1,0){.1}
   \pscircle[fillstyle=solid,fillcolor=white](1,1){.1}
   \pscircle[fillstyle=solid,fillcolor=white](1,2){.1}
   \pscircle[fillstyle=solid,fillcolor=white](2,1){.1}
   \pscircle[fillstyle=solid,fillcolor=white](2,2){.1}
   \pscircle[fillstyle=solid,fillcolor=white](3,1){.1}
   \pscircle[fillstyle=solid,fillcolor=white](4,1){.1}
   \pscircle[fillstyle=solid,fillcolor=white](4,1){.05}
   \put(2,-1){}
   }

  \end{pspicture}
  \caption{$\overline{G_1}$
for the example in Figure $\ref{fig:1-3}$-(2)}
  \label{fig:2-2-4}
 \end{center}
\end{figure}
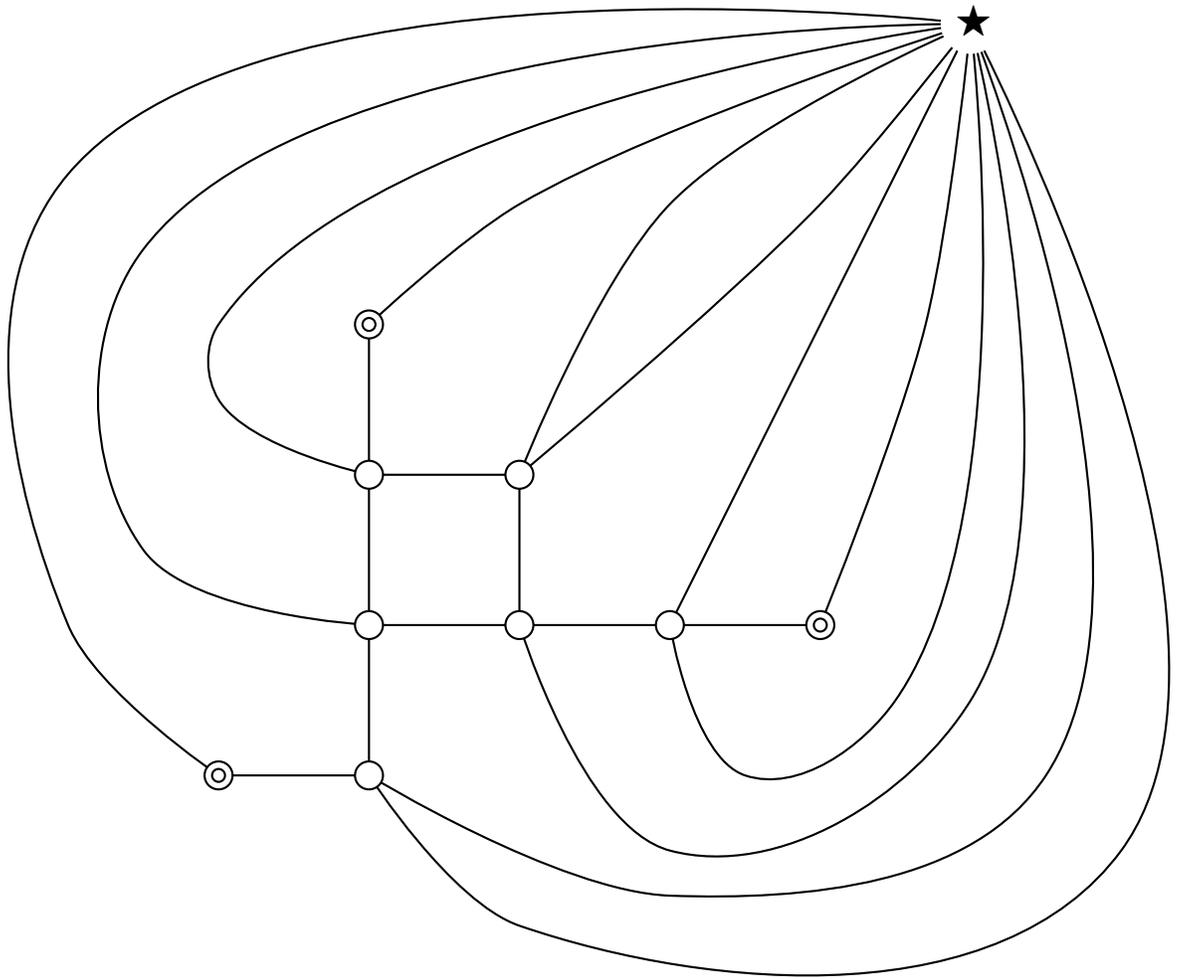

%2-2-5

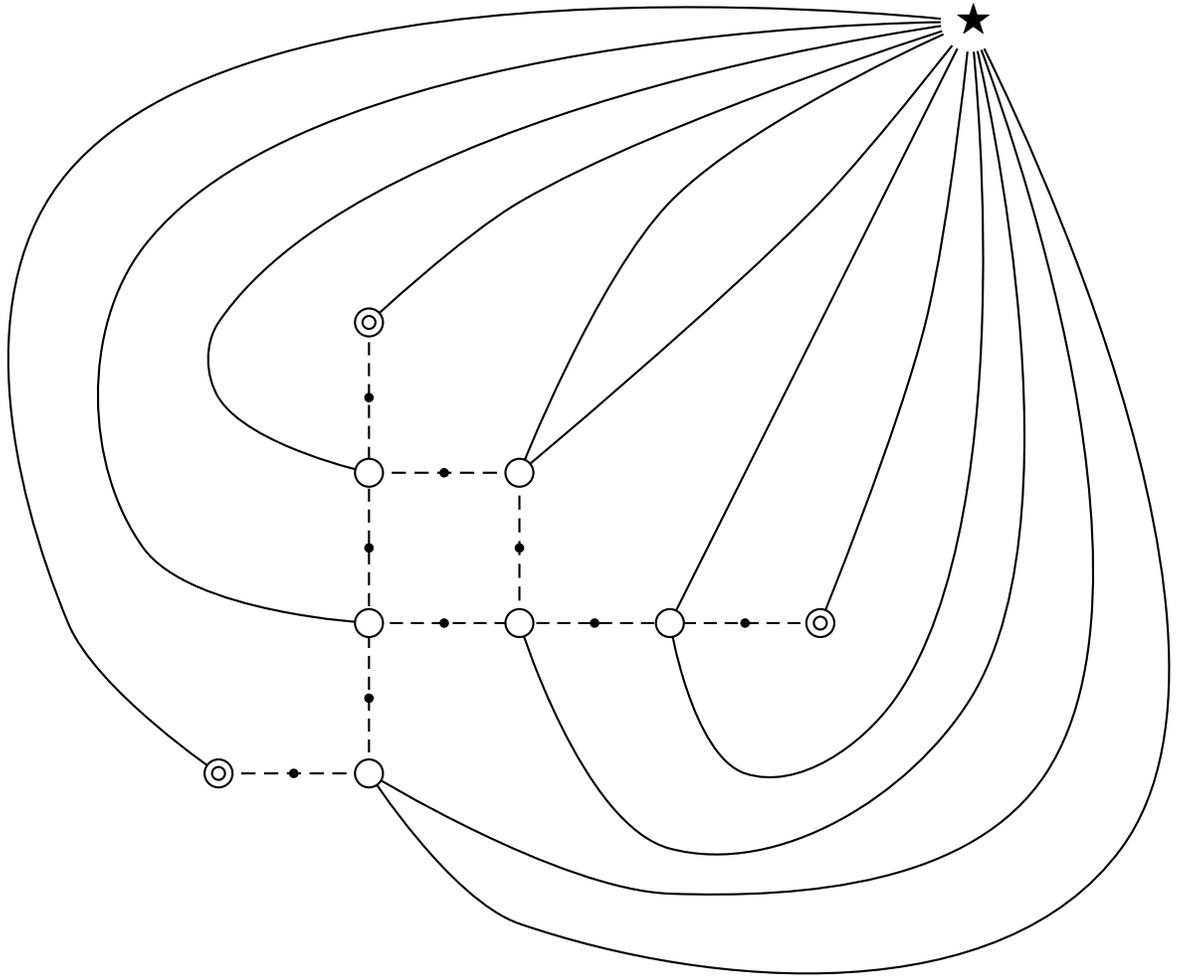
\begin{figure}[H]
 \begin{center}
  \begin{pspicture}(6,10)(0,-2.5)
   \psset{unit=2}
   \pscurve(0,0)(-1,1)(-1,4)(5,5)
   \pscurve(1,1)(-.5,1.5)(-.5,3.5)(5,5)
   \pscurve(1,2)(-.0,2.5)(-.0,3.)(5,5)
   \pscurve(1,3)(2,3.8)(5,5)
   \pscurve(2,2.)(3,3.8)(5,5)
   \pscurve(2.,2.)(4,3.8)(5,5)
   \pscurve(3,1.)(4,3.)(5,5)
   \pscurve(4,1.)(4.7,3.)(5,5)
   \pscurve(3,1)(3.5,0)(4.5,.5)(5,5)
   \pscurve(2,1)(3,-0.5)(5,.5)(5,5)
   \pscurve(1,0)(3,-0.8)(5.5,.0)(5,5)
   \pscurve(1,0)(2,-1)(6,-0.5)(5,5)
   \pscircle[fillstyle=solid,fillcolor=white,linecolor=white](5,5){.2}
   \put(4.9,4.95){\large$\bigstar$}
   \put(.5,-.5){
   \psset{linestyle=dashed}

   }
   {
   \psset{linestyle=dashed}
   \psline(0,0)(1.,0)
   \psline(1,0)(1,3)
   \psline(1,1)(4,1)\psline(1,2)(2,2)
   \psline(2,2)(2,1)
   \psdot(0.5,0)\psdot(1,0.5)\psdot(1,1.5)\psdot(1,2.5)
   \psdot(1.5,1)\psdot(2.5,1)\psdot(3.5,1)
   \psdot(2,1.5)\psdot(1.5,2)
   \psset{linestyle=solid}
   \pscircle[fillstyle=solid,fillcolor=white](0,0){.1}
   \pscircle[fillstyle=solid,fillcolor=white](0,0){.05}
   \pscircle[fillstyle=solid,fillcolor=white](1,3){.1}
   \pscircle[fillstyle=solid,fillcolor=white](1,3){.05}
   \pscircle[fillstyle=solid,fillcolor=white](1,0){.1}
   \pscircle[fillstyle=solid,fillcolor=white](1,1){.1}
   \pscircle[fillstyle=solid,fillcolor=white](1,2){.1}
   \pscircle[fillstyle=solid,fillcolor=white](2,1){.1}
   \pscircle[fillstyle=solid,fillcolor=white](2,2){.1}
   \pscircle[fillstyle=solid,fillcolor=white](3,1){.1}
   \pscircle[fillstyle=solid,fillcolor=white](4,1){.1}
   \pscircle[fillstyle=solid,fillcolor=white](4,1){.05}
   \put(2,-1){}
   }

  \end{pspicture}
  \caption{$\overline{G'_1}$
for the example in Figure $\ref{fig:1-3}$-(2)}
  \label{fig:2-2-5}
 \end{center}
\end{figure}

\noindent
For a subgraph 
$A (\subset \overline{G_1})$ 
of 
$\overline{G_1}$, 
its edge 
$e=(x,y) \in E(A)$ 
contains vertex of 
$V_3$ 
at its middle 
(except 
$e$ 
connects a terminal and 
$R$). 
Adding such middle vertices 
yields a subgraph 
$A'$ 
of 
$\overline{G'_1}$. 
We always identify 
$A$ 
with 
$A'$ 
and thus regard 
$A$ 
as a subgraph of 
$\overline{G'_1}$. 
Conversely, 
for a subgraph 
$A' (\subset \overline{G'_1})$ 
of 
$\overline{G'_1}$, 
ignoring middle vertices from   
$A'$ 
yields a subgraph 
$A$ 
of 
$\overline{G_1}$.
Similarly, 
cutting outer edges from   
$A'$, 
we obtain a subgraph 
$\tilde{A}$ 
of 
$G'_1$. 
In both cases, 
we identify 
$A'$ 
also with 
$A$ 
or
$\tilde{A}$,  
and regard 
$A$ 
also as a subgraph of 
$\overline{G_1}$
or 
$G'_1$. 
\\

To state an analogue of 
Theorem \ref{pairing}, 
we further need some notations. 
Let 
$A (\subset \overline{G'_1})$
be a subgraph of 
$\overline{G'_1}$. 
\\ 

\noindent
(1)
We say that 
$A$ 
is a TI-domain iff 
(i) $A$ contains a terminal which is unique and connected directly to 
$R$, 
and 
(ii) $A$ has no boundary vertices (Figure \ref{fig:2-2-6}). 
We also call 
$A$ 
is a TI-tree if it is a tree (and so for other ones below). 
\\
%
%2-2-6
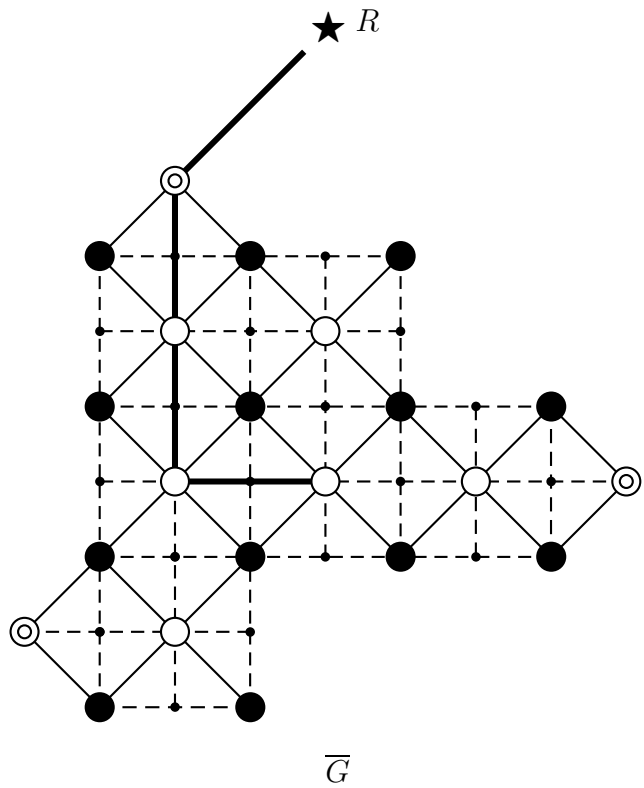
\begin{figure}[H]
 \begin{center}
  \begin{pspicture}(6,7)(0,-2.5)
   \psset{unit=2}
   \put(2.2,4){$R$}
   \psline[linewidth=0.04](2,4)(1,3)(1,1)(2,1)
   \pscircle[fillstyle=solid,fillcolor=white,linecolor=white](2,4){.2}
   \put(1.9,3.95){\large$\bigstar$}
   \put(.5,-.5){
   \psline(0,0)(2,2)(3,1)(3.5,1.5)(3,2)(2,1)(0,3)(.5,3.5)(1,3)(0,2)(1,1)}
   \psline(0,0)(2.5,2.5)
   \psline(1.5,2.5)(2.5,1.5)
   \psline(.5,.5)(1.5,-.5)
   \psline(0,0)(.5,-.5)
   \put(.5,-.5){
   \multiput(0,0)(0,1){4}{\pscircle[fillstyle=solid,fillcolor=black](0,0){.1}}
   \multiput(1,0)(0,1){4}{\pscircle[fillstyle=solid,fillcolor=black](0,0){.1}}
   \multiput(2,1)(0,1){3}{\pscircle[fillstyle=solid,fillcolor=black](0,0){.1}}
   \multiput(3,1)(0,1){2}{\pscircle[fillstyle=solid,fillcolor=black](0,0){.1}}
   \psset{linestyle=dashed}
   \psline(0,0)(0,3) \psline(1,0)(1,3) \psline(2,1)(2,3)
   \psline(3,1)(3,2)
   \psline(0,0)(1,0)\psline(0,1)(3,1)\psline(0,2)(3,2)\psline(0,3)(2,3)
   }
   {
   \psset{linestyle=dashed}
   \psline(0,0)(1.5,0)\psline(1,-.5)(1,3)\psline(0.5,1)(4,1)\psline(.5,2)(2.5,2)\psline(2,2.5)(2,.5)
   \psline(3,.5)(3,1.5)
   \psset{linestyle=solid}
   \pscircle[fillstyle=solid,fillcolor=white](0,0){.1}
   \pscircle[fillstyle=solid,fillcolor=white](0,0){.05}
   \pscircle[fillstyle=solid,fillcolor=white](1,3){.1}
   \pscircle[fillstyle=solid,fillcolor=white](1,3){.05}
   \pscircle[fillstyle=solid,fillcolor=white](1,0){.1}
   \pscircle[fillstyle=solid,fillcolor=white](1,1){.1}
   \pscircle[fillstyle=solid,fillcolor=white](1,2){.1}
   \pscircle[fillstyle=solid,fillcolor=white](2,1){.1}
   \pscircle[fillstyle=solid,fillcolor=white](2,2){.1}
   \pscircle[fillstyle=solid,fillcolor=white](3,1){.1}
   \pscircle[fillstyle=solid,fillcolor=white](4,1){.1}
   \pscircle[fillstyle=solid,fillcolor=white](4,1){.05}
   \put(2,-1){$\overline{G}$}
   }
   \psdot(0.5,0)\multiput(0,0)(0,1){3}{\psdot(1.,0.5)}
   \multiput(-1,0)(1,0){4}{\psdot(1.5,1)}
   \multiput(-1,1)(1,0){3}{\psdot(1.5,1)}
   \multiput(-1,-1)(1,0){2}{\psdot(1.5,1)}
   \multiput(-.5,-1.5)(0,1){3}{\psdot(1.5,1)}
   \multiput(.5,-.5)(0,1){3}{\psdot(1.5,1)}
   \multiput(1.5,-.5)(0,1){2}{\psdot(1.5,1)}

  \end{pspicture}
  \caption{An example of TI-domain}
  \label{fig:2-2-6}
 \end{center}
\end{figure}

\noindent
(2)
We say that 
$A$ 
is a TO-domain iff 
(i) $A$ contains a terminal which is unique and is not connected directly to 
$R$, 
and 
(ii) $A$ 
has boundary vertices all of which are connected directly to 
$R$ 
(Figure \ref{fig:2-2-7}). 
\\
%

%2-2-7
\begin{figure}[H]
 \begin{center}
  \begin{pspicture}(6,7)(0,-2.5)
   \psset{unit=2}
   \put(2.7,-.5){$R$}
   \psline[linewidth=0.04](1,3)(1,1)(2,1)(2,.5)(2.5,-.5)
   \pscircle[fillstyle=solid,fillcolor=white,linecolor=white](2.5,-.5){.2}
   \put(2.4,-.55){\large$\bigstar$}
   \put(.5,-.5){
   \psline(0,0)(2,2)(3,1)(3.5,1.5)(3,2)(2,1)(0,3)(.5,3.5)(1,3)(0,2)(1,1)}
   \psline(0,0)(2.5,2.5)
   \psline(1.5,2.5)(2.5,1.5)
   \psline(.5,.5)(1.5,-.5)
   \psline(0,0)(.5,-.5)
   \put(.5,-.5){
   \multiput(0,0)(0,1){4}{\pscircle[fillstyle=solid,fillcolor=black](0,0){.1}}
   \multiput(1,0)(0,1){4}{\pscircle[fillstyle=solid,fillcolor=black](0,0){.1}}
   \multiput(2,1)(0,1){3}{\pscircle[fillstyle=solid,fillcolor=black](0,0){.1}}
   \multiput(3,1)(0,1){2}{\pscircle[fillstyle=solid,fillcolor=black](0,0){.1}}
   \psset{linestyle=dashed}
   \psline(0,0)(0,3) \psline(1,0)(1,3) \psline(2,1)(2,3)
   \psline(3,1)(3,2)
   \psline(0,0)(1,0)\psline(0,1)(3,1)\psline(0,2)(3,2)\psline(0,3)(2,3)
   }
   {
   \psset{linestyle=dashed}
   \psline(0,0)(1.5,0)\psline(1,-.5)(1,3)\psline(0.5,1)(4,1)\psline(.5,2)(2.5,2)\psline(2,2.5)(2,.5)
   \psline(3,.5)(3,1.5)
   \psset{linestyle=solid}
   \pscircle[fillstyle=solid,fillcolor=white](0,0){.1}
   \pscircle[fillstyle=solid,fillcolor=white](0,0){.05}
   \pscircle[fillstyle=solid,fillcolor=white](1,3){.1}
   \pscircle[fillstyle=solid,fillcolor=white](1,3){.05}
   \pscircle[fillstyle=solid,fillcolor=white](1,0){.1}
   \pscircle[fillstyle=solid,fillcolor=white](1,1){.1}
   \pscircle[fillstyle=solid,fillcolor=white](1,2){.1}
   \pscircle[fillstyle=solid,fillcolor=white](2,1){.1}
   \pscircle[fillstyle=solid,fillcolor=white](2,2){.1}
   \pscircle[fillstyle=solid,fillcolor=white](3,1){.1}
   \pscircle[fillstyle=solid,fillcolor=white](4,1){.1}
   \pscircle[fillstyle=solid,fillcolor=white](4,1){.05}
   }
   \psdot(0.5,0)\multiput(0,0)(0,1){3}{\psdot(1.,0.5)}
   \multiput(-1,0)(1,0){4}{\psdot(1.5,1)}
   \multiput(-1,1)(1,0){3}{\psdot(1.5,1)}
   \multiput(-1,-1)(1,0){2}{\psdot(1.5,1)}
   \multiput(-.5,-1.5)(0,1){3}{\psdot(1.5,1)}
   \multiput(.5,-.5)(0,1){3}{\psdot(1.5,1)}
   \multiput(1.5,-.5)(0,1){2}{\psdot(1.5,1)}

  \end{pspicture}
  \caption{An example of TO-domain}
  \label{fig:2-2-7}
 \end{center}
\end{figure}
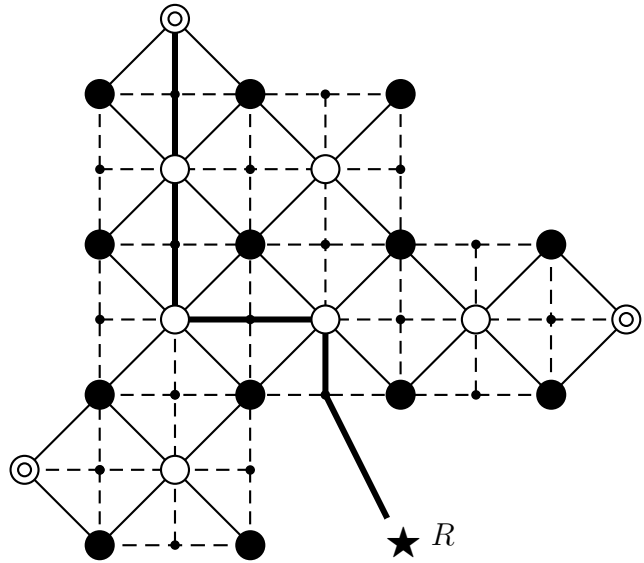

\noindent
(3)
We say that 
$A$ 
is a IO-domain iff 
(i) $A$ is not connected to terminals, and 
(ii) $A$ has boundary vertices which are connected directly to 
$R$ 
(Figure \ref{fig:2-2-8}). \\
%

%2-2-8
\begin{figure}[H]
 \begin{center}
  \begin{pspicture}(6,7)(0,-2.5)
   \psset{unit=2}
   \put(2.7,-.5){$R$}
   \psline[linewidth=0.04](1,2)(1,1)(2,1)(2,.5)(2.5,-.5)
   \pscircle[fillstyle=solid,fillcolor=white,linecolor=white](2.5,-.5){.2}
   \put(2.4,-.55){\large$\bigstar$}
   \put(.5,-.5){
   \psline(0,0)(2,2)(3,1)(3.5,1.5)(3,2)(2,1)(0,3)(.5,3.5)(1,3)(0,2)(1,1)}
   \psline(0,0)(2.5,2.5)
   \psline(1.5,2.5)(2.5,1.5)
   \psline(.5,.5)(1.5,-.5)
   \psline(0,0)(.5,-.5)
   \put(.5,-.5){
   \multiput(0,0)(0,1){4}{\pscircle[fillstyle=solid,fillcolor=black](0,0){.1}}
   \multiput(1,0)(0,1){4}{\pscircle[fillstyle=solid,fillcolor=black](0,0){.1}}
   \multiput(2,1)(0,1){3}{\pscircle[fillstyle=solid,fillcolor=black](0,0){.1}}
   \multiput(3,1)(0,1){2}{\pscircle[fillstyle=solid,fillcolor=black](0,0){.1}}
   \psset{linestyle=dashed}
   \psline(0,0)(0,3) \psline(1,0)(1,3) \psline(2,1)(2,3)
   \psline(3,1)(3,2)
   \psline(0,0)(1,0)\psline(0,1)(3,1)\psline(0,2)(3,2)\psline(0,3)(2,3)
   }
   {
   \psset{linestyle=dashed}
   \psline(0,0)(1.5,0)\psline(1,-.5)(1,3)\psline(0.5,1)(4,1)\psline(.5,2)(2.5,2)\psline(2,2.5)(2,.5)
   \psline(3,.5)(3,1.5)
   \psset{linestyle=solid}
   \pscircle[fillstyle=solid,fillcolor=white](0,0){.1}
   \pscircle[fillstyle=solid,fillcolor=white](0,0){.05}
   \pscircle[fillstyle=solid,fillcolor=white](1,3){.1}
   \pscircle[fillstyle=solid,fillcolor=white](1,3){.05}
   \pscircle[fillstyle=solid,fillcolor=white](1,0){.1}
   \pscircle[fillstyle=solid,fillcolor=white](1,1){.1}
   \pscircle[fillstyle=solid,fillcolor=white](1,2){.1}
   \pscircle[fillstyle=solid,fillcolor=white](2,1){.1}
   \pscircle[fillstyle=solid,fillcolor=white](2,2){.1}
   \pscircle[fillstyle=solid,fillcolor=white](3,1){.1}
   \pscircle[fillstyle=solid,fillcolor=white](4,1){.1}
   \pscircle[fillstyle=solid,fillcolor=white](4,1){.05}
   }
   \psdot(0.5,0)\multiput(0,0)(0,1){3}{\psdot(1.,0.5)}
   \multiput(-1,0)(1,0){4}{\psdot(1.5,1)}
   \multiput(-1,1)(1,0){3}{\psdot(1.5,1)}
   \multiput(-1,-1)(1,0){2}{\psdot(1.5,1)}
   \multiput(-.5,-1.5)(0,1){3}{\psdot(1.5,1)}
   \multiput(.5,-.5)(0,1){3}{\psdot(1.5,1)}
   \multiput(1.5,-.5)(0,1){2}{\psdot(1.5,1)}

  \end{pspicture}
  \caption{An example of IO-domain}
  \label{fig:2-2-8}
 \end{center}
\end{figure}
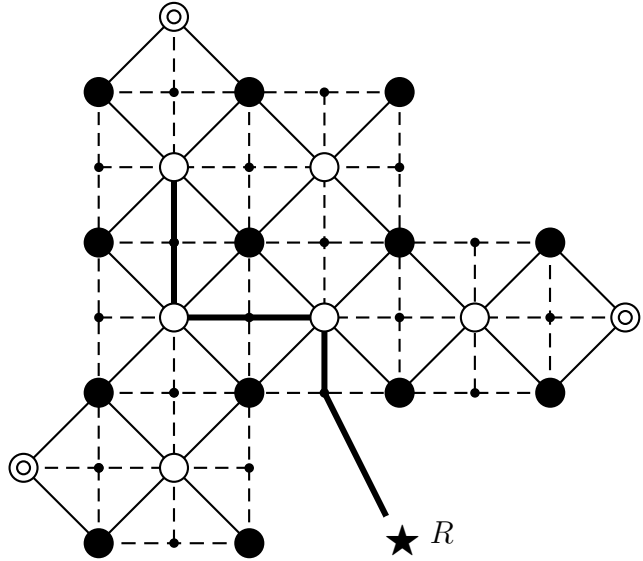

\noindent
(4)
We say that 
$A$ 
is a 
$T^l I$-domain iff 
(i) $A$ contains 
$l$-terminals 
$T_{i_1}, T_{i_2}, \cdots, T_{i_l}$, 
$i_1 < i_2 < \cdots < i_k$ 
among which only 
$T_{i_1}$ 
is connected directly to  
$R$, 
and  
(ii) $A$ has no boundary vertices. 
We regard it as a composition of a TI-domain and 
$(l-1)$ TO-domains (Figure \ref{fig:2-2-9}). 
\\
%

%2-2-9
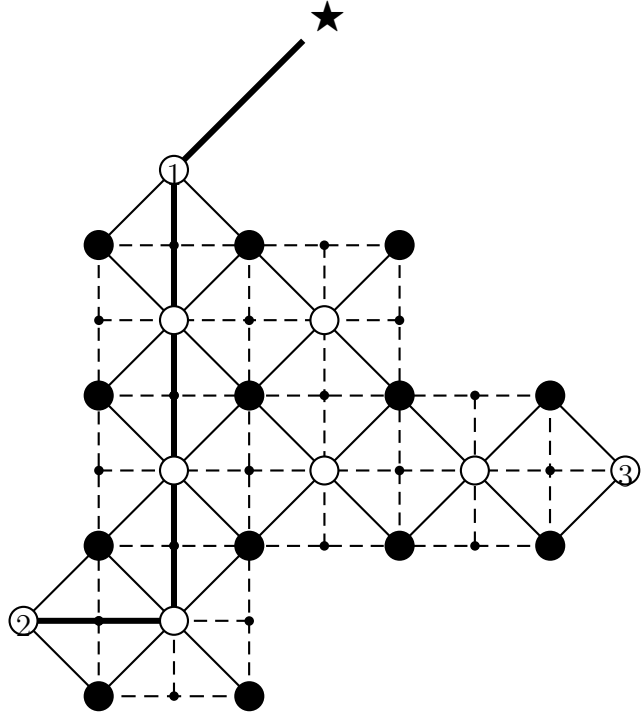
\begin{figure}[H]
 \begin{center}
  \begin{pspicture}(6,7)(0,-2.5)
   \psset{unit=2}
   \psline[linewidth=0.04](2,4)(1,3)(1,0)(0,0)
   \pscircle[fillstyle=solid,fillcolor=white,linecolor=white](2,4){.2}
   \put(1.9,3.95){\large$\bigstar$}
   \put(.5,-.5){
   \psline(0,0)(2,2)(3,1)(3.5,1.5)(3,2)(2,1)(0,3)(.5,3.5)(1,3)(0,2)(1,1)}
   \psline(0,0)(2.5,2.5)
   \psline(1.5,2.5)(2.5,1.5)
   \psline(.5,.5)(1.5,-.5)
   \psline(0,0)(.5,-.5)
   \put(.5,-.5){
   \multiput(0,0)(0,1){4}{\pscircle[fillstyle=solid,fillcolor=black](0,0){.1}}
   \multiput(1,0)(0,1){4}{\pscircle[fillstyle=solid,fillcolor=black](0,0){.1}}
   \multiput(2,1)(0,1){3}{\pscircle[fillstyle=solid,fillcolor=black](0,0){.1}}
   \multiput(3,1)(0,1){2}{\pscircle[fillstyle=solid,fillcolor=black](0,0){.1}}
   \psset{linestyle=dashed}
   \psline(0,0)(0,3) \psline(1,0)(1,3) \psline(2,1)(2,3)
   \psline(3,1)(3,2)
   \psline(0,0)(1,0)\psline(0,1)(3,1)\psline(0,2)(3,2)\psline(0,3)(2,3)
   }
   {
   \psset{linestyle=dashed}
   \psline(0,0)(1.5,0)\psline(1,-.5)(1,3)\psline(0.5,1)(4,1)\psline(.5,2)(2.5,2)\psline(2,2.5)(2,.5)
   \psline(3,.5)(3,1.5)
   \psset{linestyle=solid}
   \pscircle[fillstyle=solid,fillcolor=white](0,0){.1}
   \put(-.05,-.1){$2$}
   \pscircle[fillstyle=solid,fillcolor=white](1,3){.1}
   \put(.95,2.9){$1$}
   \pscircle[fillstyle=solid,fillcolor=white](1,0){.1}
   \pscircle[fillstyle=solid,fillcolor=white](1,1){.1}
   \pscircle[fillstyle=solid,fillcolor=white](1,2){.1}
   \pscircle[fillstyle=solid,fillcolor=white](2,1){.1}
   \pscircle[fillstyle=solid,fillcolor=white](2,2){.1}
   \pscircle[fillstyle=solid,fillcolor=white](3,1){.1}
   \pscircle[fillstyle=solid,fillcolor=white](4,1){.1}
   \put(3.95,.9){$3$}
   }
   \psdot(0.5,0)\multiput(0,0)(0,1){3}{\psdot(1.,0.5)}
   \multiput(-1,0)(1,0){4}{\psdot(1.5,1)}
   \multiput(-1,1)(1,0){3}{\psdot(1.5,1)}
   \multiput(-1,-1)(1,0){2}{\psdot(1.5,1)}
   \multiput(-.5,-1.5)(0,1){3}{\psdot(1.5,1)}
   \multiput(.5,-.5)(0,1){3}{\psdot(1.5,1)}
   \multiput(1.5,-.5)(0,1){2}{\psdot(1.5,1)}

  \end{pspicture}
  \caption{an example of $T^2I$-domain.
  Terminals with smallest index is connected directly to $R$.}
  \label{fig:2-2-9}
 \end{center}
\end{figure}

\noindent
(5) 
We say that 
$A$ 
is a 
$T^l$ O-domain iff 
(i) $A$ contains $l$ terminals 
$T_{i_1}$, $T_{i_2}$, $\cdots$, $T_{i_l}$ 
all of which are not connected directly to 
$R$, and 
(ii) 
$A$ 
has boundary vertices all of which are connected directly to $R$. 
We regard that a 
$T^l$ O-domain 
is a composition of $l$ TO-domains (Figure \ref{fig:2-2-10}). \\
%

%2-2-10
\begin{figure}[H]
 \begin{center}
  \begin{pspicture}(6,7)(0,-2.5)
   \psset{unit=2}
   \psline[linewidth=0.04](1,3)(1,0)(0,0)
   \psline[linewidth=0.04](1,1)(2,1)(2,.5)(2.5,-.5)
   \pscircle[fillstyle=solid,fillcolor=white,linecolor=white](2.5,-.5){.2}
   \put(2.4,-.55){\large$\bigstar$}
   \put(.5,-.5){
   \psline(0,0)(2,2)(3,1)(3.5,1.5)(3,2)(2,1)(0,3)(.5,3.5)(1,3)(0,2)(1,1)}
   \psline(0,0)(2.5,2.5)
   \psline(1.5,2.5)(2.5,1.5)
   \psline(.5,.5)(1.5,-.5)
   \psline(0,0)(.5,-.5)
   \put(.5,-.5){
   \multiput(0,0)(0,1){4}{\pscircle[fillstyle=solid,fillcolor=black](0,0){.1}}
   \multiput(1,0)(0,1){4}{\pscircle[fillstyle=solid,fillcolor=black](0,0){.1}}
   \multiput(2,1)(0,1){3}{\pscircle[fillstyle=solid,fillcolor=black](0,0){.1}}
   \multiput(3,1)(0,1){2}{\pscircle[fillstyle=solid,fillcolor=black](0,0){.1}}
   \psset{linestyle=dashed}
   \psline(0,0)(0,3) \psline(1,0)(1,3) \psline(2,1)(2,3)
   \psline(3,1)(3,2)
   \psline(0,0)(1,0)\psline(0,1)(3,1)\psline(0,2)(3,2)\psline(0,3)(2,3)
   }
   {
   \psset{linestyle=dashed}
   \psline(0,0)(1.5,0)\psline(1,-.5)(1,3)\psline(0.5,1)(4,1)\psline(.5,2)(2.5,2)\psline(2,2.5)(2,.5)
   \psline(3,.5)(3,1.5)
   \psset{linestyle=solid}
   \pscircle[fillstyle=solid,fillcolor=white](0,0){.1}
   \pscircle[fillstyle=solid,fillcolor=white](0,0){.05}
   \pscircle[fillstyle=solid,fillcolor=white](1,3){.1}
   \pscircle[fillstyle=solid,fillcolor=white](1,3){.05}
   \pscircle[fillstyle=solid,fillcolor=white](1,0){.1}
   \pscircle[fillstyle=solid,fillcolor=white](1,1){.1}
   \pscircle[fillstyle=solid,fillcolor=white](1,2){.1}
   \pscircle[fillstyle=solid,fillcolor=white](2,1){.1}
   \pscircle[fillstyle=solid,fillcolor=white](2,2){.1}
   \pscircle[fillstyle=solid,fillcolor=white](3,1){.1}
   \pscircle[fillstyle=solid,fillcolor=white](4,1){.1}
   \pscircle[fillstyle=solid,fillcolor=white](4,1){.05}
   }
   \psdot(0.5,0)\multiput(0,0)(0,1){3}{\psdot(1.,0.5)}
   \multiput(-1,0)(1,0){4}{\psdot(1.5,1)}
   \multiput(-1,1)(1,0){3}{\psdot(1.5,1)}
   \multiput(-1,-1)(1,0){2}{\psdot(1.5,1)}
   \multiput(-.5,-1.5)(0,1){3}{\psdot(1.5,1)}
   \multiput(.5,-.5)(0,1){3}{\psdot(1.5,1)}
   \multiput(1.5,-.5)(0,1){2}{\psdot(1.5,1)}

  \end{pspicture}
  \caption{an example of $T^2O$-domain}
  \label{fig:2-2-10}
 \end{center}
\end{figure}

%%%%%
\begin{theorem}
%{\bf (k impurity case)}\\
%
\label{k impurity case}
We have a bijection between the following two sets. 
\begin{eqnarray*}
{\cal D}(G) & := & \{ \mbox{ 
dimer coverings of $G$ } \}
\\
{\cal F}(G, Q) & := & \{ 
(T, S, \{ e_j \}_{j=1}^k )
\, | \;
T : 
\mbox{spanning tree of }
\overline{G_1},  
\;
S : 
\mbox{spanning forest of }
G_2, 
\\
&&\qquad
\{ e_j \}_{j=1}^k \subset E_2 : 
\mbox{configuration of impurities, with condition (Q)
} \}
\end{eqnarray*}
(Q) : \\
(1)
$T$ 
is composed of 
$k$ TI-trees,  
$(k-1)$ TO-trees, and the other ones are IO-trees. 
\\
\noindent
(2)
$S$ 
is composed of 
$k$ trees. 
\\
\noindent
(3)
$T$, $S$ 
are disjoint each other, and 
$k$ TI-trees of 
$T$ 
and 
$k$ 
trees of 
$S$ 
are paired by impurities. 
\end{theorem}
\begin{remark}
(1) The spanning tree 
$T$ 
of 
$\overline{G_1}$ 
uniquely determines the spannng forest 
$S$ 
of 
$G_2$ 
under the condition that they are disjoint. 
(2) In condition (Q)(1), 
$T^l I$-trees are counted as one TI-tree and 
$(l-1)$ 
TO-trees, and 
$T^l O$-trees are counted as 
$l$ TO-trees. 
(3) In each TO-trees and IO-trees,   
the boundary vertices must be unique, 
since otherwise they would not be trees.
\end{remark} 
Figure \ref{fig:2-2-11}
shows an example. \\
%

%2-2-11
\begin{figure}[H]
 \begin{center}
  \begin{pspicture}(6,7)(0,-2.5)
   \psset{unit=2}
   \psline[linewidth=.15,linecolor=gray,dotsize=.2]{*-*}(1,2)(.5,2.5)
   \psline[linewidth=.15,linecolor=gray,dotsize=.2]{*-*}(4,1)(3.5,.5)
   \psline[linewidth=0.04](0,0)(1,0)(1,1)(2,1)(2,2)(2.5,2)
   \psline[linewidth=0.04](1,2)(1,3)
   \psline[linewidth=0.04](3,1)(4,1)
   \psline(.5,.5)(.5,2.5)
   \psline(.5,1.5)(1.5,1.5)(1.5,2.5)(2.5,2.5)
   \psline(.5,-.5)(1.5,-.5)(1.5,.5)(3.5,.5)
   \psline(2.5,.5)(2.5,1.5)(3.5,1.5)
   \pscircle[fillstyle=solid,fillcolor=white,linecolor=white](2.5,-.5){.2}
   \put(2.4,-.55){\large$\bigstar$}
   \put(.5,-.5){
   \multiput(0,0)(0,1){4}{\pscircle[fillstyle=solid,fillcolor=black](0,0){.1}}
   \multiput(1,0)(0,1){4}{\pscircle[fillstyle=solid,fillcolor=black](0,0){.1}}
   \multiput(2,1)(0,1){3}{\pscircle[fillstyle=solid,fillcolor=black](0,0){.1}}
   \multiput(3,1)(0,1){2}{\pscircle[fillstyle=solid,fillcolor=black](0,0){.1}}
   \psset{linestyle=dashed}
   }
   {
   \psset{linestyle=dashed}
    \psset{linestyle=solid}
   \pscircle[fillstyle=solid,fillcolor=white](0,0){.1}
   \pscircle[fillstyle=solid,fillcolor=white](0,0){.05}
   \pscircle[fillstyle=solid,fillcolor=white](1,3){.1}
   \pscircle[fillstyle=solid,fillcolor=white](1,3){.05}
   \pscircle[fillstyle=solid,fillcolor=white](1,0){.1}
   \pscircle[fillstyle=solid,fillcolor=white](1,1){.1}
   \pscircle[fillstyle=solid,fillcolor=white](1,2){.1}
   \pscircle[fillstyle=solid,fillcolor=white](2,1){.1}
   \pscircle[fillstyle=solid,fillcolor=white](2,2){.1}
   \pscircle[fillstyle=solid,fillcolor=white](3,1){.1}
   \pscircle[fillstyle=solid,fillcolor=white](4,1){.1}
   \pscircle[fillstyle=solid,fillcolor=white](4,1){.05}
   }

  \end{pspicture}
  \caption{an example of Theorem \ref{k impurity case}. 
Thick lines are spanning tree of 
$\overline{G_1}$, and thin lines are spanning forest of 
$G_2$}
  \label{fig:2-2-11}
 \end{center}
\end{figure}
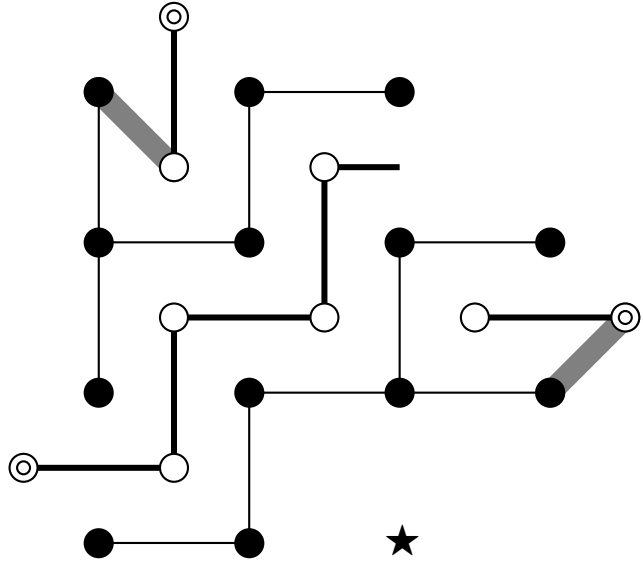

\begin{proof}
As in the proof of Theorem \ref{pairing}, 
it suffices to construct the mappings 
$T_{FD} : {\cal F}(G,Q) \to {\cal D}(G)$ 
and 
$T_{DF}=T_{FD}^{-1} : {\cal D}(G) \to {\cal F}(G,Q)$. 
For 
$T_{FD} : {\cal F}(G,Q) \to {\cal D}(G)$, 
we note that, as subgraphs of 
$G'_1$, $G'_2$,  
the numbers of vertices of a TI-tree and trees in 
$G'_2$ 
are odd, while those of a TO-tree and a IO-tree are even. 
Putting impurities, they become all even, so that 
it suffices to find (unique) dimer coverings on each trees by the argument in the proof of Proposition \ref{opposite way}. 

It then suffices to construct the mapping 
$T_{DF} : {\cal D}(G) \to {\cal F}(G,Q)$. 
Given a dimer covering on 
$G$, 
we draw slit curves as was done in Figure \ref{fig:2-1-3}, 
which divides 
$G$ 
into some domains. 
We note that vertices in 
$V_1$ 
and 
$V_2$ 
are not in the same domain so that each domains can be regarded as subgraph of either 
$G'_1$ 
or 
$G'_2$.
Moreover those domains, which are subgraphs of 
$G'_1$,  
are either  
TI-domain, TO-domain, 
$T^l I$-domain, 
$T^l O$-domain, 
or  
IO-domain, 
provided there are no loops on curves 
(if any, there would appear other sort of domains).
In what follows, 
as was already mentioned, 
we regard(count) that a 
$T^l I$-domain 
as a composition of a TI-domain and 
$(l-1)$ TO-domains, and regard a 
$T^l O$-domain as 
$l$ TO-domains. 
We then note the following facts.\\

\noindent
(i)
Because the number of vertices in 
TI-domains and those in domains in 
$G'_2$ 
are odd, they should have impurities. \\
(ii)
If 
$\sharp \{ \mbox{ TO-domains }\} = l$,
then 
$\sharp \{ \mbox{ domains in $G'_2$ } \} \ge l+1$. \\ 
(iii)
$\sharp \{ \mbox{ TI-domains }\}
+
\sharp \{ \mbox{ TO-domains } \} = 2k-1$. 
\\

\noindent
Using these facts, we proceed \\

\noindent
(a)
Since we have 
$k$ 
impurities, by (i) 
$\sharp \{ \mbox{ TI-domains } \} \le k$, 
so that by (iii) 
$\sharp \{ \mbox{ TO-domains } \} \ge k-1$.\\
\noindent
(b)
Since 
$\sharp \{ \mbox{ domains in $G'_2$ } \} \le k$, 
by (ii) we should have 
$\sharp \{ \mbox{ TO-domains } \} \le k-1$ 
so that by (iii) 
$\sharp \{ \mbox{ TI-domains } \} \ge k$. \\

By 
(a), (b), 
it follows that 
$\sharp \{ \mbox{ TI-domains } \} = k$, 
$\sharp \{ \mbox{ TO-domains }\} = k-1$. 
Each 
TO-domains has only one boundary vertex, since otherwise the number of domains in  
$G'_2$ 
is larger than 
$k$. 
Moreover, 
there are no loops in the curves, since otherwise it would  produce extra domains so that we would run out of impurities. 
Therefore 
each domains are trees, and thus we obtain the spanning forest of 
$G'_1$ 
and 
$G'_2$, 
and by ignoring middle vertices, we obtain those of 
$G_1$, $G_2$. 
Moreover, 
TI-trees and trees in 
$G_2$ 
are paired by impurities. 

We next 
connect the terminals of TI-trees directly to $R$, 
connect the boundary vertices of $T^l O$-trees directly to $R$, and connect directly to $R$ the terminal 
$T_{i_1}$ 
which has the smallest index among  
$T_{i_1},\cdots, T_{i_l}$ 
($i_1 < i_2 < \cdots < i_l$) 
of 
$T^l I$-trees. 
Then 
we obtain a spanning tree 
$T$ 
of 
$\overline{G_1}$. 
By the arguments above, 
the spanning tree $T$ 
of 
$\overline{G_1}$, 
the spanning forest 
$S$ 
of 
$G_2$ 
and the $k$ impurities satisfy the condition (Q). 
The proof of theorem \ref{k impurity case} is completed.
\QED  
\end{proof}
\begin{remark}
The boundary vertex of TO-tree must be located farther than the next terminal, 
since otherwise the pairing condition 
Q(3) would not be satisfied (Figure \ref{fig:2-2-12}). 
Hence the mapping 
$T_{DF} : {\cal D}(G) \to {\cal F}(G, Q)$ 
does not exhaust all spanning trees of 
$\overline{G_1}$. 
\end{remark}
%

%2-2-12
\begin{figure}[H]
 \begin{center}
  \begin{pspicture}(6,12)(0,-2.5)
   \psset{unit=2}

   \put(0,4){
   \put(-.5,1){$(1)$}
   \put(0.5,0){
   \psline[linestyle=dotted](0,0)(1,0)(1,1)(0,1)(0,0)
   \psline[linestyle=dotted](0,0.5)(1,0.5)
   \psline[linestyle=dotted](0.5,0)(.5,1)
   \psline[linewidth=.01](0,0)(1,1)
   \psline[linewidth=.01](1,0)(0,1)
   \pscircle[fillstyle=solid,fillcolor=black](0,0){.1}
   \pscircle[fillstyle=solid,fillcolor=black](1,0){.1}
   \pscircle[fillstyle=solid,fillcolor=black](0,1){.1}
   \pscircle[fillstyle=solid,fillcolor=black](1,1){.1}
   \psdot(.5,.5)
   }
   }

   \put(-.5,3){$(2)$}
   \psline[linewidth=0.05](0,0)(1,0)(1,1)(2,1)(2,.5)
   \psline[linewidth=0.05](1,3)(1,2)(2,2)
   \psline[linewidth=0.05](3,1)(4,1)
   \psline[linewidth=0.03](.5,.5)(.5,2.5)
   \psline[linewidth=0.03](.5,1.5)(3.5,1.5)
   \psline[linewidth=0.03](3.5,.5)(2.5,0.5)(2.5,2.5)(1.5,2.5)
   \psline[linewidth=0.03](.5,-.5)(1.5,-0.5)(1.5,.5)

   \put(.5,-.5){
   \psset{linewidth=.01}
   \psline(0,0)(2,2)(3,1)(3.5,1.5)(3,2)(2,1)(0,3)(.5,3.5)(1,3)(0,2)(1,1)}
   {
   \psset{linewidth=.01}
   \psline(0,0)(2.5,2.5)
   \psline(1.5,2.5)(2.5,1.5)
   \psline(.5,.5)(1.5,-.5)
   \psline(0,0)(.5,-.5)
   }
   \put(.5,-.5){
   \multiput(0,0)(0,1){4}{\pscircle[fillstyle=solid,fillcolor=black](0,0){.1}}
   \multiput(1,0)(0,1){4}{\pscircle[fillstyle=solid,fillcolor=black](0,0){.1}}
   \multiput(2,1)(0,1){3}{\pscircle[fillstyle=solid,fillcolor=black](0,0){.1}}
   \multiput(3,1)(0,1){2}{\pscircle[fillstyle=solid,fillcolor=black](0,0){.1}}
   \psset{linestyle=dotted}
   \psline(0,0)(0,3) \psline(1,0)(1,3) \psline(2,1)(2,3)
   \psline(3,1)(3,2)
   \psline(0,0)(1,0)\psline(0,1)(3,1)\psline(0,2)(3,2)\psline(0,3)(2,3)
   }
   {
   \psset{linestyle=dotted}
   \psline(0,0)(1.5,0)\psline(1,-.5)(1,3)\psline(0.5,1)(4,1)\psline(.5,2)(2.5,2)\psline(2,2.5)(2,.5)
   \psline(3,.5)(3,1.5)
   \psset{linestyle=solid}
   \pscircle[fillstyle=solid,fillcolor=white](0,0){.1}
   \pscircle[fillstyle=solid,fillcolor=white](0,0){.05}
   \pscircle[fillstyle=solid,fillcolor=white](1,3){.1}
   \pscircle[fillstyle=solid,fillcolor=white](1,3){.05}
   \pscircle[fillstyle=solid,fillcolor=white](1,0){.1}
   \pscircle[fillstyle=solid,fillcolor=white](1,1){.1}
   \pscircle[fillstyle=solid,fillcolor=white](1,2){.1}
   \pscircle[fillstyle=solid,fillcolor=white](2,1){.1}
   \pscircle[fillstyle=solid,fillcolor=white](2,2){.1}
   \pscircle[fillstyle=solid,fillcolor=white](3,1){.1}
   \pscircle[fillstyle=solid,fillcolor=white](4,1){.1}
   \pscircle[fillstyle=solid,fillcolor=white](4,1){.05}
   }
   \psdot(0.5,0)\multiput(0,0)(0,1){3}{\psdot(1.,0.5)}
   \multiput(-1,0)(1,0){4}{\psdot(1.5,1)}
   \multiput(-1,1)(1,0){3}{\psdot(1.5,1)}
   \multiput(-1,-1)(1,0){2}{\psdot(1.5,1)}
   \multiput(-.5,-1.5)(0,1){3}{\psdot(1.5,1)}
   \multiput(.5,-.5)(0,1){3}{\psdot(1.5,1)}
   \multiput(1.5,-.5)(0,1){2}{\psdot(1.5,1)}

  \end{pspicture}
  \caption{The boundary vertex of TO-tree must be located farther than
  the next terminal}
  \label{fig:2-2-12}
 \end{center}
\end{figure}
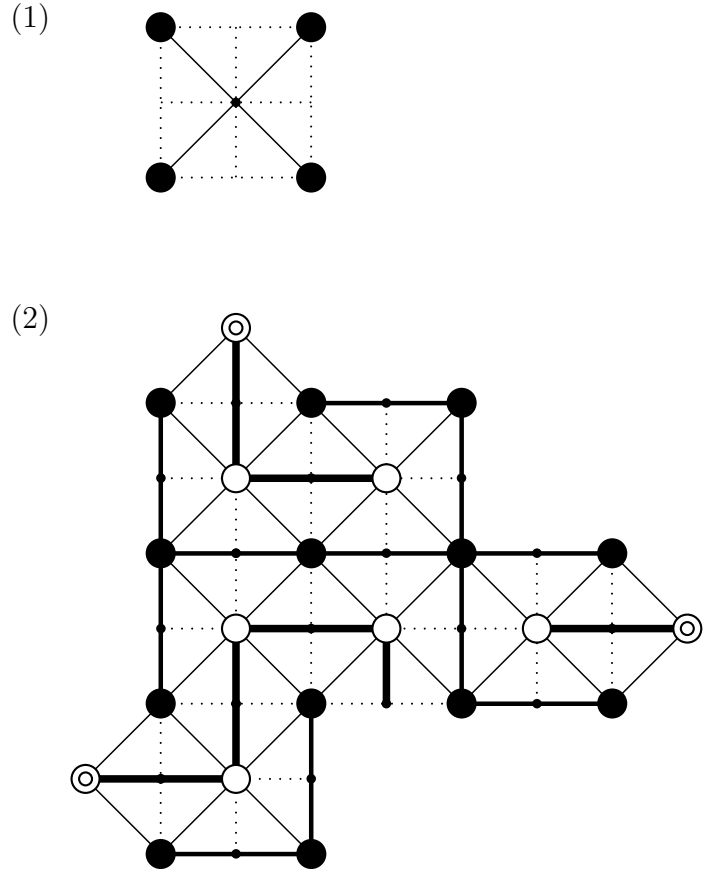

When 
$k=1$, 
we further identify terminal with the root, and then the statement of Theorem \ref{k impurity case} is simplified as follows. 
\begin{corollary}
\label{one impurity case}
We have a bijection between the following two sets. 
\begin{eqnarray*}
{\cal D}(G) & := & 
\{ \mbox{ dimer coverings in 
$G$ } \}
\\
{\cal T}(\overline{G_1}, Q') & := & 
\{ (T_1, e ) \, | \, 
T_1 : \mbox{spanning tree of $\overline{G_1}$}, 
\\
&&\qquad
e \in E_2 : 
\mbox{location of impurity, with condition (Q') } \}
\end{eqnarray*}
(Q') :  \\ 
If we regard the spanning tree of 
$\overline{G_1}$ 
as the spanning forest of 
$G_1$, 
the impurity lies on the subtree containing the terminal. 
\end{corollary}
Figure \ref{fig:2-2-13}
describes an example. \\
%

%2-2-13

\begin{figure}[H]
 \begin{center}
  \begin{pspicture}(6,6)
   \psset{unit=2}

   \psline[linewidth=.1,linecolor=gray]{*-*}(.5,.5)(1,0)
   \psline[linewidth=.01](0,0)(1.5,1.5)
   \psline[linewidth=.01](2,0)(0,2)
   \psline[linewidth=.01](1,0)(0,1)(1,2)
   \psline[linewidth=.01](1,0)(2,1)(1,2)
   \psline[linewidth=.05](0.5,0.5)(.5,1.5)(1.5,1.5)
   \psline[linewidth=.05](1.5,0.5)(1.5,0)
   \psline[linewidth=.03](2,0)(2,1)(1,1)(1,0)(0,0)(0,2)(1,2)
   \psline[linestyle=dotted](0,0)(2,0)(2,1)
   \psline[linestyle=dotted](0,0)(0,2)(1,2)
   \psline[linestyle=dotted](0.5,0)(.5,2)
   \psline[linestyle=dotted](1,0)(1,2)
   \psline[linestyle=dotted](1.5,0)(1.5,1.5)
   \psline[linestyle=dotted](0,.5)(2,.5)
   \psline[linestyle=dotted](0,1)(2,1)
   \psline[linestyle=dotted](0,1.5)(1.5,1.5)
   \multiput(0,0)(1,0){3}{
   \multiput(0,0)(0,1){2}{\pscircle[fillstyle=solid,fillcolor=black](0,0){.07}}
   }
   \pscircle[fillstyle=solid,fillcolor=black](0,2){.07}
   \pscircle[fillstyle=solid,fillcolor=black](1,2){.07}
   \multiput(0,0)(1,0){2}{
   \multiput(0,0)(0,1){2}{
   \pscircle[fillstyle=solid,fillcolor=white](0.5,0.5){.07}
   }
   }
   \pscircle[fillstyle=solid,fillcolor=white](1.5,1.5){.03}

  \end{pspicture}
  \caption{an example of Corollary \ref{one impurity case}}
  \label{fig:2-2-13}
 \end{center}
\end{figure}

%%%%
\subsection{Some features}
In this subsection
we discuss some consequences of Theorem \ref{pairing}, \ref{k impurity case}. \\
%%%%%

\noindent
(1) Naive description of the conjecture\\
The discussion above 
gives us a naive explanation of $(*)$. 
By Proposition \ref{properties of lines}(0), 
impurities lying inside of 
$G$
can always be moved to the boundary by succesive application of t-moves.
Hence 
there are as many configurations with boundary impurities as those with inner impurities. 
On the other hand, 
there are some configurations in which 
most lines live near the boundaries and do not enter inside. 
Therefore 
the number of configuration with boundary impurities 
would be much more than those with inner impurities. 
Furthermore 
the result of simulations 
(Figure \ref{fig:1-5})
implies that almost all slit curves typically lie near the boundary whereas there is a big one inside 
$G$. \\

\noindent
%%%%%
(2) Relation to the Temperley bijection\\
In our notation, 
the Temperley bijection is stated as follows. 
Consider 
$G^{(1)}$ 
as in subsection 2.2, 
eliminate edges in 
$E_2$ 
and a vertex 
$P \in V_2$ 
on the boundary, and set 
$G:= G^{(1)} \setminus \{ P \}$. 
Figure \ref{fig:2-3-1} 
gives an example. \\
%

%2-3-1

\begin{figure}[H]
 \begin{center}
  \begin{pspicture}(5,5)
   \psset{unit=2}
   \psline[linestyle=dotted](0,0)(2,0)
   \psline[linestyle=dotted](0,.5)(2,.5)
   \psline[linestyle=dotted](0,1)(2,1)
   \psline[linestyle=dotted](0,1.5)(1.5,1.5)
   \psline[linestyle=dotted](0.5,2)(1,2)
   \psline[linestyle=dotted](0,0)(0,1.5)
   \psline[linestyle=dotted](0.5,0)(0.5,2)
   \psline[linestyle=dotted](1,0)(1,2)
   \psline[linestyle=dotted](1.5,0)(1.5,1.5)
   \psline[linestyle=dotted](2,0)(2,1)
   \psline[linestyle=dotted,linecolor=lightgray](0,1.5)(0,2)(.5,2)

   \pscircle[fillstyle=solid,fillcolor=black](0,0){.07}
   \pscircle[fillstyle=solid,fillcolor=black](1,0){.07}
   \pscircle[fillstyle=solid,fillcolor=black](2,0){.07}
   \pscircle[fillstyle=solid,fillcolor=black](0,1){.07}
   \pscircle[fillstyle=solid,fillcolor=black](1,1){.07}
   \pscircle[fillstyle=solid,fillcolor=black](2,1){.07}
   \pscircle[fillstyle=solid,fillcolor=black](1,2){.07}
   \pscircle[fillstyle=solid,fillcolor=white](.5,.5){.07}
   \pscircle[fillstyle=solid,fillcolor=white](1.5,.5){.07}
   \pscircle[fillstyle=solid,fillcolor=white](.5,1.5){.07}
   \psdot(.5,0)\psdot(1.5,0)
   \psdot(0,.5)\psdot(1,.5)
   \psdot(0,.5)\psdot(2,.5)
   \psdot(0.5,1)\psdot(1.5,1)
   \psdot(0,1.5)\psdot(1,1.5)
   \psdot(0.5,2)
   \psdot[dotstyle=square,dotsize=.1](0,2)
   \put(-0.1,2.1){$P$}
   \put(1.4,1.45){\large $\bigstar$}
   
  \end{pspicture}
  \caption{an example of $G$. 
The star corresponds to the root $R$}
  \label{fig:2-3-1}
 \end{center}
\end{figure}
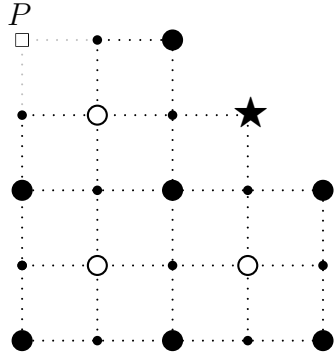

\noindent
Temperley bijection gives a bijection between the following two sets. 
\begin{eqnarray*}
{\cal D}(G) & := & \{ \mbox{ dimer covering of } G \}
\\
{\cal T}(\overline{G_1}) 
& := & 
\{ \mbox{ spanning tree of } \overline{G_1} \}
\end{eqnarray*}
This bijection is similar to that in Corollary 
\ref{one impurity case} where the impurity plays a role of  $P$. 
Figure \ref{fig:2-3-2} (1), (2) 
describes an example of the Temperley bijection which corresponds to the situation described in Figure \ref{fig:2-2-13}. \\

%2-3-2

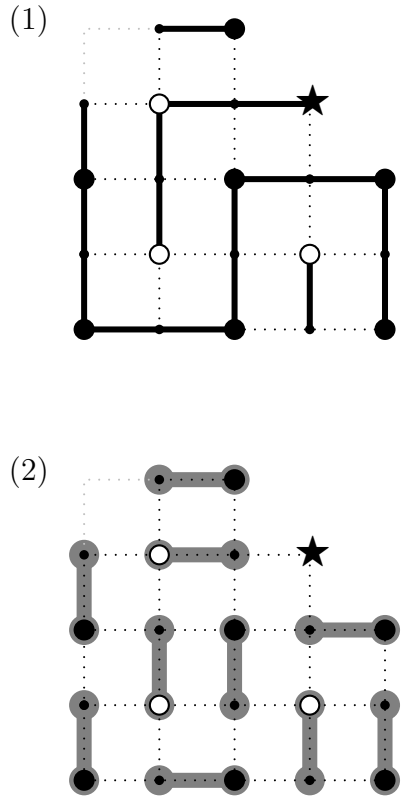
\begin{figure}[H]
 \begin{center}
  \begin{pspicture}(5,10)
   \psset{unit=2}
   \put(0,3.5){
   \put(-.5,2){$(1)$}
   \psline[linewidth=.04](0,1.5)(0,0)(1,0)(1,1)(2,1)(2,0)
   \psline[linewidth=.04](.5,.5)(0.5,1.5)(1.5,1.5)
   \psline[linewidth=.04](.5,2)(1,2)
   \psline[linewidth=.04](1.5,0)(1.5,.5)
   \psline[linestyle=dotted](0,0)(2,0)
   \psline[linestyle=dotted](0,.5)(2,.5)
   \psline[linestyle=dotted](0,1)(2,1)
   \psline[linestyle=dotted](0,1.5)(1.5,1.5)
   \psline[linestyle=dotted](0.5,2)(1,2)
   \psline[linestyle=dotted](0,0)(0,1.5)
   \psline[linestyle=dotted](0.5,0)(0.5,2)
   \psline[linestyle=dotted](1,0)(1,2)
   \psline[linestyle=dotted](1.5,0)(1.5,1.5)
   \psline[linestyle=dotted](2,0)(2,1)
   \psline[linestyle=dotted,linecolor=lightgray](0,1.5)(0,2)(.5,2)

   \pscircle[fillstyle=solid,fillcolor=black](0,0){.07}
   \pscircle[fillstyle=solid,fillcolor=black](1,0){.07}
   \pscircle[fillstyle=solid,fillcolor=black](2,0){.07}
   \pscircle[fillstyle=solid,fillcolor=black](0,1){.07}
   \pscircle[fillstyle=solid,fillcolor=black](1,1){.07}
   \pscircle[fillstyle=solid,fillcolor=black](2,1){.07}
   \pscircle[fillstyle=solid,fillcolor=black](1,2){.07}
   \pscircle[fillstyle=solid,fillcolor=white](.5,.5){.07}
   \pscircle[fillstyle=solid,fillcolor=white](1.5,.5){.07}
   \pscircle[fillstyle=solid,fillcolor=white](.5,1.5){.07}
   \psdot(.5,0)\psdot(1.5,0)
   \psdot(0,.5)\psdot(1,.5)
   \psdot(0,.5)\psdot(2,.5)
   \psdot(0.5,1)\psdot(1.5,1)
   \psdot(0,1.5)\psdot(1,1.5)
   \psdot(0.5,2)
   \put(1.4,1.45){\large $\bigstar$}
   }

   \put(0,0.5){
   \put(-.5,2){$(2)$}
   \psline[linewidth=.1,linecolor=gray,dotsize=.2]{*-*}(0,0)(0,.5)
   \psline[linewidth=.1,linecolor=gray,dotsize=.2]{*-*}(0,1)(0,1.5)
   \psline[linewidth=.1,linecolor=gray,dotsize=.2]{*-*}(0.5,0)(1,0)
   \psline[linewidth=.1,linecolor=gray,dotsize=.2]{*-*}(0.5,0.5)(0.5,1)
   \psline[linewidth=.1,linecolor=gray,dotsize=.2]{*-*}(1,0.5)(1,1)
   \psline[linewidth=.1,linecolor=gray,dotsize=.2]{*-*}(0.5,1.5)(1,1.5)
   \psline[linewidth=.1,linecolor=gray,dotsize=.2]{*-*}(0.5,2)(1,2)
   \psline[linewidth=.1,linecolor=gray,dotsize=.2]{*-*}(1.5,0)(1.5,.5)
   \psline[linewidth=.1,linecolor=gray,dotsize=.2]{*-*}(2,0)(2,.5)
   \psline[linewidth=.1,linecolor=gray,dotsize=.2]{*-*}(1.5,1)(2,1)
   \psline[linestyle=dotted](0,0)(2,0)
   \psline[linestyle=dotted](0,.5)(2,.5)
   \psline[linestyle=dotted](0,1)(2,1)
   \psline[linestyle=dotted](0,1.5)(1.5,1.5)
   \psline[linestyle=dotted](0.5,2)(1,2)
   \psline[linestyle=dotted](0,0)(0,1.5)
   \psline[linestyle=dotted](0.5,0)(0.5,2)
   \psline[linestyle=dotted](1,0)(1,2)
   \psline[linestyle=dotted](1.5,0)(1.5,1.5)
   \psline[linestyle=dotted](2,0)(2,1)
   \psline[linestyle=dotted,linecolor=lightgray](0,1.5)(0,2)(.5,2)

   \pscircle[fillstyle=solid,fillcolor=black](0,0){.07}
   \pscircle[fillstyle=solid,fillcolor=black](1,0){.07}
   \pscircle[fillstyle=solid,fillcolor=black](2,0){.07}
   \pscircle[fillstyle=solid,fillcolor=black](0,1){.07}
   \pscircle[fillstyle=solid,fillcolor=black](1,1){.07}
   \pscircle[fillstyle=solid,fillcolor=black](2,1){.07}
   \pscircle[fillstyle=solid,fillcolor=black](1,2){.07}
   \pscircle[fillstyle=solid,fillcolor=white](.5,.5){.07}
   \pscircle[fillstyle=solid,fillcolor=white](1.5,.5){.07}
   \pscircle[fillstyle=solid,fillcolor=white](.5,1.5){.07}
   \psdot(.5,0)\psdot(1.5,0)
   \psdot(0,.5)\psdot(1,.5)
   \psdot(0,.5)\psdot(2,.5)
   \psdot(0.5,1)\psdot(1.5,1)
   \psdot(0,1.5)\psdot(1,1.5)
   \psdot(0.5,2)
   \put(1.4,1.45){\large $\bigstar$}
   }

  \end{pspicture}
  \caption{(1) example of spanning tree of $\overline{G_1}$, 
(2) corresponding dimer covering of $G$}
  \label{fig:2-3-2}
 \end{center}
\end{figure}

\noindent
%%%%%
(3)
Local-move connectedness\\
We can introduce the orientation on dimers
$e \in E_1$
such that 
$o(e)\in V_1 \cup V_2$
and 
$t(e) \in V_3$, 
where 
$o(e)$
(resp. $t(e)$)
is the original of 
$e$ 
(resp. terminal of $e$). 
Similarly, 
by regarding 
the impurities as the roots of trees, 
we can introduce orientation on each subtrees consisting the spanning forest.
Then we note the following facts.
(1)
Each dimers 
are arranged along this orientation.
(Figure \ref{fig:2-3-3} (1)).
(2) 
Let 
$T_j$
($j=1,2$)
be subtrees of the spanning forest of 
$G_j$
and 
$e_j \in T_j$ 
be some neighboring dimers which are parallel each other. 
If 
$T_1$
and 
$T_2$ 
do not share the same impurity, and if 
$e_1$, $e_2$
have the opposite orientation, then s-move is possible at 
$e_1$, $e_2$
(Figure \ref{fig:2-3-3} (2)).

By moving impurities by t-move, we can adjust 
the orientation of each trees so that s-move is possible at given site with a pair of dimers belonging to different trees. 
In fact, 
we have an (not so simple) algorithm by which any dimer covering can be transformed to the specific one where all trees are parallel
(Figure \ref{fig:2-3-4}). 
These facts 
gives us another proof of the local move connectedness. \\
%
%%%%%

%2-3-3

\begin{figure}[H]
 \begin{center}
  \begin{pspicture}(8,10)(0,-1.5)
   \psset{unit=2}

   \put(0,3)
   {
   \put(-.5,2){$(1)$}
   \psline[linewidth=.1,linecolor=gray,dotsize=.2]{*-*}(3,0)(3.5,.5)
   \psline(0,1)(2,1)(2,0)(4,0)(4,2)
   \psdot(0.5,1)\psdot(1.5,1)\psdot(2,.5)\psdot(2.5,.0)\psdot(3.5,.0)
   \psdot(4,.5)\psdot(4,1.5)
   \pscircle[fillstyle=solid,fillcolor=black](0,1){.07}
   \pscircle[fillstyle=solid,fillcolor=black](1,1){.07}
   \pscircle[fillstyle=solid,fillcolor=black](2,1){.07}
   \pscircle[fillstyle=solid,fillcolor=black](2,0){.07}
   \pscircle[fillstyle=solid,fillcolor=black](3,0){.07}
   \pscircle[fillstyle=solid,fillcolor=black](4,0){.07}
   \pscircle[fillstyle=solid,fillcolor=black](4,1){.07}
   \pscircle[fillstyle=solid,fillcolor=black](4,2){.07}
   \put(4,2){\psline[arrowsize=.1]{->}(.2,0)(.2,-.5)}
   \put(4,1){\psline[arrowsize=.1]{->}(.2,0)(.2,-.5)}
   \put(4,0){\psline[arrowsize=.1]{->}(0,-0.2)(-.5,-.2)}
   \put(0,1){\psline[arrowsize=.1]{->}(0,0.2)(.5,.2)}
   \put(1,1){\psline[arrowsize=.1]{->}(0,0.2)(.5,.2)}
   \put(2,1){\psline[arrowsize=.1]{->}(0.2,0)(.2,-.5)}
   \put(2,0){\psline[arrowsize=.1]{->}(0,-0.2)(.5,-.2)}
   }

   \put(0,0)
   {
   \put(-.5,2){$(2)$}
   \psline[linewidth=.1,linecolor=gray,dotsize=.2]{*-*}(0,1)(.5,.5)
   \psline(0,1)(1,1)(1,0)(4,0)(4,1)(5,1)
   \psdot(0.5,1)\psdot(1,.5)\psdot(1.5,.0)\psdot(2.5,.0)\psdot(3.5,.0)
   \psdot(4,.5)\psdot(4.5,1)
   \pscircle[fillstyle=solid,fillcolor=black](0,1){.07}
   \pscircle[fillstyle=solid,fillcolor=black](1,1){.07}
   \pscircle[fillstyle=solid,fillcolor=black](1,0){.07}
   \pscircle[fillstyle=solid,fillcolor=black](2,0){.07}
   \pscircle[fillstyle=solid,fillcolor=black](3,0){.07}
   \pscircle[fillstyle=solid,fillcolor=black](4,0){.07}
   \pscircle[fillstyle=solid,fillcolor=black](4,1){.07}
   \put(4,1){\psline[arrowsize=.1]{->}(.2,-0.1)(.2,-.5)}
   \put(4,0){\psline[arrowsize=.1]{->}(0,-0.2)(-.5,-.2)}
   \put(3,0){\psline[arrowsize=.1]{->}(0,-0.2)(-.5,-.2)}
   \put(2,0){\psline[arrowsize=.1]{->}(0,-0.2)(-.5,-.2)}
   \put(1,0){\psline[arrowsize=.1]{->}(-0.2,0)(-.2,.5)}
   \put(1,1){\psline[arrowsize=.1]{->}(-0.1,-0.2)(-.5,-.2)}
   \put(-.5,.5){
   \psline[linewidth=.1,linecolor=gray,dotsize=.2]{*-*}(5,1)(4.5,1.5)
   \psline(0,1)(2,1)(2,0)(4,0)(4,1)(5,1)
   \psdot(0.5,1)\psdot(1.5,1)\psdot(2,.5)\psdot(2.5,.0)\psdot(3.5,.0)
   \psdot(4,.5)\psdot(4.5,1)
   \pscircle[fillstyle=solid,fillcolor=white](0,1){.07}
   \pscircle[fillstyle=solid,fillcolor=white](1,1){.07}
   \pscircle[fillstyle=solid,fillcolor=white](2,1){.07}
   \pscircle[fillstyle=solid,fillcolor=white](2,0){.07}
   \pscircle[fillstyle=solid,fillcolor=white](3,0){.07}
   \pscircle[fillstyle=solid,fillcolor=white](4,0){.07}
   \pscircle[fillstyle=solid,fillcolor=white](4,1){.07}
   \pscircle[fillstyle=solid,fillcolor=white](5,1){.07}
   \put(0,1){\psline[arrowsize=.1]{->}(0,0.2)(.5,.2)}
   \put(1,1){\psline[arrowsize=.1]{->}(0,0.2)(.5,.2)}
   \put(2,1){\psline[arrowsize=.1]{->}(0.2,0)(.2,-.5)}
   \put(2,0){\psline[arrowsize=.1]{->}(0.1,0.2)(.5,.2)}
   \put(3,0){\psline[arrowsize=.1]{->}(0.,0.2)(.5,.2)}
   \put(4,0){\psline[arrowsize=.1]{->}(-0.2,0.1)(-.2,.5)}
   \put(4,1){\psline[arrowsize=.1]{->}(0,0.2)(.5,.2)}
   \put(2.25,-0.25){\pscircle(0,0){.6}}
   }

   }

  \end{pspicture}
  \caption{(1) dimers are arranged along the orientation 
of trees, 
(2) if a pair of neighboring dimers have the opposite orientation,
  s-move is possible}
  \label{fig:2-3-3}
 \end{center}
\end{figure}
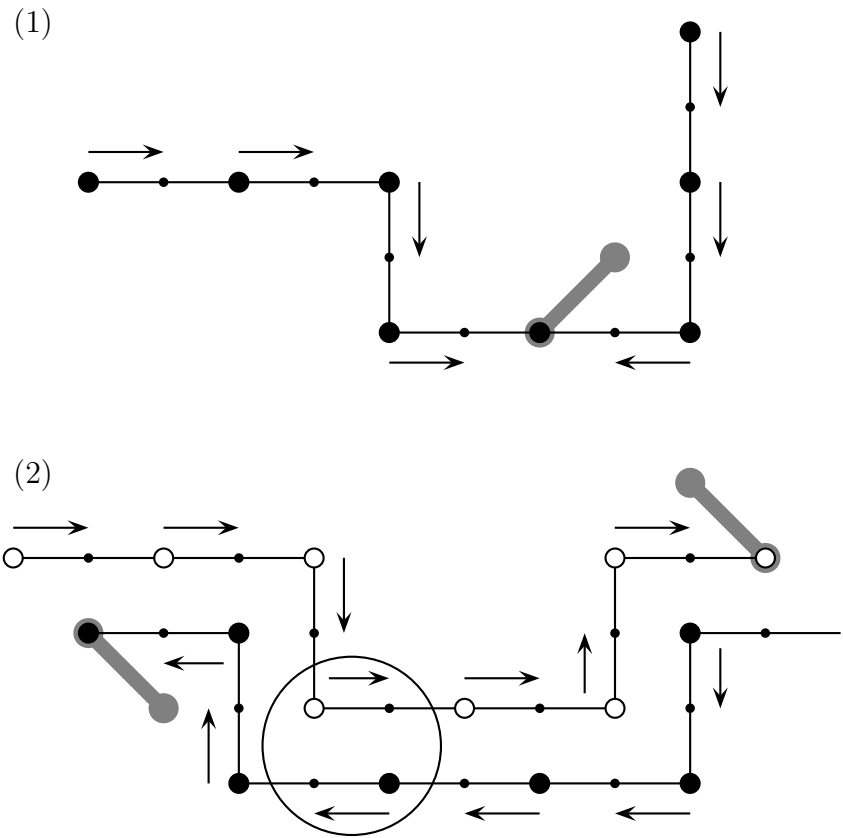

%2-3-4
\begin{figure}[H]
 \begin{center}
  \begin{pspicture}(10,12)(-1,0)
   \put(0,8){
   \put(-1,2){$(1)$}
   \pspolygon[fillstyle=solid,fillcolor=yellow](1,1)(1,2)(0,2)
   \psline[linewidth=.1,linecolor=gray]{*-*}(0,0)(1,0)
   \psline[linewidth=.1,linecolor=gray]{*-*}(1,1)(1,2)
   \psline[linewidth=.1,linecolor=gray]{*-*}(2,1)(3,1)
   \psline(0,0)(3,0)(3,2)(0,2)(0,0)
   \psline(1,0)(1,2)\psline(2,0)(2,2)
   \psline(0,1)(3,1)
   \psline(1,0)(0,1)
   \psline(2,0)(0,2)
   \psline(3,0)(2,1)(3,2)
   \psline(2,2)(1,1)
   \pscircle[fillstyle=solid,fillcolor=white](0,0){.07}
   \pscircle[fillstyle=solid,fillcolor=white](1,2){.07}
   \pscircle[fillstyle=solid,fillcolor=white](3,1){.07}
   \put(1.5,-.5){\psline[linewidth=.1,arrowsize=.3]{->}(0,0)(0,-1)}
   }

   \put(0,4){
   \put(-1,2){$(2)$}
   \pspolygon[fillstyle=solid,fillcolor=yellow,linewidth=0,linecolor=white](1,1)(.5,1.5)(0,1)(.5,.5)
   \psline[linewidth=.1,linecolor=gray]{*-*}(0,0)(1,0)
   \psline[linewidth=.1,linecolor=gray]{*-*}(0,2)(1,2)
   \psline[linewidth=.1,linecolor=gray]{*-*}(2,1)(3,1)
   \psline(0,0)(3,0)(3,2)(0,2)(0,0)
   \psline(1,0)(1,2)\psline(2,0)(2,2)
   \psline(0,1)(3,1)
   \psline(1,0)(0,1)
   \psline(2,0)(0,2)
   \psline(3,0)(2,1)(3,2)
   \psline(2,2)(1,1)
   \pscircle[fillstyle=solid,fillcolor=white](0,0){.07}
   \pscircle[fillstyle=solid,fillcolor=white](1,2){.07}
   \pscircle[fillstyle=solid,fillcolor=white](3,1){.07}
   \put(1.5,-.5){\psline[linewidth=.1,arrowsize=.3]{->}(0,0)(0,-1)}
   }

   \put(0,0){
   \put(-1,2){$(3)$}
   \pspolygon[fillstyle=solid,fillcolor=yellow,linewidth=0,linecolor=white](1,1)(.5,1.5)(0,1)(.5,.5)
   \pspolygon[fillstyle=solid,fillcolor=cyan,linewidth=0,linecolor=white](3,0)(3,1)(2,1)
   \psline[linewidth=.1,linecolor=gray]{*-*}(0,0)(1,0)
   \psline[linewidth=.1,linecolor=gray]{*-*}(0,2)(1,2)
   \psline[linewidth=.1,linecolor=gray]{*-*}(2,1)(3,1)
   \psline(0,0)(3,0)(3,2)(0,2)(0,0)
   \psline(1,0)(1,2)\psline(2,0)(2,2)
   \psline(0,1)(3,1)
   \psline(0,0)(1,1)
   \psline(0,1)(1,2)
   \psline(1,1)(2,0)
   \psline(3,0)(2,1)(3,2)
   \psline(2,2)(1,1)
   \pscircle[fillstyle=solid,fillcolor=white](0,0){.07}
   \pscircle[fillstyle=solid,fillcolor=white](1,2){.07}
   \pscircle[fillstyle=solid,fillcolor=white](3,1){.07}
   \put(4,2){\psline[linewidth=.1,arrowsize=.3]{->}(0,0)(2,4)}
   }

   \put(7,6){
   \put(-1,2){$(4)$}
   \pspolygon[fillstyle=solid,fillcolor=cyan,linewidth=0,linecolor=white](3,0)(3,1)(2,1)
   \pspolygon[fillstyle=solid,fillcolor=yellow,linewidth=0,linecolor=white](2.5,0.5)(2,0)(1.5,.5)(2,1)
   \psline[linewidth=.1,linecolor=gray]{*-*}(0,0)(1,0)
   \psline[linewidth=.1,linecolor=gray]{*-*}(0,2)(1,2)
   \psline[linewidth=.1,linecolor=gray]{*-*}(3,0)(3,1)
   \psline(0,0)(3,0)(3,2)(0,2)(0,0)
   \psline(1,0)(1,2)\psline(2,0)(2,2)
   \psline(0,1)(3,1)
   \psline(0,0)(1,1)
   \psline(0,1)(1,2)
   \psline(1,1)(2,0)
   \psline(3,0)(2,1)(3,2)
   \psline(2,2)(1,1)
   \pscircle[fillstyle=solid,fillcolor=white](0,0){.07}
   \pscircle[fillstyle=solid,fillcolor=white](1,2){.07}
   \pscircle[fillstyle=solid,fillcolor=white](3,1){.07}
   \put(1.5,-.5){\psline[linewidth=.1,arrowsize=.3]{->}(0,0)(0,-1)}
   }

   \put(7,1){
   \put(-1,2){$(5)$}
   \pspolygon[fillstyle=solid,fillcolor=yellow,linewidth=0,linecolor=white](2.5,0.5)(2,0)(1.5,.5)(2,1)
   \psline[linewidth=.1,linecolor=gray]{*-*}(0,0)(1,0)
   \psline[linewidth=.1,linecolor=gray]{*-*}(0,2)(1,2)
   \psline[linewidth=.1,linecolor=gray]{*-*}(3,0)(3,1)
   \psline(0,0)(3,0)(3,2)(0,2)(0,0)
   \psline(1,0)(1,2)\psline(2,0)(2,2)
   \psline(0,1)(3,1)
   \psline(0,0)(1,1)
   \psline(0,1)(1,2)
   \psline(1,0)(2,1)
   \psline(2,1)(3,2)
   \psline(2,0)(3,1)
   \psline(2,2)(1,1)
   \pscircle[fillstyle=solid,fillcolor=white](0,0){.07}
   \pscircle[fillstyle=solid,fillcolor=white](1,2){.07}
   \pscircle[fillstyle=solid,fillcolor=white](3,1){.07}
   }

  \end{pspicture}
  \caption{any dimer covering(1) is transformed to 
the specific one(5)}
  \label{fig:2-3-4}
 \end{center}

\end{figure}
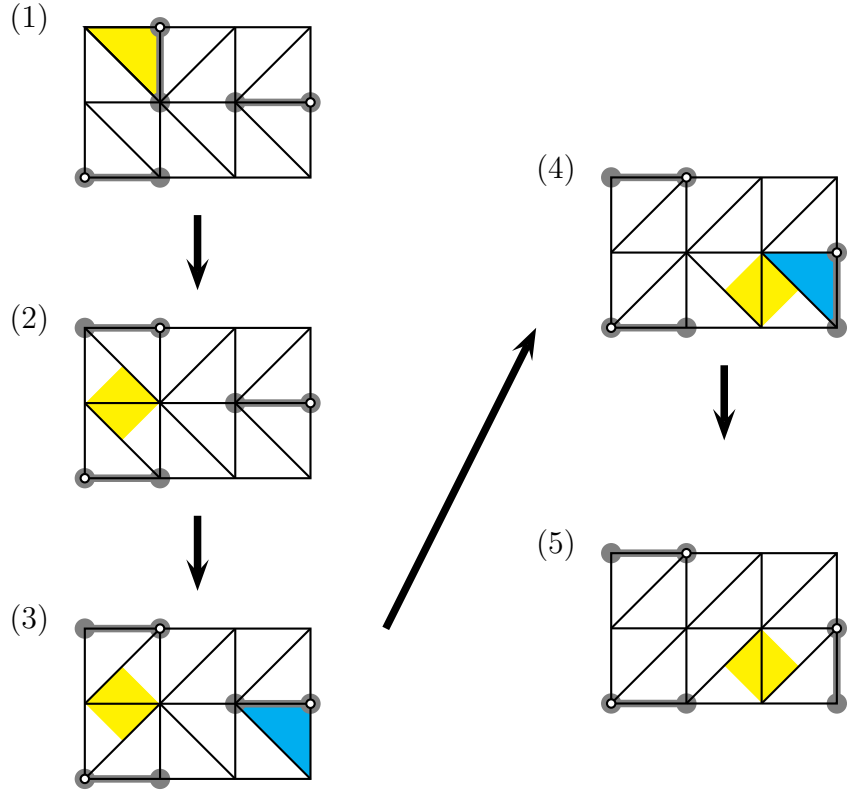

%%%%%
\subsection{Application to other graphs}
Theorem \ref{pairing}
can also be applied 
to the Bow-tie lattice and the triangular lattice. \\
%%%%%

\noindent
%%%%%
(1)
Bow-tie lattice\\
Let 
$G_B$
be the Bow-tie lattice
(Figure \ref{fig:2-4-1}(1))
which is obtained by removing vertical edges which connect vertices in 
$V_1$, $V_2$
from those in 
$G$. 
Theorem \ref{pairing}
can be directly applied, except that impurities in the vertical direction are not allowed. 
Figure \ref{fig:2-4-1}(2)
shows an example of dimer covering and the corresponding spanning forest of 
$G_1$, $G_2$. \\
%

%2-4-1
\begin{figure}[H]
 \begin{center}
  \begin{pspicture}(10,14)
   \put(0,9){
   \psset{unit=2}
   \put(-.5,3.5){$(1)$}
   \multiput(0,0)(0,1){4}{
   \psline(0,0)(4,0)
   }
   {\psset{linestyle=dotted}
   \psline(1,0)(0,1)
   \psline(2,0)(0,2)
   \psline(3,0)(0,3)
   \psline(4,0)(1,3)
   \psline(4,1)(2,3)
   \psline(4,2)(3,3)
   \psline(0,2)(1,3)
   \psline(0,1)(2,3)
   \psline(0,0)(3,3)
   \psline(1,0)(4,3)
   \psline(2,0)(4,2)
   \psline(3,0)(4,1)
   }
   \multiput(0,0)(2,0){2}{
   \multiput(0,0)(0,2){2}{
   \pscircle[fillstyle=solid,fillcolor=white](1,0){.07}
   \pscircle[fillstyle=solid,fillcolor=white](0,1){.07}
   \pscircle[fillstyle=solid,fillcolor=black](0,0){.07}
   \pscircle[fillstyle=solid,fillcolor=black](1,1){.07}
   }
   }
   \pscircle[fillstyle=solid,fillcolor=black](4,0){.07}
   \pscircle[fillstyle=solid,fillcolor=black](4,2){.07}
   \pscircle[fillstyle=solid,fillcolor=white](4,1){.07}
   \pscircle[fillstyle=solid,fillcolor=white](4,3){.07}

   \multiput(0,0)(1,0){4}{
   \multiput(0,0)(0,1){3}{
   \psdot(0.5,0.5)}
   }

   }

   \put(0,0){
   \psset{unit=2}
   \put(-.5,3.5){$(2)$}
   \psline[linewidth=.1,linecolor=gray]{*-*}(0,0)(1,0)
   \psline[linewidth=.05,linecolor=black]{*-*}(0,1)(.5,.5)
   \psline[linewidth=.05,linecolor=black]{*-*}(0,2)(.5,1.5)
   \psline[linewidth=.1,linecolor=gray]{*-*}(0,3)(1,3)
   \psline[linewidth=.05,linecolor=black]{*-*}(0.5,2.5)(1,2)
   \psline[linewidth=.1,linecolor=gray]{*-*}(3,0)(4,0)
   \psline[linewidth=.1,linecolor=gray]{*-*}(3,3)(4,3)
   \put(1,1){\psline[linewidth=.05,linecolor=black]{*-*}(0,0)(.5,-.5)}
   \put(1.5,1.5){\psline[linewidth=.05,linecolor=black]{*-*}(0,0)(.5,-.5)}
   \put(3,1){\psline[linewidth=.05,linecolor=black]{*-*}(0,0)(.5,-.5)}
   \put(3.5,1.5){\psline[linewidth=.05,linecolor=black]{*-*}(0,0)(.5,-.5)}
   \put(3.5,2.5){\psline[linewidth=.05,linecolor=black]{*-*}(0,0)(.5,-.5)}
   \put(2,0){\psline[linewidth=.05,linecolor=black]{*-*}(0,0)(.5,.5)}
   \put(2.5,1.5){\psline[linewidth=.05,linecolor=black]{*-*}(0,0)(.5,.5)}
   \put(2,2){\psline[linewidth=.05,linecolor=black]{*-*}(0,0)(.5,.5)}
   \put(1.5,2.5){\psline[linewidth=.05,linecolor=black]{*-*}(0,0)(.5,.5)}
   \multiput(0,0)(0,1){4}{
   \psline(0,0)(4,0)
   }
   {\psset{linestyle=dotted}
   \psline(1,0)(0,1)
   \psline(2,0)(0,2)
   \psline(3,0)(0,3)
   \psline(4,0)(1,3)
   \psline(4,1)(2,3)
   \psline(4,2)(3,3)
   \psline(0,2)(1,3)
   \psline(0,1)(2,3)
   \psline(0,0)(3,3)
   \psline(1,0)(4,3)
   \psline(2,0)(4,2)
   \psline(3,0)(4,1)
   }
   \multiput(0,0)(2,0){2}{
   \multiput(0,0)(0,2){2}{
   \pscircle[fillstyle=solid,fillcolor=white](1,0){.07}
   \pscircle[fillstyle=solid,fillcolor=white](0,1){.07}
   \pscircle[fillstyle=solid,fillcolor=black](0,0){.07}
   \pscircle[fillstyle=solid,fillcolor=black](1,1){.07}
   }
   }
   \pscircle[fillstyle=solid,fillcolor=black](4,0){.07}
   \pscircle[fillstyle=solid,fillcolor=black](4,2){.07}
   \pscircle[fillstyle=solid,fillcolor=white](4,1){.07}
   \pscircle[fillstyle=solid,fillcolor=white](4,3){.07}

   \multiput(0,0)(1,0){4}{
   \multiput(0,0)(0,1){3}{
   \psdot(0.5,0.5)}
   }

   }
   
  \end{pspicture}
  \caption{(1) Bow-tie lattice, 
(2) an example of dimer covering and the corresponding spanning forest}
  \label{fig:2-4-1}
 \end{center}
\end{figure}

\noindent
Since 
$V(G_B) = V(G)$
and 
$E(G_B) \subset E(G)$, 
the set 
${\cal D}(G_B)$
of the dimer coverings of 
$G_B$
satisfies 
${\cal D}(G_B) \subset {\cal D}(G)$, 
from which we deduce the following two facts. 

\noindent
(i)
Define t-move on 
$G_B$
as is descibed in 
Figure \ref{fig:2-4-2} (1), (2) \\
%

%2-4-2
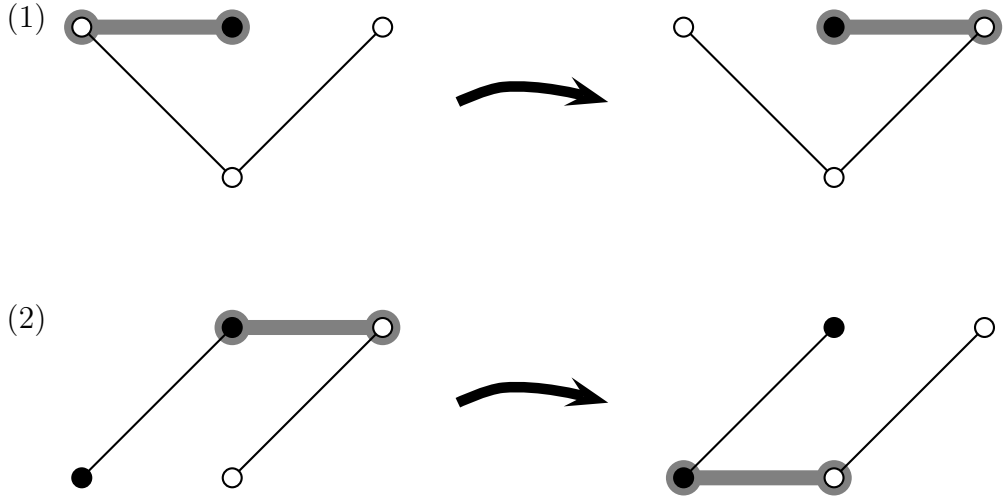
\begin{figure}[H]
 \begin{center}
  \begin{pspicture}(12,7)
   \put(0,4)
   {
   \psset{unit=2}
   \put(-.5,1){$(1)$}
   \psline[linewidth=.1,linecolor=gray]{*-*}(0,1)(1,1)
   \psline(0,1)(1,0)(2,1)
   \pscircle[fillstyle=solid,fillcolor=white](0,1){.07}
   \pscircle[fillstyle=solid,fillcolor=white](1,0){.07}
   \pscircle[fillstyle=solid,fillcolor=white](2,1){.07}
   \pscircle[fillstyle=solid,fillcolor=black](1,1){.07}
   \put(2.5,.5){\pscurve[linewidth=.07,arrowsize=.2]{->}(0,0)(.3,.1)(1,0)}

   \put(4,0){
   \psline[linewidth=.1,linecolor=gray]{*-*}(2,1)(1,1)
   \psline(0,1)(1,0)(2,1)
   \pscircle[fillstyle=solid,fillcolor=white](0,1){.07}
   \pscircle[fillstyle=solid,fillcolor=white](1,0){.07}
   \pscircle[fillstyle=solid,fillcolor=white](2,1){.07}
   \pscircle[fillstyle=solid,fillcolor=black](1,1){.07}
   }
   }

   \put(0,0)
   {
   \psset{unit=2}
   \put(-.5,1){$(2)$}
   \psline[linewidth=.1,linecolor=gray]{*-*}(2,1)(1,1)
   \psline(0,0)(1,1)\psline(1,0)(2,1)
   \pscircle[fillstyle=solid,fillcolor=white](1,0){.07}
   \pscircle[fillstyle=solid,fillcolor=white](2,1){.07}
   \pscircle[fillstyle=solid,fillcolor=black](1,1){.07}
   \pscircle[fillstyle=solid,fillcolor=black](0,0){.07}
   \put(2.5,.5){\pscurve[linewidth=.07,arrowsize=.2]{->}(0,0)(.3,.1)(1,0)}

   \put(4,0){
   \psline[linewidth=.1,linecolor=gray]{*-*}(0,0)(1,0)
   \psline(0,0)(1,1)\psline(1,0)(2,1)
   \pscircle[fillstyle=solid,fillcolor=white](1,0){.07}
   \pscircle[fillstyle=solid,fillcolor=white](2,1){.07}
   \pscircle[fillstyle=solid,fillcolor=black](1,1){.07}
   \pscircle[fillstyle=solid,fillcolor=black](0,0){.07}
   }
   }
   
  \end{pspicture}
  \caption{
  definition of t-moves for $G_B$.
Solid lines are spanning forests of $G_1$ or $G_2$
  }
  \label{fig:2-4-2}
 \end{center}
\end{figure}

\noindent
which is regarded as the composition of two t-moves for 
$G$. 
S-move
is the same as that for 
$G$.
Then we have the local move connectedness also for 
$G_B$. 
\\
(ii)
If we specify the location of impurities, 
the number of dimer coverings on 
$G$ 
and 
$G_B$
are equal. 
Hence 
a proof of conjecture 
$(*)$ 
for 
$G$ 
would also prove that for
$G_B$. \\

\noindent
(2)
Triangular lattice\\
The triangular lattice
$G_T$
is obtained by adding to 
$G_B$ 
the horizontal edges connecting vertices in 
$V_3$
(Figure \ref{fig:2-4-3}).
\\
%

%2-4-3
\begin{figure}[H]
 \begin{center}
  \begin{pspicture}(10,7)
   \psset{unit=2}
   \multiput(0,0)(0,1){4}{
   \psline(0,0)(4,0)
   }
   \multiput(0,0)(0,1){3}{
   \psline[linestyle=dashed](0,0.5)(4,0.5)
   }
   {\psset{linestyle=dotted}
   \psline(1,0)(0,1)
   \psline(2,0)(0,2)
   \psline(3,0)(0,3)
   \psline(4,0)(1,3)
   \psline(4,1)(2,3)
   \psline(4,2)(3,3)
   \psline(0,2)(1,3)
   \psline(0,1)(2,3)
   \psline(0,0)(3,3)
   \psline(1,0)(4,3)
   \psline(2,0)(4,2)
   \psline(3,0)(4,1)
   }
   \multiput(0,0)(2,0){2}{
   \multiput(0,0)(0,2){2}{
   \pscircle[fillstyle=solid,fillcolor=white](1,0){.07}
   \pscircle[fillstyle=solid,fillcolor=white](0,1){.07}
   \pscircle[fillstyle=solid,fillcolor=black](0,0){.07}
   \pscircle[fillstyle=solid,fillcolor=black](1,1){.07}
   }
   }
   \pscircle[fillstyle=solid,fillcolor=black](4,0){.07}
   \pscircle[fillstyle=solid,fillcolor=black](4,2){.07}
   \pscircle[fillstyle=solid,fillcolor=white](4,1){.07}
   \pscircle[fillstyle=solid,fillcolor=white](4,3){.07}

   \multiput(0,0)(1,0){4}{
   \multiput(0,0)(0,1){3}{
   \psdot(0.5,0.5)}
   }

  \end{pspicture}
  \caption{Triangular lattice. We have edges connecting vertices in 
$E_3$
(-~-~-)lines }
  \label{fig:2-4-3}
 \end{center}
\end{figure}
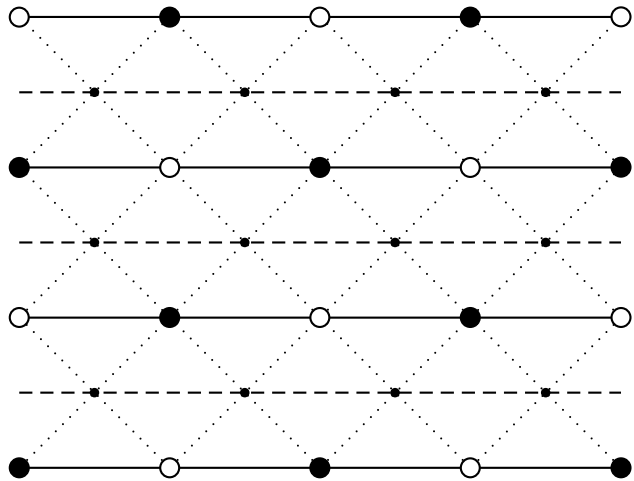

\noindent
Thus 
we can put impurities on edges belonging to 
$E_3 := \{ (x,y) \in E(G_T) \; : \; x,y \in V_3 \}$. 
If we specify the location of them, 
then 
we can cover the rest as for 
$G_B$ 
by regarding the impurities in 
$E_3$
as holes
(Figure \ref{fig:2-4-4}). 
\\
%

%2-4-4
\begin{figure}[H]
 \begin{center}
  \begin{pspicture}(10,17)
   \psset{unit=2}
   \put(0,4.5){
   \put(-.3,3.5){$(1)$}
   \psline[dotsize=.15,linecolor=gray,linewidth=.1]{*-*}(1.5,1.5)(2.5,1.5)
   \multiput(0,0)(0,1){4}{
   \psline(0,0)(4,0)
   }
   \multiput(0,0)(0,1){3}{
   \psline[linestyle=dashed](0,0.5)(4,0.5)
   }
   {\psset{linestyle=dotted}
   \psline(1,0)(0,1)
   \psline(2,0)(0,2)
   \psline(3,0)(0,3)
   \psline(4,0)(1,3)
   \psline(4,1)(2,3)
   \psline(4,2)(3,3)
   \psline(0,2)(1,3)
   \psline(0,1)(2,3)
   \psline(0,0)(3,3)
   \psline(1,0)(4,3)
   \psline(2,0)(4,2)
   \psline(3,0)(4,1)
   }
   \multiput(0,0)(2,0){2}{
   \multiput(0,0)(0,2){2}{
   \pscircle[fillstyle=solid,fillcolor=white](1,0){.07}
   \pscircle[fillstyle=solid,fillcolor=white](0,1){.07}
   \pscircle[fillstyle=solid,fillcolor=black](0,0){.07}
   \pscircle[fillstyle=solid,fillcolor=black](1,1){.07}
   }
   }
   \pscircle[fillstyle=solid,fillcolor=black](4,0){.07}
   \pscircle[fillstyle=solid,fillcolor=black](4,2){.07}
   \pscircle[fillstyle=solid,fillcolor=white](4,1){.07}
   \pscircle[fillstyle=solid,fillcolor=white](4,3){.07}

   \multiput(0,0)(1,0){4}{
   \multiput(0,0)(0,1){3}{
   \psdot(0.5,0.5)}
   }
   }

   \psline[dotsize=.15,linecolor=gray,linewidth=.1]{*-*}(1.5,1.5)(2.5,1.5)
   \psline[dotsize=.15,linecolor=lightgray,linewidth=.1]{*-*}(0,0)(1,0)
   \psline[dotsize=.15,linecolor=lightgray,linewidth=.1]{*-*}(3,0)(4,0)
   \psline[dotsize=.15,linecolor=lightgray,linewidth=.1]{*-*}(0,2)(1,2)
   \psline[dotsize=.15,linecolor=lightgray,linewidth=.1]{*-*}(1,3)(2,3)
   \psline[dotsize=.15,linecolor=lightgray,linewidth=.1]{*-*}(3,3)(4,3)
   \psline[dotsize=.15,linecolor=lightgray,linewidth=.1]{*-*}(2,0)(1.5,.5)
   \psline[dotsize=.15,linecolor=lightgray,linewidth=.1]{*-*}(2.5,0.5)(2,1)
   \put(3.5,.5){\psline[dotsize=.15,linecolor=lightgray,linewidth=.1]{*-*}(0,0)(-.5,.5)}
   \put(4,1){\psline[dotsize=.15,linecolor=lightgray,linewidth=.1]{*-*}(0,0)(-.5,.5)}
   \put(4,2){\psline[dotsize=.15,linecolor=lightgray,linewidth=.1]{*-*}(0,0)(-.5,.5)}
   \put(3,2){\psline[dotsize=.15,linecolor=lightgray,linewidth=.1]{*-*}(0,0)(-.5,.5)}
   \put(.5,2.5){\psline[dotsize=.15,linecolor=lightgray,linewidth=.1]{*-*}(0,0)(-.5,.5)}
   \put(.5,.5){\psline[dotsize=.15,linecolor=lightgray,linewidth=.1]{*-*}(0,0)(.5,.5)}
   \put(0,1){\psline[dotsize=.15,linecolor=lightgray,linewidth=.1]{*-*}(0,0)(.5,.5)}

   \psline(0,0)(1,1)(1,2)(0,1)
   \psline(1,2)(0,3)
   \psline(1,3)(2,2)
   \psline(2,3)(4,1)
   \psline(3,3)(4,2)
   \psline(3,1)(4,0)
   \psline(2,1)(3,0)
   \psline(1,1)(2,0)
   \psline(1,1)(3,1)(3,2)(1,2)
   {\psset{linestyle=dotted}
   \psline(1,0)(0,1)
   \psline(2,0)(0,2)
   \psline(3,0)(0,3)
   \psline(4,0)(1,3)
   \psline(4,1)(2,3)
   \psline(4,2)(3,3)
   \psline(0,2)(1,3)
   \psline(0,1)(2,3)
   \psline(0,0)(3,3)
   \psline(1,0)(4,3)
   \psline(2,0)(4,2)
   \psline(3,0)(4,1)
   }
   \multiput(0,0)(2,0){2}{
   \multiput(0,0)(0,2){2}{
   \pscircle[fillstyle=solid,fillcolor=white](1,0){.07}
   \pscircle[fillstyle=solid,fillcolor=white](0,1){.07}
   \pscircle[fillstyle=solid,fillcolor=black](0,0){.07}
   \pscircle[fillstyle=solid,fillcolor=black](1,1){.07}
   }
   }
   \pscircle[fillstyle=solid,fillcolor=black](4,0){.07}
   \pscircle[fillstyle=solid,fillcolor=black](4,2){.07}
   \pscircle[fillstyle=solid,fillcolor=white](4,1){.07}
   \pscircle[fillstyle=solid,fillcolor=white](4,3){.07}

   \multiput(0,0)(1,0){4}{
   \multiput(0,0)(0,1){3}{
   \psdot(0.5,0.5)}
   }
   \put(-.3,3.5){$(2)$}

  \end{pspicture}
  \caption{(1) If we specify the location of dimers in 
$E_3$, 
(2) then the rest can be covered, regarding the impurities in 
$E_3$
as holes.}
  \label{fig:2-4-4}
 \end{center}
\end{figure}

\noindent
We note that,  
if the number of impurities in 
$E_3$
is equal to 
$l$, 
then those for the rest (graph with hole) increases by 
$l$. 
For instance, 
when the size of 
$G_T$
is 
$(m, n)$, 
then the number of impurities for the rest is equal to 
$\frac {m+n+1}{2} + l$.
\\
%
%%%%%
\section{Local move connectedness}
In this section 
we show that both 
$G^{(m,n)}$ 
and 
$G^{(k)}$ 
have LMC. 
Let 
$N:=| V(G) |$ 
be the volume of 
$G$, 
and for 
$G = G^{(m,n)}$ 
let 
$k :=(m+n+1)/2 \;(=O (N^{1/2}), N \gg 1)$ 
be the number of impurities. 
\begin{theorem}
For 
$G^{(m,n)}, G^{(k)}$, 
any two dimer coverings can be transformed each other by successive application of local moves of 
$O(k^2 N^{3/2})$-steps.  
\label{LMC}
\end{theorem}
\begin{proof}
(1)
We first show LMC for 
$G:= G^{(k)}$. 
For any 
$\{ e_j \}_{j=1}^k \subset E_2(G)$, 
let 
\begin{eqnarray*}
&&{\cal D}(G ; \{ e_j \}_{j=1}^k)
:=
 \{
M \in {\cal D}(G) \, : \, 
\mbox{ impurities are on 
$\{ e_j \}_{j=1}^k$ }
\}
\\
&&{\cal D}_B (G)
:=
\{
M \in {\cal D}(G) \, : \, 
\mbox{ all impurities are on the boundary }
\}.
\end{eqnarray*}
We denote by 
$E_2(G) \cap \partial G$ 
the set of edges on the boundary of $G$. 
Then clearly we have 
\begin{eqnarray*}
{\cal D}_B (G)
&=&
\bigcup_{ \{ e_j \}_{j=1}^k \subset E_2(G) \cup \partial G}
{\cal D}(G ; \{ e_j \}_{j=1}^k).
\end{eqnarray*}
In any dimer covering, 
impurities can always be moved to the boundary by applying t-moves. 
Hence any dimer coverings 
are connected to those in  
${\cal D}_B (G)$ 
via local moves 
\footnote{
We say that 
$M, M' \in {\cal D}(G)$ 
are connected if they can be transformed via the local moves.}.
If we put 
all impurities on the boundary and fix them, say on 
$\{ e_j \}_{j=1}^k \subset E_2(G) \cap \partial G$, 
then 
our dimer problem is reduced to that of the domino tiling on 
$G \setminus \{ e_j \}_{j=1}^k$
(Figure \ref{fig:3-1}) which are connected via s-moves \cite{T, LRS}, 
in 
$O(N^{3/2})$-steps, 
implying that elements in 
${\cal D}(G ; \{ e_j \}_{j=1}^k)$ 
are connected each other for 
$\{ e_j \}_{j=1}^k \subset E_2(G) \cap \partial G$. 
It then suffices to show that an element in 
${\cal D}(G ; \{ e_j \}_{j=1}^k)$ 
is connected to some element in 
${\cal D}(G ; \{ e'_j \}_{j=1}^k)$ 
for any 
$\{ e_j \}_{j=1}^k$, $\{ e'_j \}_{j=1}^k 
\subset E_2(G) \cap \partial G$, 
unless 
${\cal D}(G ; \{ e_j \}_{j=1}^k)=\emptyset$ 
or 
${\cal D}(G ; \{ e'_j \}_{j=1}^k)=\emptyset$. \\
%

%3-1
\begin{figure}[H]
 \begin{center}
  \begin{pspicture}(6,7)(0,-2.5)
   \psset{unit=2}
   \put(1,0){\psline[dotsize=0.2,linewidth=0.15,linecolor=lightgray]{*-*}(0,0)(0,.5)}
   \put(.5,0.5){\psline[dotsize=0.2,linewidth=0.15,linecolor=lightgray]{*-*}(0,0)(0,.5)}
   \put(.5,1.5){\psline[dotsize=0.2,linewidth=0.15,linecolor=lightgray]{*-*}(0,0)(0,.5)}
   \put(1,2){\psline[dotsize=0.2,linewidth=0.15,linecolor=lightgray]{*-*}(0,0)(0,.5)}
   \put(1.5,2){\psline[dotsize=0.2,linewidth=0.15,linecolor=lightgray]{*-*}(0,0)(0,.5)}
   \put(1.5,-.5){\psline[dotsize=0.2,linewidth=0.15,linecolor=lightgray]{*-*}(0,0)(0,.5)}
   \put(2,1.){\psline[dotsize=0.2,linewidth=0.15,linecolor=lightgray]{*-*}(0,0)(0,.5)}
   \put(2.5,1.){\psline[dotsize=0.2,linewidth=0.15,linecolor=lightgray]{*-*}(0,0)(0,.5)}
   \put(0,0){\psline[dotsize=0.2,linewidth=0.15,linecolor=lightgray]{*-*}(0,0)(.5,0)}
   \put(0.5,-0.5){\psline[dotsize=0.2,linewidth=0.15,linecolor=lightgray]{*-*}(0,0)(.5,0)}
   \put(1.5,0.5){\psline[dotsize=0.2,linewidth=0.15,linecolor=lightgray]{*-*}(0,0)(.5,0)}
   \put(2.5,0.5){\psline[dotsize=0.2,linewidth=0.15,linecolor=lightgray]{*-*}(0,0)(.5,0)}
   \put(3,1){\psline[dotsize=0.2,linewidth=0.15,linecolor=lightgray]{*-*}(0,0)(.5,0)}
   \put(3.,1.5){\psline[dotsize=0.2,linewidth=0.15,linecolor=lightgray]{*-*}(0,0)(.5,0)}
   \put(1,1){\psline[dotsize=0.2,linewidth=0.15,linecolor=lightgray]{*-*}(0,0)(.5,0)}
   \put(1.,1.5){\psline[dotsize=0.2,linewidth=0.15,linecolor=lightgray]{*-*}(0,0)(.5,0)}
   \put(2,2){\psline[dotsize=0.2,linewidth=0.15,linecolor=lightgray]{*-*}(0,0)(.5,0)}
   \put(2.,2.5){\psline[dotsize=0.2,linewidth=0.15,linecolor=lightgray]{*-*}(0,0)(.5,0)}
   \put(3.5,.5){\psline[dotsize=0.2,linewidth=0.15,linecolor=lightgray]{*-*}(0,0)(.5,0.5)}
   \put(.5,2.5){\psline[dotsize=0.2,linewidth=0.15,linecolor=lightgray]{*-*}(0,0)(.5,0.5)}

   \put(0.5,-0.5){\psline[linewidth=0.04](0,0)(1,0)(1,1)(2.5,1)(2.5,1.5)(3,1.5)(3,2)(2,2)(2,3)(.5,3)(.5,2.5)(0,2.5)(0,1)(-.5,.5)(0,0)}
   \pscircle[fillstyle=solid,fillcolor=white,linecolor=white](2,4){.2}
   \put(1.9,3.95){\large$\bigstar$}
   \put(.5,-.5){
   \psline(0,0)(2,2)(3,1)(3.5,1.5)(3,2)(2,1)(0,3)(.5,3.5)(1,3)(0,2)(1,1)}
   \psline(0,0)(2.5,2.5)
   \psline(1.5,2.5)(2.5,1.5)
   \psline(.5,.5)(1.5,-.5)
   \psline(0,0)(.5,-.5)
   \put(.5,-.5){
   \multiput(0,0)(0,1){4}{\pscircle[fillstyle=solid,fillcolor=black](0,0){.1}}
   \multiput(1,0)(0,1){4}{\pscircle[fillstyle=solid,fillcolor=black](0,0){.1}}
   \multiput(2,1)(0,1){3}{\pscircle[fillstyle=solid,fillcolor=black](0,0){.1}}
   \multiput(3,1)(0,1){2}{\pscircle[fillstyle=solid,fillcolor=black](0,0){.1}}
   \psset{linestyle=dashed}
   \psline(0,0)(0,3) \psline(1,0)(1,3) \psline(2,1)(2,3)
   \psline(3,1)(3,2)
   \psline(0,0)(1,0)\psline(0,1)(3,1)\psline(0,2)(3,2)\psline(0,3)(2,3)
   }
   {
   \psset{linestyle=dashed}
   \psline(0,0)(1.5,0)\psline(1,-.5)(1,3)\psline(0.5,1)(4,1)\psline(.5,2)(2.5,2)\psline(2,2.5)(2,.5)
   \psline(3,.5)(3,1.5)
   \psset{linestyle=solid}
   \pscircle[fillstyle=solid,fillcolor=white](0,0){.1}
   \pscircle[fillstyle=solid,fillcolor=white](0,0){.05}
   \pscircle[fillstyle=solid,fillcolor=white](1,3){.1}
   \pscircle[fillstyle=solid,fillcolor=white](1,3){.05}
   \pscircle[fillstyle=solid,fillcolor=white](1,0){.1}
   \pscircle[fillstyle=solid,fillcolor=white](1,1){.1}
   \pscircle[fillstyle=solid,fillcolor=white](1,2){.1}
   \pscircle[fillstyle=solid,fillcolor=white](2,1){.1}
   \pscircle[fillstyle=solid,fillcolor=white](2,2){.1}
   \pscircle[fillstyle=solid,fillcolor=white](3,1){.1}
   \pscircle[fillstyle=solid,fillcolor=white](4,1){.1}
   \pscircle[fillstyle=solid,fillcolor=white](4,1){.05}
   }
   \psdot(0.5,0)\multiput(0,0)(0,1){3}{\psdot(1.,0.5)}
   \multiput(-1,0)(1,0){4}{\psdot(1.5,1)}
   \multiput(-1,1)(1,0){3}{\psdot(1.5,1)}
   \multiput(-1,-1)(1,0){2}{\psdot(1.5,1)}
   \multiput(-.5,-1.5)(0,1){3}{\psdot(1.5,1)}
   \multiput(.5,-.5)(0,1){3}{\psdot(1.5,1)}
   \multiput(1.5,-.5)(0,1){2}{\psdot(1.5,1)}

  \end{pspicture}
  \caption{If we fix all impurities on 
$\{ e_j \}_{j=1}^k \subset E_2(G) \cap \partial G$, 
then 
our dimer problem is reduced to that of the domino tiling on 
$G \setminus \{ e_j \}_{j=1}^k$ 
(thick lines)}
  \label{fig:3-1}
 \end{center}
\end{figure}

We recall that,  
due to the LMC for the domino tiling \cite{T, LRS}, 
all possible configurations of spanning trees of 
$\overline{G_1}$ 
are connected via s-moves, 
provided all impurities are fixed and are on the boundary. 

For any 
$M \in {\cal D}_B (G)$, 
impurities are always on the terminal each of which has two locations.
By making the TI-tree on this terminal, 
which is done in 
$O(N^{3/2})$-steps,  
these two positions can be switched (Figure \ref{fig:3-2}). \\
%

%3-2
\begin{figure}[H]
 \begin{center}
  \begin{pspicture}(6,7)(0,-2.5)
   \psset{unit=2}
   \put(.5,2.5){\psline[dotsize=0.2,linewidth=0.15,linecolor=lightgray]{*-*}(0,0)(.5,0.5)}

   {\psline[linewidth=0.06](1,2)(1,3)}
   \psline[linecolor=blue,linewidth=0.06](0,0)(1,0)(1,1)(2,1)(2,2.5)
   \psline[linecolor=blue,linewidth=0.06](3,1)(4,1)
   
   \pscurve[arrowsize=.1]{->}(.75,2.7)(.75,1.75)(1.25,1.75)(1.25,2.7)

   \pscircle[fillstyle=solid,fillcolor=white,linecolor=white](2,4){.2}
   \put(1.9,3.95){\large$\bigstar$}
   \put(.5,-.5){
   \psline(0,0)(2,2)(3,1)(3.5,1.5)(3,2)(2,1)(0,3)(.5,3.5)(1,3)(0,2)(1,1)}
   \psline(0,0)(2.5,2.5)
   \psline(1.5,2.5)(2.5,1.5)
   \psline(.5,.5)(1.5,-.5)
   \psline(0,0)(.5,-.5)
   \put(.5,-.5){
   \multiput(0,0)(0,1){4}{\pscircle[fillstyle=solid,fillcolor=black](0,0){.1}}
   \multiput(1,0)(0,1){4}{\pscircle[fillstyle=solid,fillcolor=black](0,0){.1}}
   \multiput(2,1)(0,1){3}{\pscircle[fillstyle=solid,fillcolor=black](0,0){.1}}
   \multiput(3,1)(0,1){2}{\pscircle[fillstyle=solid,fillcolor=black](0,0){.1}}
   \psset{linestyle=dashed}
   \psline(0,0)(0,3) \psline(1,0)(1,3) \psline(2,1)(2,3)
   \psline(3,1)(3,2)
   \psline(0,0)(1,0)\psline(0,1)(3,1)\psline(0,2)(3,2)\psline(0,3)(2,3)
   }
   {
   \psset{linestyle=dashed}
   \psline(0,0)(1.5,0)\psline(1,-.5)(1,3)\psline(0.5,1)(4,1)\psline(.5,2)(2.5,2)\psline(2,2.5)(2,.5)
   \psline(3,.5)(3,1.5)
   \psset{linestyle=solid}
   \pscircle[fillstyle=solid,fillcolor=white](0,0){.1}
   \pscircle[fillstyle=solid,fillcolor=white](0,0){.05}
   \pscircle[fillstyle=solid,fillcolor=white](1,3){.1}
   \pscircle[fillstyle=solid,fillcolor=white](1,3){.05}
   \pscircle[fillstyle=solid,fillcolor=white](1,0){.1}
   \pscircle[fillstyle=solid,fillcolor=white](1,1){.1}
   \pscircle[fillstyle=solid,fillcolor=white](1,2){.1}
   \pscircle[fillstyle=solid,fillcolor=white](2,1){.1}
   \pscircle[fillstyle=solid,fillcolor=white](2,2){.1}
   \pscircle[fillstyle=solid,fillcolor=white](3,1){.1}
   \pscircle[fillstyle=solid,fillcolor=white](4,1){.1}
   \pscircle[fillstyle=solid,fillcolor=white](4,1){.05}
   }
   \psdot(0.5,0)\multiput(0,0)(0,1){3}{\psdot(1.,0.5)}
   \multiput(-1,0)(1,0){4}{\psdot(1.5,1)}
   \multiput(-1,1)(1,0){3}{\psdot(1.5,1)}
   \multiput(-1,-1)(1,0){2}{\psdot(1.5,1)}
   \multiput(-.5,-1.5)(0,1){3}{\psdot(1.5,1)}
   \multiput(.5,-.5)(0,1){3}{\psdot(1.5,1)}
   \multiput(1.5,-.5)(0,1){2}{\psdot(1.5,1)}

  \end{pspicture}
  \caption{Each terminal has two positions for impurities, 
which can be switched by making a TI-tree. 
}
  \label{fig:3-2}
 \end{center}
\end{figure}

On the other hand, 
the impurity on a terminal can be moved to the one in nearest neighbor by making a $T^2$I-tree between these two terminals, provided the nearest neighbor terminal does not have impurity (Figure \ref{fig:3-3}). \\
%

%3-3
\begin{figure}[H]
 \begin{center}
  \begin{pspicture}(6,7)(0,-2.5)
   \psset{unit=2}
   \put(.5,2.5){\psline[dotsize=0.2,linewidth=0.15,linecolor=lightgray]{*-*}(0,0)(.5,0.5)}
   \put(4,1){\psline[dotsize=0.2,linewidth=0.15,linecolor=lightgray]{*-*}(0,0)(-.5,0.5)}

   \psline[linewidth=0.06](0,0)(1,0)(1,3)
   \psline[linecolor=blue,linewidth=0.06](2,2)(2,1)(4,1)
   
   \pscurve[arrowsize=.1]{->}(.75,2.7)(.8,.5)(.75,.25)(.25,.25)

   \pscircle[fillstyle=solid,fillcolor=white,linecolor=white](2,4){.2}
   \put(1.9,3.95){\large$\bigstar$}
   \put(.5,-.5){
   \psline(0,0)(2,2)(3,1)(3.5,1.5)(3,2)(2,1)(0,3)(.5,3.5)(1,3)(0,2)(1,1)}
   \psline(0,0)(2.5,2.5)
   \psline(1.5,2.5)(2.5,1.5)
   \psline(.5,.5)(1.5,-.5)
   \psline(0,0)(.5,-.5)
   \put(.5,-.5){
   \multiput(0,0)(0,1){4}{\pscircle[fillstyle=solid,fillcolor=black](0,0){.1}}
   \multiput(1,0)(0,1){4}{\pscircle[fillstyle=solid,fillcolor=black](0,0){.1}}
   \multiput(2,1)(0,1){3}{\pscircle[fillstyle=solid,fillcolor=black](0,0){.1}}
   \multiput(3,1)(0,1){2}{\pscircle[fillstyle=solid,fillcolor=black](0,0){.1}}
   \psset{linestyle=dashed}
   \psline(0,0)(0,3) \psline(1,0)(1,3) \psline(2,1)(2,3)
   \psline(3,1)(3,2)
   \psline(0,0)(1,0)\psline(0,1)(3,1)\psline(0,2)(3,2)\psline(0,3)(2,3)
   }
   {
   \psset{linestyle=dashed}
   \psline(0,0)(1.5,0)\psline(1,-.5)(1,3)\psline(0.5,1)(4,1)\psline(.5,2)(2.5,2)\psline(2,2.5)(2,.5)
   \psline(3,.5)(3,1.5)
   \psset{linestyle=solid}
   \pscircle[fillstyle=solid,fillcolor=white](0,0){.1}
   \pscircle[fillstyle=solid,fillcolor=white](0,0){.05}
   \pscircle[fillstyle=solid,fillcolor=white](1,3){.1}
   \pscircle[fillstyle=solid,fillcolor=white](1,3){.05}
   \pscircle[fillstyle=solid,fillcolor=white](1,0){.1}
   \pscircle[fillstyle=solid,fillcolor=white](1,1){.1}
   \pscircle[fillstyle=solid,fillcolor=white](1,2){.1}
   \pscircle[fillstyle=solid,fillcolor=white](2,1){.1}
   \pscircle[fillstyle=solid,fillcolor=white](2,2){.1}
   \pscircle[fillstyle=solid,fillcolor=white](3,1){.1}
   \pscircle[fillstyle=solid,fillcolor=white](4,1){.1}
   \pscircle[fillstyle=solid,fillcolor=white](4,1){.05}
   }
   \psdot(0.5,0)\multiput(0,0)(0,1){3}{\psdot(1.,0.5)}
   \multiput(-1,0)(1,0){4}{\psdot(1.5,1)}
   \multiput(-1,1)(1,0){3}{\psdot(1.5,1)}
   \multiput(-1,-1)(1,0){2}{\psdot(1.5,1)}
   \multiput(-.5,-1.5)(0,1){3}{\psdot(1.5,1)}
   \multiput(.5,-.5)(0,1){3}{\psdot(1.5,1)}
   \multiput(1.5,-.5)(0,1){2}{\psdot(1.5,1)}

  \end{pspicture}
  \caption{the impurity on a terminal can be moved to the neighbor by making a $T^2$I-tree between these two terminals.
}
  \label{fig:3-3}
 \end{center}
\end{figure}

Therefore, 
the impurities can always be moved from a terminal to the vacant one next to it, and repeating this procedure at most $O(k^2)$-steps,  
an element in 
${\cal D}(G ; \{ e_j \}_{j=1}^k)$ 
is connected to some element in 
${\cal D}(G ; \{ e'_j \}_{j=1}^k)$,  
in 
$O(k^2 N^{3/2})$-steps. 
\begin{remark}
It can happen that 
${\cal D}(G ; \{ e_j \}_{j=1}^k)=\emptyset$ 
for some 
$\{ e_j \}_{j=1}^k \subset E(G) \cap \partial G$. 
However, 
it is always possible to avoid such configuration of impurities in the argument above.
\end{remark}
(2)
We next show LMC for 
$G^{(m, n)}$.
In fact, 
it is reduced to that for 
$G^{(k)}$ 
by embedding 
$G=G^{(m,n)}$ 
to 
$G'=G^{(k)}$, 
$k = \frac {m+n+1}{2}$ 
by attaching some vertices on the boundary (Figure \ref{fig:3-4}). 
\\
%

%3-4
\begin{figure}[H]
 \begin{center}
  \begin{pspicture}(10,8)
   \psline[linestyle=dashed](0,2)(2,0)(4,2)(6,0)(10,4)(6,8)(4,6)(2,8)(0,6)(2,4)(0,2)
   \psline[linestyle=dashed](0,2)(0,0)(2,0)
   \put(4,0){\psline[linestyle=dashed](0,2)(0,0)(2,0)}
   \psline[linestyle=dashed](10,4)(10,6)(8,6)(8,8)(6,8)
   \psline[linestyle=dashed](0,6)(0,8)(2,8)
   \psline(2,2)(8,2)(8,6)(2,6)(2,2)

   \multiput(1,1)(2,0){4}{\psdot(0,0)}
   \multiput(1,1)(0,2){4}{\psdot(0,0)}
   \multiput(1,7)(2,0){4}{\psdot(0,0)}
   \psdot(9,3)\psdot(9,5)
   \multiput(0,0)(4,0){2}{\pscircle[fillstyle=solid,fillcolor=white](0,0){.15}}
   \multiput(2,2)(4,0){2}{\pscircle[fillstyle=solid,fillcolor=white](0,0){.15}}
   \multiput(4,4)(4,0){2}{\pscircle[fillstyle=solid,fillcolor=white](0,0){.15}}
   \multiput(2,6)(4,0){3}{\pscircle[fillstyle=solid,fillcolor=white](0,0){.15}}
   \multiput(0,8)(8,0){2}{\pscircle[fillstyle=solid,fillcolor=white](0,0){.15}}

   \psdots[linecolor=blue,dotstyle=+,dotsize=.3](2,0)(6,0)(0,2)(4,2)(8,2)(2,4)(6,4)(10,4)(0,6)(4,6)(8,6)(2,8)(6,8)
   \put(0,0){\pscircle(0,0){.075}}
   \put(4,0){\pscircle(0,0){.075}}
   \put(10,6){\pscircle(0,0){.075}}
   \put(0,8){\pscircle(0,0){.075}}
   \put(8,8){\pscircle(0,0){.075}}
  \end{pspicture}
  \caption{$G^{(m,n)}$ 
can be embedded to some $G^{(k)}$, $k = \frac {m+n+1}{2}$. 
Solid lines are boundary of 
$G^{(3,2)}$ 
and dotted lines are boundary of 
$G^{(3)}$.}
  \label{fig:3-4}
 \end{center}
\end{figure}
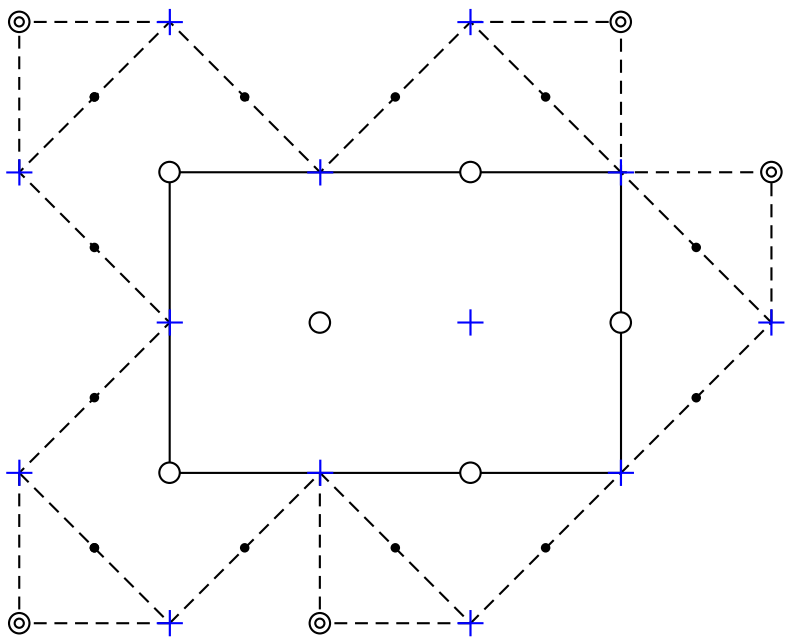

We can specify the dimer within the vertices 
$V(G') \setminus V(G)$, 
so that 
$M \in {\cal D}(G)$ 
can be regarded as 
$M' \in {\cal D}(G')$ (Figure \ref{fig:3-5} (1)). 
In particular, 
the bijection theorem for 
$G^{(k)}$ 
(Theorem \ref{k impurity case}) 
also applies for 
$G^{(m, n)}$. 
\\
%

%3-5
\begin{figure}[H]
 \begin{center}
  \begin{pspicture}(10,20)
   \put(0,11)
   {
   \put(-0.5,9){$(1)$}
   \psline[linestyle=dashed](2,2)(8,2)(8,6)(2,6)(2,2)
   \multiput(0,0)(2,0){4}{\psline[dotsize=0.3,linewidth=0.3,linecolor=lightgray]{*-*}(0,0)(1,1)}
   \multiput(0,0)(0,2){2}{\psline[dotsize=0.3,linewidth=0.3,linecolor=lightgray]{*-*}(0,0)(1,1)}
   \multiput(0,8)(0,-2){2}{\psline[dotsize=0.3,linewidth=0.3,linecolor=lightgray]{*-*}(0,0)(1,-1)}
   \multiput(0,8)(2,0){2}{\psline[dotsize=0.3,linewidth=0.3,linecolor=lightgray]{*-*}(0,0)(1,-1)}
   \multiput(8,8)(-2,0){2}{\psline[dotsize=0.3,linewidth=0.3,linecolor=lightgray]{*-*}(0,0)(-1,-1)}
   \multiput(10,6)(0,-2){2}{\psline[dotsize=0.3,linewidth=0.3,linecolor=lightgray]{*-*}(0,0)(-1,-1)}

   \psline(0,0)(2,2)
   \psline(4,0)(6,2)
   \psline(8,4)(10,6)
   \psline(6,6)(8,8)
   \psline(2,6)(0,8)
   \multiput(0,0)(4,0){2}{\pscircle[fillstyle=solid,fillcolor=white](0,0){.15}}
   \multiput(2,2)(4,0){2}{\pscircle[fillstyle=solid,fillcolor=white](0,0){.15}}
   \multiput(4,4)(4,0){2}{\pscircle[fillstyle=solid,fillcolor=white](0,0){.15}}
   \multiput(2,6)(4,0){3}{\pscircle[fillstyle=solid,fillcolor=white](0,0){.15}}
   \multiput(0,8)(8,0){2}{\pscircle[fillstyle=solid,fillcolor=white](0,0){.15}}

   \psdots[linecolor=blue,dotstyle=+,dotsize=.3](2,0)(6,0)(0,2)(4,2)(8,2)(2,4)(6,4)(10,4)(0,6)(4,6)(8,6)(2,8)(6,8)
   \put(0,0){\pscircle(0,0){.075}}
   \put(4,0){\pscircle(0,0){.075}}
   \put(10,6){\pscircle(0,0){.075}}
   \put(0,8){\pscircle(0,0){.075}}
   \put(8,8){\pscircle(0,0){.075}}

   }

   \put(-0.5,9){$(2)$}
   \psline[linestyle=dashed](2,2)(8,2)(8,6)(2,6)(2,2)
   \multiput(0,0)(4,0){2}{\psline[dotsize=0.3,linewidth=0.3,linecolor=lightgray]{*-*}(0,0)(1,1)}
   \put(0,8){\psline[dotsize=0.3,linewidth=0.3,linecolor=lightgray]{*-*}(0,0)(1,-1)}
   \put(8,8){\psline[dotsize=0.3,linewidth=0.3,linecolor=lightgray]{*-*}(0,0)(-1,-1)}
   \put(10,6){\psline[dotsize=0.3,linewidth=0.3,linecolor=lightgray]{*-*}(0,0)(-1,-1)}

   \psline(0,0)(2,2)
   \psline(4,0)(6,2)
   \psline(8,4)(10,6)
   \psline(6,6)(8,8)
   \psline(2,6)(0,8)
   \multiput(0,0)(4,0){2}{\pscircle[fillstyle=solid,fillcolor=white](0,0){.15}}
   \multiput(2,2)(4,0){2}{\pscircle[fillstyle=solid,fillcolor=white](0,0){.15}}
   \multiput(4,4)(4,0){2}{\pscircle[fillstyle=solid,fillcolor=white](0,0){.15}}
   \multiput(2,6)(4,0){3}{\pscircle[fillstyle=solid,fillcolor=white](0,0){.15}}
   \multiput(0,8)(8,0){2}{\pscircle[fillstyle=solid,fillcolor=white](0,0){.15}}

   \psdots[linecolor=blue,dotstyle=+,dotsize=.3](2,0)(6,0)(0,2)(4,2)(8,2)(2,4)(6,4)(10,4)(0,6)(4,6)(8,6)(2,8)(6,8)
   \put(0,0){\pscircle(0,0){.075}}
   \put(4,0){\pscircle(0,0){.075}}
   \put(10,6){\pscircle(0,0){.075}}
   \put(0,8){\pscircle(0,0){.075}}
   \put(8,8){\pscircle(0,0){.075}}

   \put(2,0){\psarc[linecolor=red](0,0){1}{90}{180}}
   \put(2,2){\psarc[linecolor=red](0,0){1}{270}{360}}
   \put(0,2){\psarc[linecolor=red](0,0){1}{270}{360}}
   \put(2,2){\psarc[linecolor=red](0,0){1}{90}{180}}
   \put(2,6){\psarc[linecolor=red](0,0){1}{180}{270}}
   \put(0,6){\psarc[linecolor=red](0,0){1}{0}{90}}
   \put(2,8){\psarc[linecolor=red](0,0){1}{180}{270}}
   \put(2,6){\psarc[linecolor=red](0,0){1}{0}{90}}
   \put(4,2){\psarc[linecolor=red](0,0){1}{270}{360}}
   \put(6,0){\psarc[linecolor=red](0,0){1}{90}{180}}
   \put(6,2){\psarc[linecolor=red](0,0){1}{270}{360}}
   \put(8,4){\psarc[linecolor=red](0,0){1}{270}{360}}
   \put(10,4){\psarc[linecolor=red](0,0){1}{90}{180}}
   \put(8,6){\psarc[linecolor=red](0,0){1}{270}{360}}
   \put(8,6){\psarc[linecolor=red](0,0){1}{90}{180}}
   \put(6,6){\psarc[linecolor=red](0,0){1}{90}{180}}
   \put(6,8){\psarc[linecolor=red](0,0){1}{270}{360}}
  \end{pspicture}
  \caption{(1) 
We can specify the dimer within the vertices 
$V(G') \setminus V(G)$, 
so that 
$M \in {\cal D}(G)$ 
can be regarded as 
$M' \in {\cal D}(G')$. 
(2)
The slit curve shows 
${\cal D}_B(G)$ 
and 
${\cal D}_B(G')$ 
has one to one correspondence.}
  \label{fig:3-5}
 \end{center}
\end{figure}

Moreover, 
we can also see that, 
by looking at the slit curve corresponding to 
Figure 3-4 explicitly, 
${\cal D}_B(G)$ 
and 
${\cal D}_B(G')$ 
has one to one correspondence (Figure \ref{fig:3-5} (2)).
Hence 
the same argument as in (1) also works for 
$G^{(m,n)}$, 
completing the proof of Theorem \ref{LMC}. 
\QED
\end{proof}
%

%%%%%
\section{Estimate for the number of dimer covering and  probability of impurity configuration}
In \cite{NS}, 
we studied one impurity case 
$G^{(1)}$, 
in which case the dimer covering is given by spanning tree and location of the impurity (Corollary \ref{one impurity case}). 
The uniform distribution 
on the set of spanning trees is generated by the loop erased random walk \cite{BP}, 
and 
the number of configurations of impurity is depends only on the length of the tree. 
Thus, 
by using the random walk and the matrix tree theorem, 
we obtained an explicit formula for  
$|{\cal D}(G^{(1)})|$ 
and 
the probability of finding the impurity at given edge. 
These argument 
can not be directly applied to k impurity case 
($G^{(k)}$), 
because 
the number of configurations of impurity is not simply determined by the length of the tree, and 
the mapping (in Theorem \ref{k impurity case}, from the set 
${\cal D}(G)$ 
of dimer coverings to the set of spanning tree) is not surjective. 
Thus 
we can only deduce bounds of them. 
In this section, 
we set 
$G = G^{(k)}$. 
We first 
introduce some notations. 
Let 
\begin{eqnarray*}
{\cal T}(\overline{G}_1)
& := &
\{ \mbox{ spanning tree in } \overline{G}_1 \}
\\
{\cal T}(\overline{G}_1, Q)
& := &
\{ T \in {\cal T}(\overline{G}_1) 
\, | \, 
\exists(T, S, \{e_j \}_{j=1}^k )
\in {\cal F}(G, Q) 
\}
\end{eqnarray*}
be the set of spanning trees of 
$\overline{G_1}$ 
and set of those which corresponds to the dimer covering of 
$G$ 
by the bijection given in Theorem \ref{k impurity case}. 
Moreover 
let 
$N := | V(G) |$, 
and let 
$\overline{A} = \{ \overline{a}_{ij} \}_{i, j=1, \cdots, | G_1 |+1}$
be the 
($R^{|G_1|+1} \times R^{|G_1|+1}$)-matrix given by 
\[
\overline{a}_{ij} := \cases{
\mbox{deg}(i) & ($i=j$) \cr
-1 & ($\exists e=(i,j) \in E(\overline{G_1})$) \cr
0 & (otherwise) \cr}
\]
Let 
$A = \{ a_{ij} \}_{i, j=1, \cdots, | G_1 |}$
be the restriction of 
$\overline{A}$
onto 
$G_1$
and let 
${\bf b} = ( b_j )_{j=1}^{| G_1 |}$, 
${\bf p} := ( p_j )_{j=1}^{| G_1 |}
\in {\bf R}^{| G_1 |}$
be 
\begin{eqnarray*}
b_j &:=& \cases{
1 & ($j$ \mbox{ corresponds to a terminal}) \cr
0 & (otherwise) \cr
}
\\
{\bf p} &:=& A^{-1} {\bf b}.
\end{eqnarray*}
$p_j$
is the probability that the random walk starting at 
$j$ 
hits the root 
$R$ 
through a termimal. 
For 
$j \in V_1$, let 
\[
E_2(j) := \{ e=(x,j) \in E_2 
\, | \, \exists x \in V_2 \}
\]
be the set of edges in 
$E_2$ 
with 
$j$ 
being one of endpoint. 
\begin{theorem}
%{\bf (estimate for k impurity case)}\\
\label{estimate for k impurity case}
(1)
\[
2^k |{\cal T}(\overline{G_1}, Q)|
\le
| {\cal D}(G) |
\le
| \det A |
\left(
\frac Nk
\right)^k
\]
(2)
For 
$j \in V_1$, 
the probability of having an impurity on 
$j$ 
is estimated by 
\[
{\bf P}(
\mbox{
$\exists$ impurity $\in E_2(j)$})
\le
\frac {| {\cal T}(\overline{G_1}) |
\left(
\frac Nk
\right)^k }
{2^k | {\cal T}(\overline{G_1}, Q) |}
\,
p_j.
\]
\end{theorem}
\begin{remark}
(1)
We can find a constant 
$C_k$ 
dependint only on 
$k$ 
with 
$| {\cal T}(\overline{G_1}, Q) | 
\ge 
C_k 
| {\cal T}(\overline{G_1}) | $. 
\\
(2)
Since 
$| \det A | \ne 0$, 
$A$ 
does not have zero eigenvalues : 
$d:= d(\sigma (A), 0) > 0$. 
Thus 
letting 
$d(j, \{ T_j \}_{j=1}^k)$ 
be the distance between 
$j$ 
and terminals in 
$G_1$, 
the Combes-Thomas bound implies (e.g.,\cite{Ai})
\[
p_j \le C_G e^{- \frac {d}{16} d(j, \{ T_j \}_{j=1}^k) }.
\]
Typically, 
$d \sim |G|^{-1}$,  
$C_G \sim |G|$.
\end{remark}
\begin{proof}
(1)
Let 
$f_j(T)$ 
be the number of impurity configuration of the 
$j$-th TI-tree composing 
$T \in {\cal T}(\overline{G_1}, Q)$. 
Then by Theorem \ref{k impurity case}, 
\[
F(T) :=
f_1(T)  \cdot f_2(T) \cdots  \cdot f_k(T)
=
 \prod_{j=1}^{k} f_j (T)
\]
is equal to the number of dimer covering of 
$G$ 
corresponding to 
$T \in {\cal T}(\overline{G_1}, Q)$ 
hence 
\begin{equation}
| {\cal D}(G) |
=
\sum_{T \in {\cal T}(\overline{G_1}, Q)}
F(T).
\label{FT}
\end{equation}
Since 
$2 \le f_j (T) \le (\mbox{length of the $j$-th tree})$, 
we have 
$
2^k 
\le
F(T) 
\le 
\left(
\frac Nk
\right)^k
$. 
Substituting it to 
(\ref{FT}) and using the matrix tree theorem yields Theorem \ref{estimate for k impurity case}(1). \\
(2)
If we have an impurity in 
$E_2 (j)$, 
then we find a TI-tree or a $T^l I$-tree containing the vertex 
$j$ 
by which 
$j$ 
is connected to the root. 
Let
$C$ 
be the event given by 
\[
C := \{ 
\mbox{ $j$
is connected to the root through a 
terminal }
\}.
\]
Given 
$T \in {\cal T}(\overline{G_1},Q)$, 
the number of configuration 
$F'(T)$ 
of the rest 
$(k-1)$ 
impurities is bounded from above by 
$F'(T) \le \left(
\frac Nk
\right)^k$. 
Thus 
\begin{eqnarray*}
{\bf P}(\mbox{ $\exists$ 
impurity $\in E_2(j)$ })
& \le &
\frac {| {\cal T}(\overline{G_1}) |}
{| {\cal D}(G) |}
\sum_{T \in {\cal T}(\overline{G_1}, Q)}
F'(T)
\frac {1_C}
{| {\cal T}(\overline{G_1}) |}
\\
& \le &
\frac {| {\cal T}(\overline{G_1}) |\left(
\frac Nk
\right)^k }
{| {\cal D}(G) |}
\sum_{T \in {\cal T}(\overline{G_1})}
\frac {1_C}
{| {\cal T}(\overline{G_1}) |}
\\
&=&
\frac {| {\cal T}(\overline{G_1}) |
\left(
\frac Nk
\right)^k
 }
{| {\cal D}(G) |}
\;
p_j.
\end{eqnarray*}
$1_C$ 
is the indicator function of the event 
$C$. 
It remains to substitute the lower bound for 
$| {\cal D}(G) |$ 
in 
Theorem \ref{estimate for k impurity case}(1).
\QED
\end{proof}
\begin{remark}
It is possible to apply the above argument to 
$G^{(m,n)}$, 
by putting an imginary vertex 
$R$
to 
$G_1$. 
The result is similar.
\end{remark}
%
%%%%%
\section{Appendix : some examples}
In this section, 
we apply the argument in previous sections to some discrete examples. 
\subsection{One dimensional strip}
We consider 
$G = G^{(2k, 1)}$ 
which is the strip of width 
$1$ (Figure \ref{fig:5-1}).
In this subsection 
we compute 
$D(2k) := | {\cal D}(G^{(2k, 1)}) |$ 
by using Theorem \ref{pairing}.  \\
%

%5-1
\begin{figure}[H]
 \begin{center}
  \begin{pspicture}(12,2)
   \psline(0,0)(7,0) \psline(8,0)(13,0)
   \psline(0,2)(7,2) \psline(8,2)(13,2)
   \multiput(0,0)(2,0){4}{\psline(0,0)(0,2)}
   \multiput(9,0)(2,0){3}{\psline(0,0)(0,2)}
   \multiput(0,0)(4,0){2}{\pscircle[fillcolor=white,fillstyle=solid](0,0){.1}}
   \multiput(9,0)(4,0){2}{\pscircle[fillcolor=white,fillstyle=solid](0,0){.1}}
   \multiput(2,2)(4,0){2}{\pscircle[fillcolor=white,fillstyle=solid](0,0){.1}}
   \pscircle[fillcolor=white,fillstyle=solid](11,2){.1}
   \multiput(2,0)(4,0){2}{\pscircle[fillcolor=black,fillstyle=solid](0,0){.1}}
   \multiput(11,0)(4,0){1}{\pscircle[fillcolor=black,fillstyle=solid](0,0){.1}}
   \multiput(0,2)(4,0){2}{\pscircle[fillcolor=black,fillstyle=solid](0,0){.1}}
   \multiput(9,2)(4,0){2}{\pscircle[fillcolor=black,fillstyle=solid](0,0){.1}}
  \end{pspicture}
  \caption{$G^{(2k, 1)}$}
  \label{fig:5-1}
 \end{center}
\end{figure}

$D(2) = 8$
can be seen explicitly. 
By Theorem \ref{pairing}, 
the length of tree, composing the spanning forest satisfying  the pairing condition, must be less than 
$2 \sqrt{2}$, 
and it must be less than 
$\sqrt{2}$ 
if it lies on the end (Figure \ref{fig:5-2}). 
\\
%

%5-2
\begin{figure}[H]
 \begin{center}
  \begin{pspicture}(12,2)
   \psline[dotsize=.3,linecolor=lightgray,linewidth=0.3]{*-*}(0,0)(0,2)
   \psline[dotsize=.3,linecolor=lightgray,linewidth=0.3]{*-*}(2,0)(4,0)
   \psline[dotsize=.3,linecolor=lightgray,linewidth=0.3]{*-*}(6,2)(8,2)
   \psline[dotsize=.3,linecolor=lightgray,linewidth=0.3]{*-*}(10,0)(12,0)
   \psline(0,0)(12,0)(12,2)(0,2)(0,0)
   \psline(6,2)(8,0)(10,2)
   \psline(0,0)(2,2)
   \psline[linestyle=dashed](2,0)(4,2)(6,0)
   \psline[linestyle=dashed](10,0)(12,2)
   
   \multiput(0,0)(2,0){6}{\psline(0,0)(0,2)}
   \multiput(0,0)(4,0){4}{\pscircle[fillcolor=white,fillstyle=solid](0,0){.1}}
   \multiput(2,2)(4,0){3}{\pscircle[fillcolor=white,fillstyle=solid](0,0){.1}}
   \multiput(2,0)(4,0){3}{\pscircle[fillcolor=black,fillstyle=solid](0,0){.1}}
   \multiput(0,2)(4,0){4}{\pscircle[fillcolor=black,fillstyle=solid](0,0){.1}}
  \end{pspicture}
  \caption{the length of tree must be less than 
$2 \sqrt{2}$ (inside)
and 
$\sqrt{2}$ (end)}
  \label{fig:5-2}
 \end{center}
\end{figure}
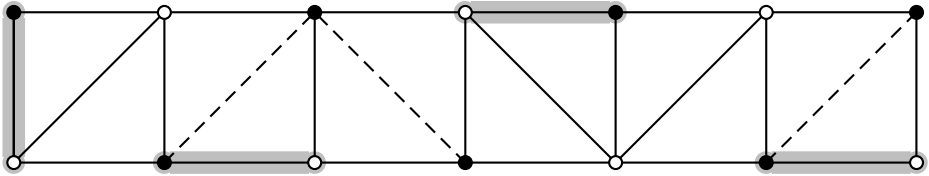

Hence 
if we put the block of two unit cubes 
($G^{(2,1)}$) 
to the left end of 
$G^{(2k,1)}$, 
the tree configuration must be one of the two shown in Figure \ref{fig:5-3}.
\\
%

%5-3
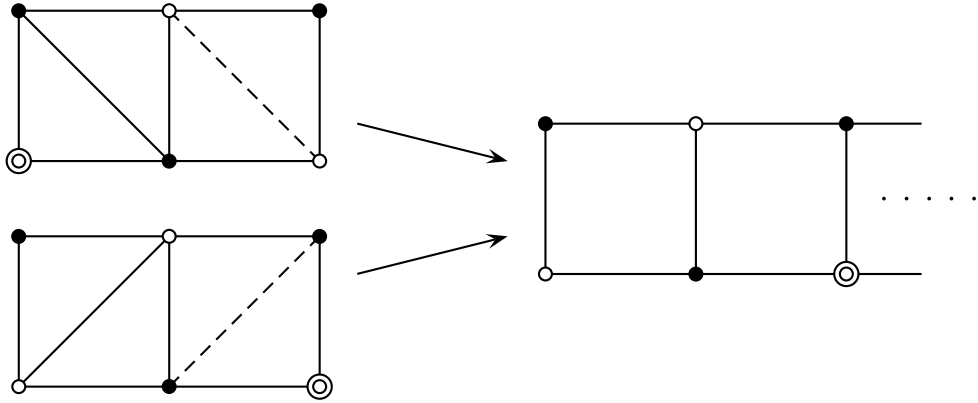
\begin{figure}[H]
 \begin{center}
  \begin{pspicture}(12,5)
   \psline(0,0)(4,0)(4,2)(0,2)(0,0)
   \psline(2,0)(2,2)
   \psline(0,0)(2,2)
   \psline[linestyle=dashed](2,0)(4,2)
   \pscircle[fillstyle=solid,fillcolor=white](0,0){.1}
   \pscircle[fillstyle=solid,fillcolor=white](2,2){.1}
   \pscircle[fillstyle=solid,fillcolor=white](4,0){.175}
   \pscircle[fillstyle=solid,fillcolor=white](4,0){.1}
   \psdots[dotsize=.2](0,2)(2,0)(4,2)
   
   \put(0,3){
   \psline(0,0)(4,0)(4,2)(0,2)(0,0)
   \psline(2,0)(2,2)
   \psline(0,2)(2,0)
   \psline[linestyle=dashed](2,2)(4,0)
   \pscircle[fillstyle=solid,fillcolor=white](0,0){.175}
   \pscircle[fillstyle=solid,fillcolor=white](0,0){.1}
   \pscircle[fillstyle=solid,fillcolor=white](2,2){.1}
   \pscircle[fillstyle=solid,fillcolor=white](4,0){.1}
   \psdots[dotsize=.2](0,2)(2,0)(4,2)
   }

   \put(7,1.5){
   \psline(0,0)(5,0)
   \psline(5,2)(0,2)(0,0)
   \psline(4,0)(4,2)
   \psline(2,0)(2,2)
   \pscircle[fillstyle=solid,fillcolor=white](0,0){.1}
   \pscircle[fillstyle=solid,fillcolor=white](2,2){.1}
   \pscircle[fillstyle=solid,fillcolor=white](4,0){.175}
   \pscircle[fillstyle=solid,fillcolor=white](4,0){.1}
   \psdots[dotsize=.2](0,2)(2,0)(4,2)
   \multiput(4.5,1)(0.3,0){5}{\psdot[dotsize=.05](0,0)}
   }

   \psline[arrowsize=.2]{->}(4.5,1.5)(6.5,2.)
   \psline[arrowsize=.2]{->}(4.5,3.5)(6.5,3.)
  \end{pspicture}
  \caption{The tree configuration of 
$G^{(2k+2, 1)}$ 
is given by putting one of two $G^{(2,1)}$'s to the left end of 
$G^{(2k, 1)}$. }
  \label{fig:5-3}
 \end{center}
\end{figure}

Therefore, 
taking possible configuraton of impurities into consideration, we have 
$D(2k+2)=2 \cdot 2 \cdot \frac 32 \cdot D(2k) = 6 D(2k)$ 
so that 
\[
D(2k) = 8 \cdot 6^{k-1}.
\]
It is also possible to deduce the same result by the transfer-matrix approach \cite{NS2}. 
\subsection{Rectangle}
We consider 
$G^{(1)}$ 
when it is the 
$m \times n$ 
rectangle (Figure \ref{fig:5-4}).
\\

%5-4
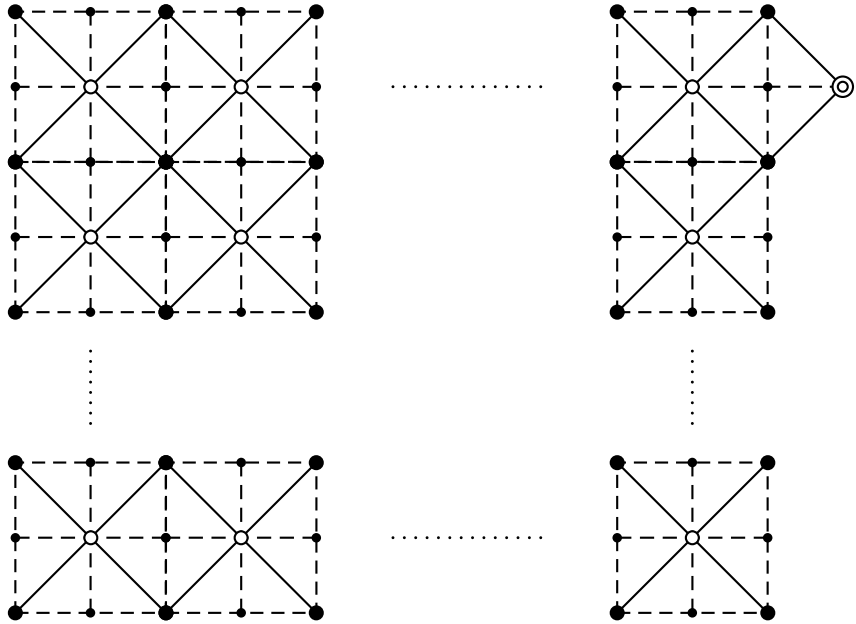
\begin{figure}[H]
 \begin{center}
  \begin{pspicture}(11,8)(0,0)
   \psset{unit=2}
   \multiput(0,0)(1,0){2}{
   \psline[linestyle=dashed](0,0)(1,0)(1,1)(0,1)(0,0)
   \psline[linestyle=dashed](.5,0)(.5,1)
   \psline[linestyle=dashed](0,.5)(1,.5)
   \psline(0,0)(1,1)\psline(0,1)(1,0)
   \put(.5,.5){\pscircle[fillstyle=solid,fillcolor=white](0,0){.05}}
   \put(0,0){\pscircle[fillstyle=solid,fillcolor=black](0,0){.05}}
   \put(1,0){\pscircle[fillstyle=solid,fillcolor=black](0,0){.05}}
   \put(1,1){\pscircle[fillstyle=solid,fillcolor=black](0,0){.05}}
   \put(0,1){\pscircle[fillstyle=solid,fillcolor=black](0,0){.05}}
   \psdots(0.5,0)(1,0.5)(0.5,1)(0,0.5)
   }
   \put(4,0){
   \psline[linestyle=dashed](0,0)(1,0)(1,1)(0,1)(0,0)
   \psline[linestyle=dashed](.5,0)(.5,1)
   \psline[linestyle=dashed](0,.5)(1,.5)
   \psline(0,0)(1,1)\psline(0,1)(1,0)
   \put(.5,.5){\pscircle[fillstyle=solid,fillcolor=white](0,0){.05}}
   \put(0,0){\pscircle[fillstyle=solid,fillcolor=black](0,0){.05}}
   \put(1,0){\pscircle[fillstyle=solid,fillcolor=black](0,0){.05}}
   \put(1,1){\pscircle[fillstyle=solid,fillcolor=black](0,0){.05}}
   \put(0,1){\pscircle[fillstyle=solid,fillcolor=black](0,0){.05}}
   \psdots(0.5,0)(1,0.5)(0.5,1)(0,0.5)
   }
   \multiput(0,2)(0,1){2}{
   \multiput(0,0)(1,0){2}{
   \psline[linestyle=dashed](0,0)(1,0)(1,1)(0,1)(0,0)
   \psline[linestyle=dashed](.5,0)(.5,1)
   \psline[linestyle=dashed](0,.5)(1,.5)
   \psline(0,0)(1,1)\psline(0,1)(1,0)
   \put(.5,.5){\pscircle[fillstyle=solid,fillcolor=white](0,0){.05}}
   \put(0,0){\pscircle[fillstyle=solid,fillcolor=black](0,0){.05}}
   \put(1,0){\pscircle[fillstyle=solid,fillcolor=black](0,0){.05}}
   \put(1,1){\pscircle[fillstyle=solid,fillcolor=black](0,0){.05}}
   \put(0,1){\pscircle[fillstyle=solid,fillcolor=black](0,0){.05}}
   \psdots(0.5,0)(1,0.5)(0.5,1)(0,0.5)
   }
   }

   \multiput(4,2)(0,1){2}{
   \psline[linestyle=dashed](0,0)(1,0)(1,1)(0,1)(0,0)
   \psline[linestyle=dashed](.5,0)(.5,1)
   \psline[linestyle=dashed](0,.5)(1,.5)
   \psline(0,0)(1,1)\psline(0,1)(1,0)
   \put(.5,.5){\pscircle[fillstyle=solid,fillcolor=white](0,0){.05}}
   \put(0,0){\pscircle[fillstyle=solid,fillcolor=black](0,0){.05}}
   \put(1,0){\pscircle[fillstyle=solid,fillcolor=black](0,0){.05}}
   \put(1,1){\pscircle[fillstyle=solid,fillcolor=black](0,0){.05}}
   \put(0,1){\pscircle[fillstyle=solid,fillcolor=black](0,0){.05}}
   \psdots(0.5,0)(1,0.5)(0.5,1)(0,0.5)
   }

   \psline(5,4)(5.5,3.5)(5,3)
   \psline[linestyle=dashed](5.5,3.5)(5,3.5)
   \pscircle[fillstyle=solid,fillcolor=white](5.5,3.5){.075}
   \pscircle[fillstyle=solid,fillcolor=white](5.5,3.5){.04}

   \psline[linestyle=dotted,linewidth=0.02](2.5,.5)(3.5,.5)
   \psline[linestyle=dotted,linewidth=0.02](.5,1.25)(.5,1.75)
   \psline[linestyle=dotted,linewidth=0.02](2.5,3.5)(3.5,3.5)
   \psline[linestyle=dotted,linewidth=0.02](4.5,1.25)(4.5,1.75)
  \end{pspicture}
  \caption{$m \times n$-rectangle}
  \label{fig:5-4}
 \end{center}
\end{figure}

We embed 
$G$ 
to 
${\bf R}^2$ 
so that 
$V_1:=V(G) \cap {\cal V}_1
=
\{ (x,y) \in {\bf Z}^2 
\, : \, 
x=1, 2, \cdots,n, 
\;
y=1, 2, \cdots, m\}$ 
and the coordinate of the root is 
$(n+1, m)$. 
Then the result in \cite{NS} tells us that,
\begin{eqnarray}
&&| {\cal D}(G) | 
=
| \det A | ( 4 \sum_{r \in V_1} p(r) + 2)
\label{rectangle1}
\\
&&{\bf P}(
\mbox{ $\exists$ impurity in $E_2(r)$ })
=
\frac {p(r)}{4 \sum_{r' \in V_1}p(r')+2}, 
\quad
r \in V_1
\label{rectangle2}
\end{eqnarray}
where 
\begin{eqnarray*}
A &:=& I - K, 
\quad
K 
\mbox{ : the incidence matrix of }
V_1, 
\end{eqnarray*}
is the 
${\bf R}^{|V_1|} \times {\bf R}^{|V_1|}$-matrix and 
\begin{eqnarray*}
p(r) &=& \langle {\bf e}_r,  A^{-1} {\bf b} \rangle, 
\quad
r=(r_x, r_y) \in V_1, 
\\
{\bf e}_r &=& \{e_r(x,y)\},
\quad 
{\bf b}=\{b(x,y)\} \in {\bf R}^{| V_1 |}, 
\\
e_r(x,y) &:=& \cases{
1 & $((x,y)=(r_x, r_y))$ \cr
0 & (otherwise) \cr}, 
\quad
b(x,y) := \cases{
1 & $((x,y)=(m,n))$ \cr
0 & (otherwise) \cr}
\end{eqnarray*}
In this case, 
we can diagonalize 
$A$ 
explicitly and its eigenvalues
$\{ e_{kl} \}$ 
and eigenvectors 
$\{ \phi_{kl} \}$ 
are given by
\begin{eqnarray*}
\phi_{kl}(x,y)
&=&
\frac {1}N_{kl} 
\sin(p_{k}x) \cdot \sin(q_{l}y), 
\quad
(x,y) \in V_1, 
\\
e_{kl}
&=&
2 \cos p_{k} + 2 \cos q_{l} + 4,
\quad
k=1, 2, \cdots, n, 
\;
l=1, 2, \cdots, m, 
\end{eqnarray*}
where 
$p_{k} = \frac {k \pi}{n+1}$,
$q_{l} = \frac {l \pi}{m+1}$ 
and 
$N_{kl} := \| \phi_{kl} \|_2$
is the normalization constant. 
$\det A$ 
and 
$p({\bf r})$
are thus equal to 
\begin{eqnarray*}
\det A &=& \prod_{kl} e_{kl}, 
\\
p(r) &=&
\sum_{k,l} 
\frac {1}{e_{kl}}
\phi_{kl}(r_x, r_y) \phi_{kl}(n,m),
\end{eqnarray*}
and we only have to substitute them to 
(\ref{rectangle1}), (\ref{rectangle2}). 
The free energy is 
\[
F := 
\lim_{n, m \to \infty}
\frac {1}{nm} \log | {\cal D}(G^{(1)}) |
=
\frac {1}{\pi^2}
\int \int_{[0,\pi]^2}
\log (4 + 2 \cos x + 2 \cos y) dx dy.
\]
which defer from the case of domino tiling, 
implying that the contribution of impurity is not negligible. 
\\

{\bf Acknowledgement}
The authors would like to thank 
Yusuke Higuchi and Tomoyuki Shirai for discussions and comments. 
%

%%%%% REFERENCES %%%%%%%%%%%%%%%%%%%%%
%


\small
\begin{thebibliography}{99}
%
\bibitem{Ai}
Aizenman, M.,
Localization at Weak Disorder: 
Some Elementary Bounds, 
Rev. Math. Phys. {\bf 6}(1994), 1163-1182. 
%
\bibitem{BP}
Burton, R., Pemantle, R. : 
Local characteristics, entropy and limit theorems for spanning trees and domino tilings via transfer-impedances, 
Annals of Probability, {\bf 21}(1993), p.1329-1371. 
%
\bibitem{HL}
Heilmann, O. J., Lieb, E. H. : 
Theory of Monomer-Dimer Systems, 
Commun. Math. Phys. {\bf 25}(1972), p.190-232. 
%
\bibitem{K}
Kenyon, R. : 
The Planar Dimer Model with Boundary : A Survey, 
Directions in Mathematical Quasicrystals, 
M. Baake, R. Moody, ed. 
CRM Monograph Series, 
{\bf 13}(2000), p.307-328. 
%
\bibitem{KPW}
Kenyon, R. W., Propp, J. G., and  Wilson, D. B.:
Trees and matchings.  Electron. J. Combin.  7  (2000), Research Paper 25, 34 pp. (electronic).
%
\bibitem{LRS}
Luby, M., Randall, D., and  Sinclair, A.: 
Markov chain algorithms for planar lattice structures.  SIAM J. Comput.  31  (2001),  no. 1, 167--192 (electronic).
%
\bibitem{NS}
Nakano, F., Sadahiro, T., :
Domino tilings with diagonal impurities, 
arXiv:0901.4824.
%
\bibitem{NS2}
Nakano, F., Sadahiro, T., :
Perfect matchings in non-bipartite chain-like graphs, 
in preparation. 
%
\bibitem{OS}
Ono, H., Sadahiro, T., : 
Connectedness of domino tilings with diagonal impurities, 
preprint. 
%
\bibitem{T}
Thurston, W. P., : 
Groups, Tilings and Finite State Automata, 
Summer 1989 AMS Colloquium Lectures. 
%
\end{thebibliography}
\end{document}